\documentclass[12pt]{article}

\usepackage{color}
\usepackage{dsfont}
\usepackage{enumerate}
\usepackage{graphicx,float}
\usepackage{caption}
\usepackage{subcaption}
\usepackage{setspace}
\usepackage{hyperref}
\usepackage{xcolor}
\usepackage{diagbox}
\usepackage{amsmath, amsfonts, amssymb, amsthm, amscd,graphicx}
\usepackage{mathtools} % pour tous les ams[...] - American Mathematical Society

\newtheorem{Theorem}{Theorem}[part]
\newtheorem{Definition}{Definition}[part]
\newtheorem{Proposition}{Proposition}[part]
\newtheorem{Assumption}{Assumption}[part]
\newtheorem{Lemma}{Lemma}[part]

\newtheorem{Corollary}{Corollary}[part]
\newtheorem{Remark}{Remark}[part]

\def \I{\mathbb{I}}
\def \N{\mathbb{N}}
\def \R{\mathbb{R}}

\def \E{\mathbb{E}}
\def \F{\mathbb{F}}
\def \G{\mathbb{G}}

\def \P{\mathbb{P}}

\def \D{\mathbb{D}}
\def \S{\mathbb{S}}

\def \1{\mathds{1}}

\def \Ac{\mathcal{A}}
\def \Bc{\mathcal{B}}

\def \Dc{\mathcal{D}}

\def \Fc{\mathcal{F}}
\def \Gc{\mathcal{G}}
\def \Hc{\mathcal{H}}
\def \Ic{\mathcal{I}}

\def \Lc{\mathcal{L}}
\def \Mc{\mathcal{M}}
\def \Nc{\mathcal{N}}

\def \Pc{\mathcal{P}}

\def \Rc{\mathcal{R}}

\def \Tc{\mathcal{T}}
\def \Uc{\mathcal{U}}
\def \Vc{\mathcal{V}}

\def \Yc{\mathcal{Y}}

\def \d{\textrm{d}}

\def \ni{\noindent}
\def \eps{\varepsilon}

\def\beqs{\begin{eqnarray*}}
\def\enqs{\end{eqnarray*}}
\def\beq{\begin{eqnarray}}
\def\enq{\end{eqnarray}}

\providecommand{\keywords}[1]
{
  \small
  \textbf{Keywords---} #1
}

 \title{A Stochastic Target Problem for Branching Diffusions}

 \author{ Idris Kharroubi\footnote{Research of the author partially supported by ANR grant RELISCOP.}\\ LPSM, UMR CNRS 8001,\\ Sorbonne Universit\'e
  and Universit\'e Paris Cit\'e,
  \\ idris.kharroubi @ sorbonne-universite.fr
  \and Antonio Ocello\\  LPSM, UMR CNRS 8001,\\ Sorbonne Universit\'e
  and Universit\'e Paris Cit\'e,
  \\ antonio.ocello @ sorbonne-universite.fr
  }

             \date{\today}

\begin{document}

\maketitle

\begin{abstract}
We consider an optimal stochastic target problem for branching diffusion processes. This problem consists in finding the minimal condition for which a control allows the underlying branching process to reach a target set at a finite terminal time for each of  its branches. This problem is motivated by an example from fintech where we look for the super-replication price of options on blockchain based cryptocurrencies. We first state a dynamic programming principle for the value function of the stochastic target problem. We then show that the value function can be reduced to a new function with a finite dimensional argument by a so called branching property. Under wide conditions, this last function is shown to be the unique viscosity solution to an HJB variational inequality.
\end{abstract}

\noindent \textbf{MSC Classification- }35K10, 49L20, 49L25, 	60J80,  91G20

\noindent \keywords{Stochastic target control, fintech, cryptocurrencies options, branching diffusion process, dynamic programming principle, Hamilton-Jacobi-Bellman equation, viscosity solution.}

%\listoftodos
%\tableofcontents

\section{Introduction}

The theory of optimal stochastic control has been extensively developed since the pioneering works in the 1950 decade. One reason for the growing attraction of this theory is  the variety of its applications, such as physics, biology, economics or finance.

In the last field, stochastic control theory appears to be a very natural tool as it provides solutions to the optimal portfolio choice issue. The need to control risks related to financial investments leads to  new stochastic optimization problems. Here, one looks for the minimal initial endowment needed to find a financial strategy whose final position satisfies some given constraints. Such optimization problems are called optimal stochastic target problem and have been widely studied (see e.g. \cite{SonerTouzi, SonerTouziSIAM, bouchard02, BouchardElieTouzi, bouchard2018quenched}).

The classical stochastic control theory has also been developed for other kind of stochastic processes such as branching diffusions. Those processes describe the evolution of a population of individuals with similar features concerning their dynamics and their reproduction. Branching processes have been first studied by Skorohod \cite{Sk64} and Ikeda et al. \cite{INW681,INW682,INW69}, who provided Feynmann-Kac presentation of solution to parabolic semi-linear partial differential equations (PDE for short). Since those pioneering works, branching processes have been extensively studied
in particular their scaling limits and the link with superprocesses (see \cite{Dawson}).
Recently, they were also used by Henry-Labord\`ere et al. \cite{HLTT} for Monte Carlo based numerical approximation of solutions to semilinear parabolic PDEs.

In the case where the branching processes are controlled, \"Ust\"unel \cite{ustunel} considers a finite horizon optimization problem. He restricted to Markov controls acting only on the drift coefficient. Following a martingale problem approach, he proved existence of optimal controls under wide conditions.  Nisio \cite{Nisio} considers the case where both the drift and diffusion coefficients are controlled.  She characterizes the related value function as a viscosity solution to a nonlinear parabolic PDE of HJB type. Then, Claisse \cite{claisse18} extends the previous results by allowing controls that may not preserve independance of the particles and  considering the lifespan and the progeny coefficients to depend on the position and the control.
Following the approach of Fleming and Soner \cite{FS} which relies on a result due to Krylov \cite{Krylov87}, the value function is approximated by a sequence of smooth value functions corresponding to small perturbations of the initial problem. This is what allows to prove a dynamic programming principle (DPP for short) and to derive a related dynamic programming equation.

In this paper, we investigate a stochastic target problem where the underlying controlled process is a branching diffusion. The problem consists in finding a minimal initial condition for a given target branching diffusion such that it dominates a function of another controlled branching diffusion for each  particle {alive at a given terminal time}.

We then give an extended equivalent formulation of the problem. Indeed, as the starting condition of the target branching process may contain several points, the previous problem is not well posed. We therefore look for the minimal value dominating all starting points such that the related branching process satisfies the terminal constraint.

%As the starting point of the target branching process can be seen as a measure and may contain several atoms, the previous problem is not well posed.
%We therefore give an extended equivalent formulation of the problem which consists in finding the minimal value for which one can find an initial measure condition whose atoms are dominated by this value and such that the related branching processes satisfy the terminal constraint.

%Such a problem finds an application in mathematical finance, when dealing with the optimal investment on crypto-currencies. For these assets, branching may appear due to their structure, leading to new assets (see e.g. \cite{Saleh2021}). In this framework, the super-replication issue remains unsolved and our framework provides a possible solution. A detailed example is given in this article.
  Such a problem finds an application in mathematical finance, when dealing with the optimal investment on crypto-currencies. For these assets, branching may appear due to their structure, leading to new assets (see e.g. \cite{Saleh2021}). In this framework, the super-replication issue remains to the best of our knowledge unsolved. Our setting provides a possible solution and we give a detailed example as an illustration.

We adopt a DPP approach  to  characterize the value function of our branching stochastic target problem.
Contrary to \cite{claisse18}, our argument do not rely on the existence of regular solution to approximated PDEs but on probabilistic results. %\textcolor{red}{Instead, we prove the dynamic programming principle by probabilistic arguments}.
We use a measurable selection theorem similar to that of \cite{SonerTouzi}. Combining it with a conditioning property for the law of the controlled process, we get the DPP.

We use it to identify the value function as a solution to a dynamic programming PDE.
We first show, as in \cite{claisse18}, a branching property. It
%on the value function. This  property
relates the value function at a given starting condition to the optimal values {at} its points. This allows to see the value function as a sequence of functions from $[0,T]\times\R^d$ to $\R$ indexed by the (countable) set of particle labels.
Contrary to the classical branching property, ours writes the value function as a maximum instead of a product. Hence, it entails irregularity bringing us out of the range of regular solutions.

We therefore adopt the framework of viscosity solutions. %\textcolor{red}
{ The dependence in the label variable leads to adapt the definition of viscosity solutions and to impose a continuous bound in the label.}
%
%The specific form of the label variable leads us to introduce a definition of viscosity solutions where we impose test functions to be uniformly bounded by a continuous function not depending on $i\in\Ic$. }
Using the DPP, the value function is %\textcolor{red}
{shown to be} a viscosity solution to a partial differential inequality of two terms. The first one is the classical nonlinear second order operator for classical diffusion processes, written as a supremum of a linear operator over controls that kill the diffusive part.  The restriction to these controls is due to the terminal constraint, imposed with probability one (see \cite{BouchardElieTouzi}). The second term expresses a monotonicity with respect to the label. More precisely, the value function taken at some label must be greater than its value on any other offspring label.
Surprisingly, our PDE do not contain any polynomial of the value function function as we  classically have in PDEs related to branching processes. This is due to the specific structure of the  considered control problem. We complete this parabolic PDE property by a terminal condition.

To get a {full} characterization of our value function, we finally consider the uniqueness to the PDE.   Under additional assumptions, we prove a comparison theorem using the classical approach of doubling variable combined with Ishii's lemma. This shows that the value function is the unique viscosity solution to the PDE. As a byproduct we get the continuity of the value function on the parabolic interior of the domain.

The remainder of the paper is organized as follows. In Section \ref{sec2}, we present the branching stochastic target problem and provide an example of application inspired from fintech. In Section \ref{sec3}, we set the dynamic programming principle. We finally show in Section \ref{sec4} viscosity properties of the value function and provide a uniqueness result to the related PDE. Finally we relegate some technical results needed in the proof of the conditionning property to the appendix.

\section{The problem}\label{sec2}
\subsection{Branching diffusions}\label{subsecdef}
We start by a description of the underling controlled processes. As those processes are of branching type, we first introduce the label set.

\paragraph{Label set}
For $n\geq 1$, a multi-integer $i=(i_1,\ldots,i_n)\in\N^n$ is simply denoted by $i=i_1\ldots i_n$.  For $n,m\geq 1$ and two multi-integers $i=i_1\ldots i_n\in \N^n$ and $j=j_1\ldots j_m\in \N^m$, we define their concatenation $i j\in\N^{n+m}$ by
\beq\label{concatenation}
i j & = & i_1\ldots i_n j_1\ldots j_m\;.
\enq
To describe the evolution of the particle population, we  introduce the set of labels $\Ic$ defined by
\beqs
\Ic & = & \{\varnothing\}\cup\bigcup_{n=1}^{+\infty}\N^n\;.
\enqs
The label $\varnothing$ corresponds to the mother particle. We extend the concatenation \eqref{concatenation} to the whole set $\Ic$ by
\beqs
\varnothing i & = & i \varnothing~~=~~i
\enqs
for all $i\in\Ic$. When the  particle labelled $i=i_1\ldots i_n\in \N^n$ gives birth to $k$ particles, the off-springs are labelled $i0,\ldots,i(k-1)$.
We also define the partial ordering relation $\preceq$ (resp. $\prec$) by
\beqs
j\preceq i & \Leftrightarrow & \exists \ell\in\Ic~:~i=j\ell\\
(\textrm{resp.}~j\prec i & \Leftrightarrow & \exists \ell\in\Ic\setminus \{\varnothing\}~:~i=j\ell)
\enqs
for all $i,j\in\Ic$.
We introduce the distance $d_\Ic$ on $\Ic$ defined by
\beqs
d^\Ic(i,j) = \sum_{\ell = p+1}^n (i_\ell +1) + \sum_{\ell' = p+1}^m (j_{\ell'} +1)\;,
\enqs
for $i=i_1\cdots i_{n}\in\N^n, j=j_1\cdots j_{m} \in \N^m$, with
\beqs
p & = & \max\{\ell\geq1~:~i_\ell=j_\ell\}\;.
\enqs
We next write
$|i|:= d^\Ic(i,\emptyset)$ for $i\in\Ic$.

\paragraph{Set of finite measures} In the sequel we shall consider finite measure on $\Ic\times \R^\ell$ for $\ell\geq 1$. For that, we endow the set $\Ic \times \R^{\ell}$  with the metric $d$ defined by
\beqs
d\left((i,x),(j,y)\right) & = & d_\Ic(i,j)+|x-y|\;,\qquad i,j\in \Ic\;, x,y\in \R^{\ell}\;.
\enqs
$\Ic \times \R^{\ell}$ is then separable and complete.
We denote by $\mathcal{M}_F(\Ic \times \R^\ell)$ set of the set of finite measures on  $\Ic \times \mathbb{R}^\ell$.  From Lemma 4.5 \cite{book:KALLENBERG-RM},  $\mathcal{M}_F(\Ic \times \R^{\ell})$ endowed with the topology of the weak convergence is Polish.
We recall that we say that a sequence $(\nu_n)_{n\geq0}$ weakly converges to $\nu$ in $\mathcal{M}_F(\Ic \times \R^{\ell})$ if $\int fd\nu_n\rightarrow \int fd\nu$ as $n\rightarrow+\infty$ for any continuous and bounded function $f$ from $\Ic \times \R^\ell$ to $\R$.
A possible metric associated to the weak topology on $\Mc_F(\Ic \times \R^\ell)$ is the Prokhorov metric (see Lemma 4.3 in \cite{book:KALLENBERG-RM}). We next define the subset $E_\ell$ of $\Mc_F(\Ic\times\R^\ell)$ by
\beq\label{defEell}
E_\ell & = & \left\{ \sum_{i\in V}\delta_{(i,x)}\;;~V\subseteq \Ic\;,~V\mbox{ finite}\;,~x^i\in\R^\ell \mbox{ and }i\nprec j \mbox{ for }i,j\in V \right\}\;.
\enq
By Proposition \ref{PropEclosed}, $E_d $ is Polish as well.

\paragraph{Probabilistic setting}
We fix a deterministic terminal time $T>0$ and a filtered probability space $(\Omega,\Fc,\bar \F=(\Fc_t)_{t\in[0,T]},\P)$ satisfying the usual conditions. Suppose that this probability space is endowed with a family of processes $(B^i,Q^i)_{i\in\Ic}$ such that
\begin{itemize}
\item $(B^i_t)_{t\in[0,T]}$ is an $ \F$-standard Brownian motion in $\R^m$ for all $i\in\Ic$;
\item $Q^i(\d t,\d k)$ is an $ \F$-Poisson random measure on $[0,T]\times\N$ with intensity measure $\d t\, \gamma p_k\delta_k$ for all $i\in\Ic$, with $\gamma>0$, $p_k\geq0$ for $k\geq0$ and $\sum_{k\geq 0}p_k=1$, $\delta_k$ being the Dirac measure at $k$;
\item $\{B^i, Q^j\;, i,j\in\Ic\}$ forms a family of mutually independent processes.
\end{itemize}
Having in mind these processes, we precise a better probability space.
\begin{itemize}
    \item Let $\Omega^0$ be the space of continuous functions from $[0,T]$ that are $\R^m$-valued starting at $0$. Let ${\mathbb{F}}^0:= ({\mathcal{F}}_t^0)_{t\in[0,T]}$ the filtration generated by the canonical process $B(\omega^0):=\omega^0,\ \omega^0\in\Omega^0$. We endow $(\Omega^0,{\mathcal{F}}^0_T)$ with the Wiener measure $\mathbb{P}^0$.

    \item Let $\Omega^1$ be the set of measures $\omega^1$ on $\R_+ \times \N$ of the form $\omega^1=\sum_{k\geq0}\delta_{(t_k,n_k)}$.
    Let ${\mathbb{F}}^1 := ({\mathcal{F}}_t^1)_{t\in[0,T]}$ be the filtration generated by the canonical process $Q(\omega^1)=\omega^1$:
    \beqs
    {\mathcal{F}}_t^1 : =\sigma\left(Q([0,s]\times \{k\})\ : s \in [0,t], \ k \in \N\right)\;,\quad t\in[0,T]\;.
    \enqs
We endow $(\Omega^1,{\mathcal{F}}^1_T)$ with the Poisson measure $\mathbb{P}^1$ of intensity $\d t\, \gamma\sum_{k\geq0}p_k\delta_k$, that is the probability measure such that $Q$ is a Poisson point process with intensity $\d t\, \gamma\sum_{k\geq0}p_k\delta_k$.
\end{itemize}

Following the structure we expect for $\{B^i, Q^j\;, i,j\in\Ic\}$, we define the filtered space $(\Omega, \mathcal{F}, {\mathbb{F}},\mathbb{P})$, where $\Omega = (\Omega^0\times\Omega^1)^{\Ic}$, $\mathbb{P} = (\mathbb{P}^0\otimes\mathbb{P}^1)^{\otimes\Ic}$, $\mathcal{F}$ is the $\P$-augmentation of $({\mathcal{F}}_T^0\otimes{\mathcal{F}}_T^1)^{\otimes\Ic}$ and ${\mathbb{F}} = ({\Fc}_t)_{t\in [0,T]}$ is the $\mathbb{P}$-augmentation of the filtration $(({\Fc}_t^0\otimes{\Fc}_t^1)^{\otimes\Ic})_{t\in [0,T]}$. On this space we extend the definition of the processes $B^i$ and $Q^i$ for $i\in \Ic$ as the previously described processes $B$ and $Q$ composed with the projections on each component, $i.e.$
\beqs
B^i(\omega):=\omega^{0,i}, \ Q^j(\omega):=\omega^{1,j},\quad \omega = (\omega^{0,i}, \omega^{1,i})_{i \in \Ic} \in \Omega\ .
\enqs
We  also define the process $\xi$ valued in $\mathcal{M}_F(\Ic \times \N \times \R^{m+1})$ by
\beq\label{defXi}
\xi_t & = & \sum_{i\in\Ic,n\in\N}\frac{1}{2^{2(|i|+n)}}\delta_{(i,n,B^i_t,Q^i([0,t]\times\{n\}))}
\enq
for $t\in[0,T]$.  We then notice that the filtration $ \F$ is the completed filtration generated by the process $\xi$.

To stress the dependence in time, we will use the following notations. For $t\in[0,T]$ and $\omega=(\omega^0,\omega^1)\in\Omega$, we define the stopped path at time $t$ by $\omega_{.\wedge t}=(\omega^0_{.\wedge t},\omega^1_{.\wedge t})$ where
\beqs
\omega^0_{.\wedge t} ~=~ (\omega^0_{s\wedge t})_{s\geq0} & \mbox{ and } & \omega^1_{.\wedge t} ~=~ \omega^1(\cdot\cap [0,t] \times\N)\;.
\enqs
For  a process $(X_t)_{t\in[0,T]}$ and a random time $\tau:~\Omega\rightarrow[0,T]$, we denote by $(X_{t\wedge\tau})_{t\in[0,T]}$ the process defined by
\beqs
X_{t\wedge\tau}(\omega) & = & X_t(\omega_{.\wedge \tau(\omega)})\;,\quad t\in[0,T],\; \omega\in\Omega\;.
\enqs
For $ \omega,\; \tilde \omega\in\Omega$ and a random time $\tau:~\Omega\rightarrow[0,T]$, we define the concatenation path $\omega\oplus_\tau\tilde \omega=(\omega^{0,i}\oplus_\tau\tilde\omega^{0,i},\omega^{1,i}\oplus_\tau\tilde \omega^{1,i})_{i\in\Ic}$ by
\beqs
(\omega^{0,i}\oplus_\tau\tilde\omega^{0,i})_s  & = &  \omega^{0,i}_{s}\1_{s< \tau(\omega)} +(\tilde\omega^{0,i}_s-\tilde\omega^{0,i}_{\tau(\omega)}+\omega^{0,i}_{\tau(\omega)})\1_{s\geq \tau(\omega)}\;,\quad s\in[0,T]\;,
\enqs
and
\beqs
 \omega^{1,i}\oplus_{\tau(\omega)}\tilde \omega^{1,i} & = & \omega^{1,i}(\cdot\cap [0,\tau(\omega)] \times\N)+\tilde \omega^{1,i}(\cdot\cap (\tau(\omega),T] \times\N)\;.
\enqs
for $i\in\Ic$. For a random variable  $S$ valued in some Polish space, we also define the shifted random variable $S^{\tau,\omega}$ by
\beq\label{def v.a. cond}
S^{\tau,\omega}(\tilde \omega) & = & S(\omega\oplus_{\tau}\tilde \omega)\;,\quad \tilde \omega\in \Omega\;.
\enq

\paragraph{Alive particles}
We define the set $\Vc_t$ of alive particles at time $t$ as follows.
\begin{itemize}
\item At time $t=0$, the set is reduced to the mother particle : $\Vc_0=\{\varnothing\}$.
\item For a time $t\geq0$, a particle $i\in \Vc_t$ dies at the first time $\tau_i> t$ the related Poisson measure $Q^i$ jumps after $t$, $i.e.$
\beqs
\tau_i & = & \inf\{s> t~:~Q^i((t,s]\times\N)=1\}\;.
\enqs
\item At time $\tau_i$, this particle gives birth to $k$ particles $i0,\ldots,i(k-1)$, with $k$ such that $Q^i(\{\tau_i\}\times \{k\})=1$:
\beqs
\Vc_{\tau_i} & = & \left(\Vc_{\tau_i-}\setminus\{i\}\right)\cup\{i0,\ldots,i(k-1)\} \;.
\enqs

\end{itemize}
%\textcolor{blue}{
%\begin{Remark}
%In such a probability space the different sources of randomness are clearly indentified. Applying Doob's representation theorem, being in a Polish space one can always find, via this independence structure, a copy of this probability space immersed in a larger one. This minimality with respect to inclusion will allow us to have useful measurability properties.
%\end{Remark}
%}

\paragraph{Controlled population} Take $A$ a Polish space with metric $d_A$. We assume $d_A$ to be bounded (if not so, we replace $d_A$ by $d_A\wedge 1$ and still have a Polish space). We define a control $\alpha$ as a family $(\alpha^i)_{i\in\Ic}$ of $\F$-progressively measurable processes valued in $A$. We denote by $\Ac$ the set of such controls.

Let $\lambda:~\R^d\times A\rightarrow\R^d$ and $\sigma:~\R^d\times A\rightarrow\R^{d\times m}$ be measurable functions. For a given control $\alpha\in \Ac$, each particle $i\in\Ic$ of the controlled population is born, evolves and dies to give birth to off-springs according to the set $\Vc$ defined above. We denote by $X^i_s$ the position  at time $s$ of a particle $i\in \Vc_s$.
For $ i\in \Ic$ alive at time $t$, let $\tau_i\geq t$ be the random time of its death, giving birth to $k$ off-springs. The position at a time $s\geq \tau_i$ of the off-springs $i0,\ldots,i(k-1)$ are given by
\beq
X^{i\ell}_{\tau_i} & = & X^{i}_{\tau_i}\label{def-dynX1}\\
\d X ^{i\ell}_s & = & \lambda(X_s^{i\ell},\alpha^{i\ell}_s)\d s+\sigma(X_s^{i\ell},\alpha^{i\ell}_s)\d B^{i\ell}_s\;,\label{def-dynX2}
\enq
for $\ell=0,\ldots,k-1$, such that $i\ell$ is alive at time $s$. We  represent the population of alive particles by the following measure valued process
\beqs
    Z_s & = & \sum_{i\in \Vc_s}\delta_{(i,X^i_s)}\;,\quad s\geq 0\;.
\enqs
The process $Z$ takes values in the Polish space $E_d$ defined by \eqref{defEell}.

\vspace{2mm}

For a function $f:\Ic\times \R^d\to \R$, and a measure $\mu=\sum_{i\in V}\delta_{(i,x_i)}\in E_d$, we set
\beqs
f(\mu) & = & \int_{\Ic\times \R^d}fd\mu~~=~~\sum_{i\in V}f_i(x_i)\;.
\enqs
We introduce the second order local operators $L^a$, $a\in A$ defined by
\beqs
L^a\varphi(x) & = & \lambda(x,a)^\top D\varphi(x) +\frac{1}{2} \textrm{Tr}\left(\sigma\sigma^\top(x,a)D^2\varphi(x)\right)\;,\quad x\in \R^d,
\enqs
for $\varphi\in C^2(\R^d)$, where $D\varphi$ and $D^2\varphi$ denote respectively, the gradient and the Hessian matrix of $\varphi$.

For a control $\alpha\in \Ac$ and a function $f: ~[0,T]\times\Ic\times \R^d\rightarrow\R$ such that $f_i(\cdot)\in C^{1,2}([0,T]\times\R^d)$ for all $i\in \Ic$, the following SDE characterises the behaviour of $Z$:
\beq\nonumber
f(t,Z_t) & = & f(s,Z_s)+\int_s^t\sum_{i\in V_u}D f_i(u,X^i_u)^\top\sigma(X^i_u,\alpha^i_u)\d B^i_u\\\label{dynZ}
 & & +\int_s^t\sum_{i\in V_u}(\partial_t+L^{\alpha^i_u}) f_i(u,X^i_u)\d u\\\nonumber
  & & + \int_{(s,t]\times\N}\sum_{i\in V_{u-}}\sum_{k\geq0}\left( \sum_{\ell=0}^{k-1} f_{i\ell}-f_i\right)(u,X^i_u)
  Q^i(\d u\d k)
\enq
for all $s,t\in[0,T]$ such that $s\leq t$.
\paragraph{Target branching diffusion} To each alive particle $i\in \Vc_s$, we associate a target position at time $s$ denoted by $Y^i_s$. Let $\lambda_{Y}:\R^d \times \R \times A \to \R$ and $\sigma_{Y}: \R^d \times A \to \R^{1 \times m}$ be measurable functions.  Let $\tau_i \geq t$ be the random time of death of $i\in\Ic$, the target position at time $s\geq \tau_i$ is given by
\beq
Y^{i\ell}_{\tau_i} & = & Y^{i}_{\tau_i}\label{def-dynY1}\\
dY ^{i\ell}_s & = & \lambda_{Y}(X^{i\ell,\alpha}_s, Y^{i\ell,\alpha}_s,\alpha_s^{i\ell})\, \d s + \sigma_{Y}(X^{i\ell,\alpha}_s,\alpha_s^{i\ell})\,
\d B^{i\ell}_s \: , \label{def-dynY2}
\enq
for $\ell=0,\ldots,k-1$, such that particle $i\ell$ is alive at time $s$.

We use the notation $\hat{\cdot}$ to define the quantities associated to the pair $\left(\begin{smallmatrix}
  X^i_s\\
  Y^i_s
\end{smallmatrix}\right)$, considering the previous problem but on $\R^{d+1}$. Therefore, we have $\hat{X}_s^i := \left(\begin{smallmatrix}
  X^i_s\\
  Y^i_s
\end{smallmatrix}\right)$, $\hat{\lambda}(\hat{X}_s^i,\alpha^i_s) := \left(\begin{smallmatrix}
  \lambda(X_s^i,\alpha^i_s)\\
  \lambda_{Y}(X^{i}_s, Y^{i}_s,\alpha_s)
\end{smallmatrix}\right)$ and $\hat{\sigma}(\hat{X}_s^i,\alpha^i_s) := \left(\begin{smallmatrix}
  \sigma(X_s^i,\alpha^i_s)\\
  \sigma_{Y}(X^{i}_s, \alpha_s)
\end{smallmatrix}\right)$. Under those hypotheses, assuming $i$ is alive, its position $\hat{X}^i$ evolves according to
\beq
d\hat{X}_s^i
& = & \hat{\lambda}(\hat{X}_s^i,\alpha^i_s) \d s + \hat{\sigma}(\hat{X}_s^i,\alpha^i_s)\,
\d B^i_s \: .
\enq
The resulting population process valued in $E_{d+1}$ is
\beqs
\hat Z_t & = & \sum_{i\in \Vc_s}\delta_{(i,X^i_s,Y^i_s)}\;,\quad s\geq 0\;.
\enqs

As before, we define the related second order local operators $\hat L^a$, $a\in A$ by
\beqs
\hat L^a\hat \varphi( \hat x) & = & \hat \lambda(\hat x,a)^\top D \hat \varphi(\hat x) +\frac{1}{2} \textrm{Tr}(\hat \sigma\hat \sigma^\top(\hat x,a)D^2\hat \varphi(\hat x))\;,\quad \hat x\in \R^{d+1},
\enqs
for $\hat \varphi\in C^2(\R^{d+1})$, where $D\hat \varphi$ and $D^2\hat \varphi$ denote respectively, the gradient and the Hessian matrix of $\hat \varphi$.

For a control $\alpha\in \Ac$ and a function $\hat f: ~[0,T]\times\Ic\times \R^{d+1}\rightarrow\R$ such that $\hat f_i(\cdot)\in C^{1,2}([0,T]\times\R^{d+1})$ for all $i\in \Ic$, the SDE related to  $\hat Z$ takes the following form:
\beq\nonumber
\hat f(t,\hat Z_t) & = & \hat f(s,\hat Z_s)+\int_s^t\sum_{i\in V_u}D\hat  f_i(u,\hat X^i_u)^\top\hat \sigma(\hat X^i_u,\alpha^i_u)\d B^i_u\\\label{dynhatZ}
 & & +\int_s^t\sum_{i\in V_u}(\partial_t+\hat L^{\alpha^i_u})\hat  f_i(u,\hat X^i_u)\d u\\\nonumber
  & & + \int_{(s,t]\times\N}\sum_{i\in V_{u-}}\sum_{k\geq0}\left( \sum_{\ell=0}^{k-1}\hat  f_{i\ell}-\hat f_i\right)(u,\hat X^i_u)
  Q^i(\d u\d k)
\enq
for all $s,t\in[0,T]$ such that $s\leq t$.

\paragraph{Well posedness}  To ensure the well definition of the presented  controlled processes, we make the following assumption.

\begin{Assumption}\label{Assumption_H_0}
\begin{enumerate}[(i)]
\item The coefficients $p_k$, $k\geq 0$, satisfy
\beqs
\sum_{k\geq0}kp_k~=~M  & < & +\infty\;.
\enqs
\item The functions  $\lambda$, $\sigma$, $\lambda_{Y}$ and $\sigma_{Y}$ satisfy
\beqs
\sup_{a\in A}|\lambda(0,a)|+|\sigma(0,a)| +|\lambda_{Y}(0,0,a)|+|\sigma_{Y}(0,a)| < +\infty~.
\enqs
\item There exists a constant $L > 0$ such
that
\beqs
|\lambda(x,a)-\lambda(x',a)|+|\sigma(x,a)-\sigma(x',a)| & & \\
+|\lambda_{Y}(x,y,a)-\lambda_{Y}(x',y',a')|+|\sigma_{Y}(x,a)-\sigma_{Y}(x',a)| & \leq & L\left(|x-x'|+|y-y'|\right)
\enqs
for all $x,x'\in\R^d$, $y,y'\in\R$ and $a\in A$.
\item There exists a nondecreasing  function $w:~\R_+\rightarrow\R_+$ such that  $w(r)\xrightarrow[r\rightarrow0]{} 0$ and
\beqs
|\lambda(x,a)-\lambda(x,a')|+|\sigma(x,a)-\sigma(x,a')| & & \\
+|\lambda_{Y}(x,y,a)-\lambda_{Y}(x,a')|+|\sigma_{Y}(x,a)-\sigma_{Y}(x,a')| & \leq & w(d_A(a,a'))
\enqs
for all $x\in\R^d$, $y\in\R$ and $a,a'\in A$.

\end{enumerate}
\end{Assumption}

\vspace{2mm}
For any initial condition $t\in[0,T]$, $\mu=\sum_{i\in V}\delta_{(i,x_i)}\in E_d$ and $y_i\in\R$ for $i\in V$, we extend the controlled branching processes $(X,Y)$. For that the set of alive particles $\Vc^{t,\mu}$ is defined as follows.
\begin{itemize}
\item For $s\in[0,t]$, $\Vc^{t,\mu}_s=V$.
\item For $s\geq t$, a particle $i\in \Vc_s$ dies at the first time $\tau_i> s$ the related Poisson measure $Q^i$ jumps after $s$:
\beqs
\tau_i & = & \inf\{r> s~:~Q^i(]s,r]\times\N)=1\}\;.
\enqs
\item At time $\tau_i$, the particle $i$ gives birth to $k$ particles $i0,\ldots,i(k-1)$, with $k$ such that $Q^i(\{\tau_i\}\times \{k\})=1$:
\beqs
\Vc^{t,\mu}_{\tau_i} & = & \left(\Vc^{t,\mu}_{\tau_i-}\setminus\{i\}\right)\cup\{i0,\ldots,i(k-1)\} \;.
\enqs
\end{itemize}
Then, the controlled branching population process $X^{t,\mu,\alpha}=(X^{t,\mu,\alpha,i}_s, i\in \Vc_s^{t,\mu})_{s\in[0,T]}$ is defined by the initial condition
\beqs
X^{t,\mu,\alpha}_s & = & (x_i,\;i\in V)\;,\quad s\in[0,t]\;,
\enqs
together with dynamics \eqref{def-dynX1}-\eqref{def-dynX2}. We also denote by $\hat \mu\in E_{d+1}$ the extended measure as
\beqs
\hat \mu & = & \sum_{i\in V}\delta_{(i,x_i,y_i)}\;,
\enqs
and $Y^{t,\hat \mu,\alpha}=(Y^{t,\hat \mu,\alpha,i}_s, i\in \Vc_s^{t,\mu})_{s\in[0,T]}$ the controlled branching target process with initial condition
\beqs
Y^{t,\hat \mu,\alpha,i}_s & = & y_i\;,\quad s\in[0,t]\;,
\enqs
for all $i\in V$, together with dynamics \eqref{def-dynY1}-\eqref{def-dynY2}.
Let $Z^{t,\mu,\alpha}$  and $\hat Z^{t,\hat \mu,\alpha}$ be
\beqs
Z^{t,\mu,\alpha}_s = \sum_{i\in \Vc^{t,\mu}_s}\delta_{(i,X^{t,\mu,\alpha,i}_s)} & \mbox{ and }
 & \hat Z^{t,\hat \mu,\alpha}_s = \sum_{i\in \Vc^{t,\mu}_s}\delta_{(i,X^{t,\mu,\alpha,i}_s,Y^{t,\hat \mu,\alpha,i}_s)}
\enqs
for $s\in[0,T]$.

In this setting, we have the following non-explosion result.

\begin{Proposition}\label{well-pose-branch-pop}
Suppose that Assumptions \ref{Assumption_H_0} (i)-(ii)-(iii) hold. Fix $t\in[0,T]$, $\mu=\sum_{i\in V}\delta_{(i,x_i)}\in E_d$, $\hat \mu=\sum_{i\in V}\delta_{(i,x_i,y_i)}\in E_{d+1}$ and $\alpha\in \Ac$.

\vspace{2mm}

\noindent (i) The set of alive particles $\Vc^{t,\mu}_s$ is uniquely defined and is finite for all $s\in[0,T]$. More precisely, we have
\beqs
\E\left[\sup_{s\in[0,T]}|\Vc_s^{t,\mu}|\right] & \leq & |V|e^{\gamma M(T-t)}
\enqs
where $|\Vc|$ stands for the cardinal of a subset $\Vc$ of $\Ic$.

\noindent (ii)  There exists a unique $\F$-adapted process $( Z^{t,\mu,\alpha})$ (resp. $(\hat Z^{t,\hat \mu,\alpha})$) valued in $E_{d} $ (resp. $E_{d+1} $). Moreover, the process   $Z^{t,\mu,\alpha}$ (resp. $(\hat Z^{t,\hat \mu,\alpha})$) satisfies \eqref{dynZ} (resp. \eqref{dynhatZ}).
\end{Proposition}

We refer to \cite[Proposition 2.1]{claisse18} for the proof of this proposition.

\begin{Remark}\label{rem-extention}For any $i\in \Ic$ the processes $X^{t,\mu,\alpha,i}$ and $Y^{t,\hat\mu,\alpha,i}$ are defined on times $s\in[t,T]$ such that $i\in \Vc^{t,\mu}_s$.  However, we can extend their definition  to the whole interval $[t,T]$. Suppose first that $i$ has no ancestor in $\Vc_t^{t,\mu}$:
\beqs
j~\npreceq~i & \mbox{for all } & j\in \Vc_t^{t,\mu}\;.
\enqs
Then we define  processes $X^{t,\mu,\alpha,i}$ and $Y^{t,\hat\mu,\alpha,i}$ as the unique solutions to
\beqs
\d X ^{t,\mu,\alpha,i}_s & = & \lambda(X_s^{t,\mu,\alpha,i},\alpha^{i}_s)\d s+\sigma(X_s^{t,\mu,\alpha,i},\alpha^{i}_s)\d B^i_s\\
dY ^{t,\mu,\alpha,i}_s & = & \lambda_{Y}(X^{t,\mu,\alpha,i}_s, Y^{t,\hat\mu,\alpha,i}_s,\alpha_s^{i})\, \d s + \sigma_{Y}(X^{t,\mu,\alpha,i}_s,\alpha_s^{i})\,
\d B^{i}_s
\enqs
for $s\in[t,T]$, with initial condition $X^{t,\mu,\alpha,i}_t=0$ and $Y^{t,\mu,\alpha,i}_t=0$. On the complementary case, it exists $j\in \Vc_t^{t,\mu}$ such that $j\preceq i$. Then there exists $k\geq 1$ and $\ell_1,\ldots,\ell_k$ such that
 \beqs
 i & = & j\ell_1\ldots\ell_k\;.
 \enqs
We denote the associated branching times  by $(S_0,\dots, S_k)$:
\beqs
S_{m} & = & \inf\left\{s>S_{m-1}~:~Q^{j\ell_1\ldots\ell_{m}}\left((S_{m-1},s]\times\{n_{m}\}\right)=1\right\}
\enqs
where $n_m\geq \ell_{m+1}+1$ for $m=0,\ldots,k$ with $S_{-1}=t$. Then we define  the extended processes $X^{t,\mu,\alpha,i}$ and $Y^{t,\mu,\alpha,i}$ by
\beqs
X ^{t,\mu,\alpha,i}_s & = & \mathds{1}_{[t,S_{0})}(s)X ^{{t,\mu,\alpha,j}}_s+\sum_{m=1}^{k-1}\mathds{1}_{[S_{m-1},S_{m})}(s)X ^{t,\mu,\alpha,j\ell_1\ldots\ell_m}_s+\mathds{1}_{[S_{k-1},+\infty)}(s)X ^{{t,\mu,\alpha,i}}_s \\
Y ^{{t,\hat \mu,\alpha,i}}_s & = & \mathds{1}_{[t,S_{0})}(s)Y ^{{t,\hat \mu,\alpha,j}}_s+\sum_{m=1}^{k-1}\mathds{1}_{[S_{m-1},S_{m})}(s)Y ^{t,\hat \mu,\alpha,j\ell_1\ldots\ell_m}_s+\mathds{1}_{[S_{k-1},+\infty)}(s)Y ^{{t,\hat \mu,\alpha,i}}_s
\enqs
for $s\in[t,T]$.

These extended processes can be seen as solution to a Brownian stochastic differential equation with Lipschitz coefficients. Obvious in the first case, to show it in the second one, we consider the \emph{ancestor Brownian motion} $\bar{B}^i$ defined by

\beqs
\bar{B}^i_s &  = & B^j_s\mathds{1}_{[t,S_0)}+ \sum_{m=1}^{k-1}\mathds{1}_{[S_{m-1},S_m)}(s)\left(B ^{j\ell_1\ldots\ell_m}_s-B ^{j\ell_1\ldots\ell_m}_{S_{m-1}}+B ^{j\ell_1\ldots\ell_{m-1}}_{S_{m-1}}\right)\\ & & +\mathds{1}_{[S_{k-1},+\infty)}(s)\left(B ^{i}_s-B ^{i}_{S_{k-1}}+B ^{j\ell_1\ldots\ell_{k-1}}_{S_{k-1}}\right)\;,
\enqs
for $s\in[t,T]$. This process is continuous, centered, with independent increments and variance equal to $t$, therefore a Brownian motion by Lévy's characterisation. Then the extended processes $X^{t,\mu,\alpha,i}$ and $Y^{t,\mu,\alpha,i}$ are the unique solutions to the SDE
\beq\label{extEDSX}
\d X ^{t,\mu,\alpha,i}_s & = & \bar \lambda(s,X_s^{t,\mu,\alpha,i})\d s+ \bar\sigma(s,X_s^{t,\mu,\alpha,i})\d \bar B^i_s\\ \label{extEDSY}
dY ^{t,\hat\mu,\alpha,i}_s & = & \bar \lambda_{Y}(s,X^{t,\mu,\alpha,i}_s, Y^{t,\hat\mu,\alpha,i}_s)\, \d s + \bar \sigma_{Y}(s,X^{t,\mu,\alpha,i}_s)\,
\d \bar B^{i}_s
\enq
 for $s\in[t,T]$, with initial condition $X^{t,\mu,\alpha,i}_t=x_i$ and $Y^{t,\mu,\alpha,i}_t=y_i$. The coefficients being given by
 \beqs
  \bar \lambda(s,x) & = & \mathds{1}_{[t,S_0)}\lambda(x,\alpha^j_s)+\sum_{m=1}^{k-1}\mathds{1}_{[S_{m-1},S_m)}(s)\lambda(x,\alpha^{j\ell_1\ldots\ell_m}_s)+\mathds{1}_{[S_{k-1},+\infty)}(s)\lambda(x,\alpha^i_s)
  \\
  \bar \sigma(s,x) & = & \mathds{1}_{[t,S_0)}\sigma(x,\alpha^j_s)+\sum_{m=1}^{k-1}\mathds{1}_{[S_{m-1},S_m)}(s)\sigma(x,\alpha^{j\ell_1\ldots\ell_m}_s)+\mathds{1}_{[S_{k-1},+\infty)}(s)\sigma(x,\alpha^i_s)
  \\
  \bar \lambda_{Y}(s,x,y) & = & \mathds{1}_{[t,S_0)}\lambda_{Y}(x,y,\alpha^j_s)+\sum_{m=1}^{k-1}\mathds{1}_{[S_{m-1},S_m)}(s)\lambda_{Y}(x,y,\alpha^{j\ell_1\ldots\ell_m}_s)+\mathds{1}_{[S_{k-1},+\infty)}(s)\lambda_{Y}(x,y,\alpha^i_s)
  \\
  \bar \sigma_{Y}(s,x) & = & \mathds{1}_{[t,S_0)}\sigma_{Y}(x,\alpha^j_s)+\sum_{m=1}^{k-1}\mathds{1}_{[S_{m-1},S_m)}(s)\sigma_{Y}(x,\alpha^{j\ell_1\ldots\ell_m}_s)+\mathds{1}_{[S_{k-1},+\infty)}(s)\sigma_{Y}(x,\alpha^i_s)
 \enqs
for $(s,x,y)\in[0,T]\times\R^d\times\R$. Under Assumption \ref{Assumption_H_0}, those coefficients satisfy classical Lipschitz and boundedness assumption to have uniqueness and stability of solutions. In the sequel, we shall refer by $X^{t,\mu,\alpha,i}$ and $Y^{t,\hat \mu,\alpha,i}$ either to the processes themselves or to their extended definitions if the processes are considered outside their living interval.
\end{Remark}
Under the additional regularity assumption on the coefficients with respect to the control, we have a stability result for the branching system.
 \begin{Proposition}\label{prop-stabBD}
Suppose that Assumptions \ref{Assumption_H_0} holds and fix $t\in[0,T]$, $\mu=\sum_{i\in V}\delta_{(i,x_i)}\in E_d$, $\hat \mu=\sum_{i\in V}\delta_{(i,x_i,y_i)}\in E_{d+1}$, $\alpha\in \Ac$. Let $(t_n)_{n\geq1}$,   $\left(\hat \mu_n=\sum_{i\in V_n}\delta_{(i,x_i^n,y_i^n)}\right)_{n\geq1}$ and $(\alpha^n)_{n\geq1}$ be sequences of $\R_+$, $E_{d+1} $  and $\Ac$ such that
\beqs
(t_n,\hat \mu_n) & \xrightarrow[n\rightarrow+\infty]{} & (t,\hat \mu)\;,
\enqs
and
\beqs
\E\int_0^T d_A\left(\alpha^i_s,\alpha^{n,i}_s\right)\d s \xrightarrow[n\rightarrow+\infty]{} & 0
\enqs
for all $i\in\Ic$. Then,
\beqs
\E\left[\left(|X^{t_n,\mu_n,\alpha_n,i}_s\mathds{1}_{i\in \Vc_s^{t_n,\mu_n}}-X^{t,\mu,\alpha,i}_s\mathds{1}_{i\in \Vc_s^{t,\mu}}|^2+|Y^{t_n,\hat \mu_n,\alpha_n,i}_s\mathds{1}_{i\in \Vc_s^{t_n,\mu_n}}-Y^{t,\hat \mu,\alpha,i}_s\mathds{1}_{i\in \Vc_s^{t,\mu}}|^2\right)\right]
& \xrightarrow[n\rightarrow+\infty]{}  & 0
\enqs
for all $s\in[t,T]$,
where $\mu_n=\sum_{i\in V_n}\delta_{(i,x_i^n)}\in E_d$ for any $n\geq 1$.
\end{Proposition}
\begin{proof}
We proceed in three steps.

\vspace{2mm}

\ni \textit{Step 1.} We first prove that
\beqs
\mathds{1}_{\Vc_s^{t_n,\mu_n}}(i) & \xrightarrow[n\rightarrow+\infty]{\P-a.s.} & \mathds{1}_{\Vc_s^{t,\mu}}(i)
\enqs
for all $i\in\Ic$. For that, we distinguish two cases.

\ni \textit{Case 1.} Suppose that $\mathds{1}_{\Vc_s^{t,\mu}}(i)=1$. Then, there exist $j\in\Vc_t^{t,\mu}$ and $\ell_1,\ldots, \ell_k$ such that $i=j\ell_1\ldots\ell_k$ and
\beqs
t~<~S_1~<\cdots<~S_{k-1}\leq s < S_k
\enqs
where $S_1,\ldots,S_k$ are the successive branching times:
\beq\label{cond-branch-times}
S_{m} & = & \inf\left\{r>S_{m-1}~:~Q^{j\ell_1\ldots\ell_{m}}\left((S_{m-1},r]\times\{n_{m}\}\right)=1\right\}
\enq
with $n_m\geq \ell_{m+1}+1$ for $m=1,\ldots,k$. Since $\hat \mu_n\rightarrow\hat \mu$ and $j\in \Vc_t^{t,\mu}$, there exists $N\geq 1$ such that
\beq\label{cond-ini-label}
j & \in &  \Vc_{t_n}^{t_n,\mu_n} \mbox{ for all }n\geq N.
\enq
We then get from \eqref{cond-branch-times} and \eqref{cond-ini-label} that
\beqs
i & \in &  \Vc_{s}^{t_n,\mu_n} \;.
\enqs
for $n$ large enough.

\ni \textit{Case 2.} Suppose that $\mathds{1}_{\Vc_s^{t,\mu}}(i)=0$. We then have two subcases.

\ni \textit{Subcase 2.1.} There exist $j\in\Vc_t^{t,\mu}$ and $\ell_1,\ldots, \ell_k$ such that $i=j\ell_1\ldots\ell_k$. We then have
\beq\label{subscase11}
 s~ >~ S_k & \mbox{ or } & s~<~S_{k-1}
\enq
where $S_1,\ldots,S_k$ are defined by \eqref{cond-branch-times}.
Since $\hat\mu_n\rightarrow \hat\mu$, we have $i\in  \Vc_{t_n}^{t_n,\mu_n}$ for $n$ large enough and we get from \eqref{subscase11} that $\mathds{1}_{\Vc_s^{t_n,\mu_n}}(i)=0$ large enough.

\ni \textit{Subcase 2.1.} $j\notin\Vc_t^{t,\mu}$ for any $j\preceq i$. Since the set of ancestor of $i$ is finite and  $\hat\mu_n\rightarrow \hat \mu$, we get $j\notin\Vc_{t_n}^{t_n,\mu_n}$ for any $j\preceq i$ for $n$ large enough. Therefore, we have $\mathds{1}_{\Vc_s^{t_n,\mu_n}}(i)=0$ for $n$  large enough.

\vspace{2mm}

\ni \textit{Step 2.} We prove that
\beqs
\E\left[\left(|X^{t_n,\mu_n,\alpha_n,i}_s-X^{t,\mu,\alpha,i}_s|^2+|Y^{t_n,\hat\mu_n,\alpha_n,i}_s-Y^{t,\hat \mu,\alpha,i}_s|^2\right)\right]  & \xrightarrow[n\rightarrow+\infty]{} & 0
\enqs
for $s\in[0,T]$ and $i\in \Ic$. Since $\hat\mu_n\rightarrow\hat\mu$ as $n\rightarrow+\infty$, we have
$X^{t_n,\mu_n,\alpha_n,i}_{t_n}\rightarrow X^{t,\mu,\alpha,i}_{t}$ and $Y^{t_n,\hat \mu_n,\alpha_n,i}_{t_n}\rightarrow Y^{t,\hat\mu,\alpha,i}_{t}$ as $n\rightarrow+\infty$. Using Assumption \ref{Assumption_H_0} (iii), we can apply Theorem 8.1 in \cite{Krylov} and we get the result.

\vspace{2mm}

\ni \textit{Step 3.} We then write
\beqs
\E\left[\left(|X^{t_n,\mu_n,\alpha_n,i}_s\mathds{1}_{i\in \Vc_s^{t_n,\mu_n}}-X^{t,\mu,\alpha,i}_s\mathds{1}_{i\in \Vc_s^{t,\mu}}|^2+|Y^{t_n,\hat\mu_n,\alpha_n,i}_s\mathds{1}_{i\in \Vc_s^{t_n,\mu_n}}-Y^{t,\hat\mu,\alpha,i}_s\mathds{1}_{i\in \Vc_s^{t,\mu}}|^2\right)\right] & \leq & \\
2\E\left[\left(|X^{t_n,\mu_n,\alpha_n,i}_s-X^{t,\mu,\alpha,i}_s|^2+|Y^{t_n,\hat\mu_n,\alpha_n,i}_s-Y^{t,\hat\mu,\alpha,i}_s|^2\right)\right] &  & \\
+2\E\left[\left(|X^{t,\mu,\alpha,i}_s|^2+|Y^{t,\hat\mu,\alpha,i}_s|^2\right)\left(\mathds{1}_{i\in \Vc_s^{t,\mu}}-\mathds{1}_{i\in \Vc_s^{t_n,\mu_n}}\right)^2\right]\;. &  &
\enqs
Using the dominated convergence theorem we get from Step 1
\beqs
\E\left[\left(|X^{t,\mu,\alpha,i}_s|^2+|Y^{t,\hat\mu,\alpha,i}_s|^2\right)\left(\mathds{1}_{i\in \Vc_s^{t,\mu}}-\mathds{1}_{i\in \Vc_s^{t_n,\mu_n}}\right)^2\right] & \xrightarrow[n\rightarrow+\infty]{} & 0 \;.
\enqs
This last convergence  and Step 2 give the result.
\end{proof}

Focusing on conditional laws of the controlled processes, we have a representation result. Define $\D([0,T],\Mc_F(I\times \R^{m+1}))$
as the set of \textit{c\`adl\`ag} functions from $[0,T]$ to $\Mc_F(I\times \R^{m+1})$.
We endow this set with the Skrorkhod metric related to the Prokhorov distance and the related Borel $\sigma$-algebra. From Doob's functional representation Theorem (see e.g. Lemma 1.13 in \cite{book:KALLENBERG-FMP}) for any control $\alpha$, there exists a $\Bc([0,T])\otimes \Bc(\D([0,T],\Mc_F(I\times \R^{m+1}))$-measurable function $\tilde \alpha:~[0,T]\times \D([0,T],\Mc_F(I\times \R^{m+1}))\rightarrow A^\Ic$ such that $\alpha^i(\omega)  =  \tilde \alpha ^i(s,\xi(\omega_{.\wedge s}))=\tilde \alpha ^i(s,\xi(\omega))$ for any $s\in[0,T]$, $\omega\in \Omega$ and $i\in\Ic$ .  In the sequel, we identify the control $\alpha$ with its related function $\tilde \alpha$ and we still denote by $\Ac$ the set of those controls.

For $\alpha\in\Ac$, an $\F$-stopping time $\tau$ and $\omega\in\Omega$, we define the control $\alpha^{\tau(\omega),\omega}$ by
\beqs
\left(\alpha^{\tau(\omega),\omega}\right)^i(s,\xi(\tilde \omega)) & = & \alpha^i\left(s,\xi^{\tau(\omega),\omega}(\tilde\omega)\right)
\enqs
for $i\in\Ic$, $s\geq0$ and $\tilde\omega\in\Omega$, where $\xi^{\omega,\tau(\omega)}$ is given by \eqref{def v.a. cond}.
\begin{Theorem}[Conditioning property] \label{condTHM}Suppose  that Assumption \ref{Assumption_H_0} holds and fix $t\in[0,T]$,  $\hat\mu=\sum_{i\in V}\delta_{(i,x_i,y_i)}\in E_{d+1}$ and $\alpha\in \Ac$. Then, for any bounded measurable function $f:~\D([0,T],E_{d+1})\rightarrow \R$ and any $\F$-stopping time $\tau$, we have
\beqs
\E\left[ f\left(\hat X^{t,\hat \mu,\alpha}\right) \Big|\Fc_\tau\right](\omega) & = & F\left( \tau(\omega),\hat X^{t,\hat \mu,\alpha}_{.\wedge \tau}(\omega),\alpha^{\tau(\omega),\omega} \right)\;,\qquad \P(d\omega)-a.s.
\enqs
where
\beqs
F( s,\hat x,\beta ) & = & \E\left[ f\left( (\hat x_t\mathds{1}_{t< s}+ \hat X^{s,\hat x_s,\beta}_t \mathds{1}_{t\geq s} )_{t\in[0,T]} \right) \right]
\enqs
for all $s\in[0,T]$, $\hat x\in\D([0,T], E_{d+1})$ and $\beta\in \Ac$.
\end{Theorem}
The proof of this result is postponed to Appendix \ref{proof-condtioning}. It follows the same lines as the proof of Theorem 2 in \cite{CTT13}, and relies on a uniqueness property for the related branching martingale controlled problem which is studied in Appendix \ref{sec-mart-cont-pb}.

\subsection{The stochastic target problem}

To define the stochastic target problem, let $g:~\Ic\times\R^d\rightarrow\R$ be a function satisfying the following assumption.

\begin{Assumption} \label{Assumption_H_1} The function $g_i$ is continuous on $\R^d$ for all $i\in \Ic$.
\end{Assumption}

Fix an initial time $t\in[0,T]$ and an initial population $\mu=\sum_{i\in V}\delta_{(i,x_i)}$. We look for an  initial position $y$ for  the target process and a control $\alpha\in \Ac$ such that
\beqs
Y^{t,\hat \mu,\alpha,i}_t & = & y\;,\quad i\in V,
\enqs
and $Y^{t,\hat \mu,\alpha}$ and $X^{t, \mu,\alpha}$ satisfies the terminal constraints
\beqs
Y^{t,\hat\mu,\alpha,i}_T & \geq  & g_i(X^{t,\mu,\alpha,i}_T)\;,\quad  i \in \Vc^{t,\mu}_T\;.
\enqs
More precisely, we look for the \textit{reachability set}
\beqs
\Rc(t,\mu) & = & \Big\{y \in \R,~ \ :~\exists \alpha\in\Ac ~:~ Y^{t,\hat\mu,\alpha,i}_T ~ \geq ~  g_i(X^{t,\mu,\alpha,i}_T)\;,\quad  i \in \Vc^{t,\mu}_T\\
 & &\qquad\qquad\qquad\qquad\qquad\qquad\qquad \qquad\mbox{ with } \hat \mu =\sum_{i\in V}\delta_{(i,x_i,y)}\Big\}\ .
\enqs
for $t\in [0,T]$ and  $\mu=\sum_{i\in V}\delta_{(i,x_i)}\in E_d $.
Since the target processes $Y^i$ has an explicit impact only on its drift $\lambda_Y$ and not on its diffusion coefficient $\sigma_Y$,  the reachability set satisfies the following \textit{monotonicity property}.
\begin{Proposition}\label{monotonicity_property} Suppose that Assumptions \ref{Assumption_H_0} holds.
 For $\mu=\sum_{i\in V}\delta_{(i,x_i)}\in E_d $ and $y \in \Rc(t,\mu)$ we have $[y,\infty[ \subseteq \Rc(t,\mu)$.
\end{Proposition}

\begin{proof}
Fix a control $\alpha=(\alpha^i)_{i \in \Ic}$ and a starting point $(t,\mu)$. We take $y\in\Rc(t,\mu)$, $y'\geq y$ and write $\hat \mu$ (resp. $\hat \mu'$) for $\sum_{i\in V}\delta_{(i,x_i,y)}$ (resp. $\sum_{i\in V}\delta_{(i,x_i,y')}$), $Y^{i}$ (resp. $Y'^{i}$) for $Y^{t,\hat\mu,\alpha,i}$ (resp. $Y^{t,\hat\mu',\alpha,i}$) and $\delta{Y}^i$ for ${Y}'^{i}-{Y}^{i}$.
We then have
\beqs
\delta {Y}^i_s
 & = &  (y' -y) +
\int_t^s \chi_u\delta {Y}^i_u \d u
\enqs
for $s\geq t$, where $\chi$ is given by
\beqs
\chi_u :=\frac{
\bar\lambda_{Y}\left(u,{X}^i_u,{Y}^i_u\right) - \bar \lambda_{Y}\left(u,{X}^i_u,Y'^i_u\right)
}{\delta {Y}_u} \ , \quad u\geq 0\;,
\enqs
with $\bar \lambda_{Y}$ defined in Remark \ref{rem-extention}.
From the Lipschitz property of $\lambda_{Y}$ in Assumption \ref{Assumption_H_0},  $\chi$ is bounded and
\beqs
\delta \bar{Y}^i_T  &  = & (y'-y)\exp\left(\int_t^T \chi_u \d u\right) ~\geq~ 0\;,\quad  \P-a.s.
\enqs
Since  $y \in \Yc(t,\mu)$, we get
\beqs
Y^{t,\mu,\alpha,y', i}_T& \geq & Y^{t,\mu,\alpha,y, i}_T ~\geq~ g_i\left(X_T^{t,\mu,\alpha, i}\right)\;,\quad \P- a.s.
\enqs
This is true for all $i \in \Vc^{t,\mu}_T$, therefore $y' \in \Rc(t,\mu)$.
\end{proof}

From Proposition \ref{monotonicity_property},  the closure $\overline{\Rc(t,\mu)}$ of the reachability set is a half line interval characterized by its lower bound. We then define the value function  $v$ as the infimum of $\Rc$:
\beq
v(t,\mu) &  := &  \inf \Rc(t,\mu)\nonumber\\
 & = & \inf\Big\{y\in\R~:~\exists \alpha\in \Ac \;,~
Y^{t,\hat\mu,\alpha,i}_t ~=~ y ~ \forall i \in V, \nonumber \\
 & & \qquad \qquad \qquad \qquad \mbox{ and }
Y^{t,\hat\mu,\alpha,i}_T  ~\geq~  g_i\left(X^{t,\mu,\alpha,i}_T\right) ~\forall i \in \Vc^{t,\mu}_T\text{ a.s.}\Big\}\label{def-value-function}
\enq
for all $t\in[0,T]$ and $\mu=\sum_{i\in V}\delta_{(i,x_i)}\in E_d$, with the usual convention that $\inf(\emptyset)=+\infty$.
Our aim is to provide an analytical characterisation of the value function $v$.
\begin{Remark}
The value function $v$ or the reachability set $\Rc$ might not be well defined in the case where an extinction of the alive population of particle  happens before $T$. In this case we take the convention that the terminal condition is always satisfied if $\Vc_T^{t,\mu}=\emptyset$. In the sequel, we keep this convention for other constraints on $(X^{t,\mu,\alpha,i}_\theta,Y^{t,\hat\mu,\alpha,i}_\theta)$ with $\Vc_\theta^{t,\mu}=\emptyset$ and $\theta$ a stopping time.
\end{Remark}

We next provide a new formulation of the function $v$.

\begin{Proposition} Under Assumptions \ref{Assumption_H_0}, the value function function $v$ satisfies the following identity
\beq\label{ext-def-value-function}
v(t,\mu) &  = &  \inf\Big\{y\in\R~:~\exists \alpha\in \Ac \;,~\exists \hat \mu=\sum_{i\in V}\delta_{(i,x_i,y_i)}\in E_{d+1}\;\mbox{ such that }\nonumber \\
 & & \qquad \quad y_i ~\leq~ y ~ \forall i \in V,
 \mbox{ and }
Y^{t,\hat\mu,\alpha,i}_T  ~\geq~  g_i(X^{t,\mu,\alpha,i}_T) ~\forall i \in \Vc^{t,\mu}_T\text{ a.s.}\Big\}
\enq
for all $t\in[0,T]$ and $\mu=\sum_{i\in V}\delta_{(i,x_i)}\in E_d$.
\begin{proof} Denote by $\tilde v(t,\mu)$ the right hand side of \eqref{ext-def-value-function}.
Since the set whose  infimum is  $v(t,\mu)$ is included in the one whose infimum is  $\tilde v(t,\mu)$, we obviously have
\beqs
\tilde v(t,\mu) & \leq &  v(t,\mu)\;.
\enqs
Fix now $y\in\R$ for which there exist $\alpha\in\Ac$ and $\hat \mu=\sum_{i\in V} \delta_{(i,x_i,y_i)}\in E_{d+1}$ such that
\beqs
y_i & \leq & y \;,\quad i\in V\;,
\enqs
and
\beqs
Y^{t,\hat\mu,\alpha,i}_T   & \geq &   g_i(X^{t,\mu,\alpha,i}_T) \;,\quad i\in \Vc^{t,\mu}_T
\enqs
Set $\bar\mu=\sum_{i\in V} \delta_{(i,x_i,y_i)}\in E_{d+1}$. By the comparison argument used in the proof of Proposition \ref{monotonicity_property}, we have
\beqs
Y^{t,\bar\mu,\alpha,i}_T~\geq ~Y^{t,\hat\mu,\alpha,i}_T   & \geq &   g_i(X^{t,\mu,\alpha,i}_T) \;,\quad i\in \Vc^{t,\mu}_T
\enqs
Therefore $y\geq v(t,\mu)$ and $\tilde v(t,\mu)\geq v(t,\mu)$.
\end{proof}
\end{Proposition}

\subsection{An example of application from fintech}\label{example}
Fintech is the contraction of the  words finance and technology. It refers to recent technologies that allows for the improvement and the automation of the delivery and use of financial services. The field has emerged at the beginning of the 21-st century and covered technologies used by established financial institutions. Since that time, the field has evolved to also include crypto-currencies which are decentralised financial assets. Those assets are based on the block-chain technology. The main idea of that structure is to keep any new transaction registered in a chain by adding new blocks and sharing the extension of the original chain over the network, so that every user keeps in mind the transaction and can certify it.
We refer to \cite{Nakamoto2008} for a description of how a block-chain base crypto-currency works in the case of the Bitcoin.

Due to the structure of this kind of assets, a fork can appear in the chain (see \cite{Saleh2021}). In this case, the original asset is transformed into several assets. A natural question that arises is how to evaluate  an option on crypto-currencies in this case. We present here the example of the super-replication of options on asset that may fork and show that it is a particular case of the branching stochastic target presented above.

We consider a financial market on which is defined a crypto-currency with price process $(S_t)_{t\in[0,T]}$. We suppose that the process $S$ is a branching diffusion and describe its dynamics.
We first define the set  $\Vc_t$ of alive particles at time $t\in[0,T]$ as previously done in Section \ref{subsecdef}. The initial condition for the process $S$ is  a constant $(S_0>0)$.
Assume the version $i\in\Ic$ of the crypto-currency is alive at time $t\in[0,T]$, dies at some random time $\tau_i\geq t$ and gives birth to $k$ new versions $i0,\ldots,i(k-1)$.  The position at a time $s\geq \tau_i$ of the new the crypto-currencies are given by
\beq
S^{i\ell}_{\tau_i} & = & S^{i}_{\tau_i}\label{def-dynS1}\\
\d S ^{i\ell}_s & = & S_s^{i\ell}\left(b\d s+c\d B^{i\ell}_s\right)\;,\label{def-dynS2}
\enq
for $\ell=0,\ldots,k-1$ and $s\geq \tau_i$ such that version $i\ell$ is alive at time $s$.
Here $b$ and $c$ are two positive constants.

In addition to that asset, we assume that there exists on the market a non-risky asset $S^0$ with deterministic interest rate $r>0$ and with initial condition $S^0_0=1$, that is $S_t=e^{rt}$ for $t\in[0,T]$.

An investment strategy consists in a process $\pi=(\pi^i_t)_{t\in[0,T],i\in\Ic}$ of $\F$-progressive processes valued in $[0,1]$, where $\pi^i_t$ represents the proportion of the wealth invested in the version $S^i$ of the crypto-currency.  We denote by $\Ac$ the set of such strategies.
For $\pi\in \Ac$, we also denote by $V^{V_0,\pi}$ the self financing wealth process related to the initial capital $V_0$ and strategy $\pi$. According to \eqref{def-dynS1}-\eqref{def-dynS2} it is given by
\beq
V^{V_0,\pi,i\ell}_{\tau_i} & = & V^{V_0,\pi,i}_{\tau_i}\label{def-dynV1}\\
\d V ^{V_0,\pi,i\ell}_s & = & V_s^{V_0,\pi,i\ell}\left(((b-r)\pi_s^{i\ell}+r)\d s+c\pi_s^{i\ell}\d B^{i\ell}_s\right)\;,\label{def-dynV2}
\enq
for $\ell=0,\ldots,k-1$ and $s\geq \tau_i$ such that version $i\ell$ is alive at time $s$.

We then consider a financial derivative on the asset $S$ that consists in a Put Option but with a strike  $K_i$ depending on the version of the crypto-currency $S$. Such a product can express the need to hedge againts a decrease of the value of the asset $S$ that depends on the branch.

The computation of the super-replication problem leads to solve the following stochastic target problem
\beqs
w_0 & = & \inf\left\{\nu\in\R_+~:~\exists \pi\in \Ac \;,~
V^{\nu,\pi,i}_T  ~\geq~  (K_i-S^{i}_T)_++\kappa ~\forall i \in \Vc^{}_T\text{ a.s.}\right\}\;,
\enqs
where $\kappa$ is a positive constant representing some friction.
We next modify this problem to satisfy our assumptions. For that, we first define the processes
\beqs
Y_t^{y,\pi,i} & = & \log \left( V^{e^y,\pi,i}_t \right)\\
X^{i}_t & = &  \log \left( S^{i}_t \right)
\enqs
for $t\in[0,T]$ and $i\in \Vc_t$. From \eqref{def-dynS1}-\eqref{def-dynS2} and \eqref{def-dynV1}-\eqref{def-dynV1}, we get
\beq\label{def-dynX}
X^{i\ell}_{\tau_i} ~ = ~ X^{i}_{\tau_i}\;,&  &
\d X ^{i\ell}_s ~ = ~ (b-\frac{c^2}{2})\d s+c\d B^{i\ell}_s\;,\\ \label{def-dynY}
Y^{y,\pi,i\ell}_{\tau_i} ~ = ~ Y^{y,\pi,i}_{\tau_i}\;, & &
\d Y ^{y,\pi,i\ell}_s ~ = ~ \left((b-r)\pi_s^{i\ell}-\frac{1}{2}c^2(\pi_s^{i\ell})^2+r\right)\d s+c\pi_s^{i\ell}\d B^{i\ell}_s\;,
\enq
for $\ell=0,\ldots,k-1$ and $s\geq \tau_i$ such that version $i\ell$ is alive at time $s$.
We observe that the dynamics of the processes $Y$ and $X$ satisfy Assumption \ref{Assumption_H_0}. We also define the functions $g$ as
\beqs
g_i\left( x \right) & = & \log\left(\left( K_i-e^x \right)_++\kappa \right)\;,\quad (x,i)\in \R\times\Ic\;,
\enqs
which satisfies Assumption \ref{Assumption_H_1}. Finally, we define the optimal value
\beqs
v_0 & = & \inf\left\{y\in\R~:~\exists \pi\in \Ac \;,~
Y^{y,\pi,i}_T  ~\geq~  g_i(X^{i}_T) ~\forall i \in \Vc^{}_T\text{ a.s.}\right\}\;,
\enqs
a special case of \eqref{def-value-function}. We notice that the optimal value $w_0$ is related to $v_0$ by
\beqs
w_0= \exp(v_0)\;.
\enqs
We suppose that $\bar K:=\sup_{i\in\Ic} K_i<+\infty$. The value function $v$ related to $v_0$ is then bounded. Indeed, by taking the initial condition $t\in[0,T]$ and $y=-r(T-t)+\log(\bar K+\kappa)$ and the control $\pi^{i}_t=0$ for $i\in\Ic$ and $t\in[0,T]$, we get from \eqref{def-dynY}
\beqs
Y^{t,\hat \mu,\pi,i}_T & \geq & g_i(X^{t,\mu,i}_T)\;,\quad i\in\Vc^{t,\mu}_T
\enqs
for $\mu=\sum_{i\in V}\delta_{(i,x^i)}\in E_d$ and $\mu=\sum_{i\in V}\delta_{(i,x^i,y)}\in E_{d+1}$. Therefore
\beqs
v(t,\mu) & \leq & -r(T-t)+\log(\bar K+\kappa)\;,\quad (t,\mu)\in [0,T]\times E_d\;.
\enqs
Moreover, for any $y\in \Rc(t,\mu)$ and $\pi$ the related admissible control, we have
\beqs
\left((b-r)\pi_s^{i\ell}-\frac{1}{2}c^2(\pi_s^{i\ell})^2+r\right) & \leq & \left(\frac{b-r}{c}\right)^2+r
\enqs
Therefore we get
\beqs
y+\left(\left(\frac{b-r}{c}\right)^2+r\right)(T-t) ~~ \geq ~~ \E\left[Y_T^{t,\hat \mu,\pi,i}\right] & \geq & \E\left[g_i(X_T^{t,\mu,\pi,i})\right]~~\geq~~\log(\kappa)\;.
\enqs
Therefore,
\beqs
v(t,\mu) & \geq & -\left(\left(\frac{b-r}{c}\right)^2+r\right)(T-t)+\log(\bar K+\kappa)\;,\quad (t,\mu)\in [0,T]\times E_d\;.
\enqs
In particular, $v$ satisfies the growth condition  \eqref{uniqueness:hyp:i_to_infty} of the comparison Theorem \ref{Thm:comparison}. If we suppose also that $r=0$ and $g_i=0$ for $i\in \Ic$ of the form $i=i_1\cdots i_n$ with $i_\ell\geq I$ for some $\ell$ where $I$ is a given bound,  then $v$ also satisfies condition \eqref{thm:growthinx} of Theorem \ref{Thm:comparison}.
\section{Dynamic programming}\label{sec3}

\subsection{Measurable selection}
In establishing a dynamic programming principle, we need an admissible control as concatenation of admissible controls depending on the position of the branching processes at an intermediary time. For this end, we use a \textit{measurable selection approach}.

Let $\mathcal{U}$ be the \textit{target set} defined by
\beqs
\Uc(t,\hat \mu) & = & \left\{\alpha \in \Ac\ :~
Y^{t,\hat\mu,\alpha,i}_T ~ \geq ~ g_i(X^{t,\mu,\alpha,i}_T)\ \forall i \in \Vc^{t,\mu,\alpha}_T a.s.\right\}\ , \enqs
for $(t,\hat \mu)\in [0,T]\times E_{d+1}$ with $\hat\mu=\sum_{i\in V}\delta_{(i,x_i,y_i)}$ and $
\mu=\sum_{i\in V}\delta_{(i,x_i)}\in E_d$.
Let $S:= [0,T]\times E_{d+1}$ and
\beqs
D & := &
\left\{(t,\hat \mu) \in S \ : \ \Uc(t,\hat \mu) \neq \emptyset \right\} \; .
\enqs
Our aim is to exhibit a function that associates to each $(t,\hat \mu)\in D$ a control $\alpha\in\Uc(t,\hat \mu)$ in a measurable way.

We denote by $\Pc(S)$ the set of probability measures on $(S,\Bc([0,T])\otimes\Bc(E_{d+1}))$ and we endow $\Ac$ with the Borel $\sigma$-algebra $\Bc(\Ac)$ related to the distance
\beqs{}
(\alpha,\alpha') & \mapsto & \sum_{i\in\Ic} \frac{1}{2^{|i|}}\wedge\E\int_0^T|\alpha^i_s-\alpha'^i_s|\d s
\enqs{}
where $|i|=i_1+\cdots+i_n$ for $i=(i_1,\ldots,i_n)\in\N ^n$ and $n\geq 1$. We then have the following measurable selection result.
\begin{Lemma} \label{measurable_selector} Suppose that Assumptions \ref{Assumption_H_0} and \ref{Assumption_H_1} hold.
For each $\nu \in \Pc(S)$, there exists a measurable function $\phi_{\nu}:(D,\Bc(D)) \to (\Ac, \Bc(\Ac))$ such that
\beqs
\phi_{\nu}(t,\hat \mu) \in \Uc(t,\hat\mu) \mbox{ for }\nu\mbox{-a.e. }(t,\hat \mu)\in D\ .
\enqs
\end{Lemma}
\begin{proof}
$S$ being endowed with the product $\sigma$-algebra $\Bc([0,T])\otimes\Bc(E_{d+1})$ is a Borel space as product of Borel spaces. Also $\Ac$ endowed with $\Bc(\Ac)$ is a Borel space. Let $C$ be the following set
\beqs
C & := & \left\{ (t,\hat \mu) \in S\times \Ac \ : \ \alpha \in \Uc(t,\hat \mu) \right\}\ .
\enqs
From Proposition \ref{prop-stabBD}  and Assumption \ref{Assumption_H_1}, $C$ is closed and a fortiori a Borel subset of $S\times \Ac$.

\vspace{2mm}

    \noindent $\bullet$ \textit{Step 1: Measurable selector.}

Since $C$ is a Borel set, it is analytic by \cite[Proposition 7.36]{book:Bertsekas-Shreve}. From the Jankov-von Neumann measurable selection theorem (see e.g.  \cite[Proposition 7.49]{book:Bertsekas-Shreve}), there exists an analytically measurable function $\phi:D \to \Ac$ such that
\beqs
\left\{ \left(t,\hat\mu,\phi(t,\hat \mu)\right)~:~(t,\hat \mu)\in S\right\}& \subset & C\;.
\enqs

    \noindent $\bullet$   \textit{Step 2: Construction of a Borel measurable $\phi_{\nu}$ such that $\phi_{\nu} = \phi$ $\nu$-almost everywhere.}

Fix $\nu\in\Pc(S)$ and denote by $\Bc_\nu(S)$ the completion of the Borel $\sigma$-algebra $\Bc(S)$ under $\nu$. From  \cite[Corollary 7.42.1]{book:Bertsekas-Shreve} any analytic set is universally measurable. Therefore $\phi$ is universally measurable, and, from the definition of the universal $\sigma$-algebra, $\phi$ is $\Bc_\nu(S)$-measurable. Since $\Bc_\nu(S)$ is the completion of $\Bc(S)$ under $\nu$, there exists a Borel measurable map $\phi_{\nu}$ such that $\phi_\nu(t,\hat\mu)=\phi(t,\hat\mu)$ for $\nu$-almost every $(t,\hat\mu)\in S$.
\end{proof}

\subsection{Dynamic programming principle}
For $t\in[0,T]$, we denote by $\Tc_{[t,T]}$ the set of $\F$ stopping times valued in $[t,T]$. The dynamic programming principle may be stated as follows.

\begin{Theorem}\label{theorem:DPP}
Under Assumptions \ref{Assumption_H_0} and \ref{Assumption_H_1}, the value function satisfies
\beq\label{eq:DPP1}
v(t,\mu) & = & \inf\bigg\{y\in\R~:~\exists \alpha\in \Ac \;,~\exists \hat \mu=\sum_{i\in V}\delta_{(i,x_i,y_i)}\in E_{d+1}\;\mbox{ such that }\nonumber \\
 & & \qquad \quad y_i ~\leq~ y ~ \forall i \in V,
 \mbox{ and }
Y^{t,\hat\mu,\alpha,i}_\theta  ~\geq~  v\left(\theta,\delta_{\left(i,X^{t,\mu,\alpha,i}_{\theta}\right)}\right) ~\forall i \in \Vc^{t,\mu}_\theta\text{ a.s.}\bigg\}
\enq
for any $(t,\mu)\in [0,T]\times E_d$ and $\theta \in \mathcal{T}_{[t,T]}$.

\end{Theorem}

\begin{proof} We first define the \textit{reachability sets}  by
\beqs
\Yc(t, \mu) & := & \bigg\{(y_i)_{i\in V} \in \R^V~ \ :~\Uc(t,\hat \mu) \neq \emptyset \mbox{ with } \hat \mu =\sum_{i\in V}\delta_{(i,x_i,y_i)}\bigg\}\; .
\enqs
and
\beqs
\Yc^{\theta}(t,\mu) & = & \bigg\{(y_i)_{i\in V} \in \R^V \ :\ \exists \alpha\in \Ac\; \mbox{ such that }\nonumber \\
 & & ~ \qquad  Y^{t,\hat \mu,y,\alpha, i}_{\theta} \geq v\left(\theta,\delta_{(i,X^{t,\mu,\alpha,i}_{\theta})}\right) ~ \forall i \in \Vc^{t,\mu}_{\theta}\text{ a.s.} \mbox{ with }\hat \mu=\sum_{i\in V}\delta_{(i,x_i,y_i)}\in E_{d+1}\bigg\}\ .
\enqs
for $t\in [0,T]$ and  $\mu=\sum_{i\in V}\delta_{(i,x_i)}\in E_d $ and $\theta\in \Tc_{[t,T]}$.
Fix now $t\in[0,T]$ and $\mu=\sum_{i\in V}\delta_{(i,x_i)}\in E_d$. Denote by $v_\theta(t,\mu)$ the right hand side of \eqref{eq:DPP1}.

To prove $v(t,\mu)  \geq v_{\theta}(t,\mu)$, we show that $\Yc(t,\mu) \subset \Yc^{\theta}(t,\mu)$. Let $(y_i)_{i\in V} \in \Yc(t,\mu)$. By definition there exists $\alpha \in \Ac$ such that
\beqs
Y^{t,\hat \mu,\alpha, j}_T \geq  g_j\left(X^{t,\mu,\alpha, j}_T\right)\quad \forall j \in \Vc^{t,\mu}_T \ .
\enqs
From the uniqueness of solutions to \eqref{def-dynX1}-\eqref{def-dynX2} and \eqref{def-dynY1}-\eqref{def-dynY2} (or equivalently \eqref{extEDSX}-\eqref{extEDSX}) we get the following flow property
\beqs
X^{t,{\mu},\alpha,j}_T  &  =  &  X^{\theta,\delta_{\left(i,{X}^{t,\mu,\alpha,i}_{\theta}\right)},\alpha,j}_T\;, \\
Y^{t,{\hat\mu},\alpha,j}_T  &  =  &  Y^{\theta,\delta_{\left(i,{X}^{t,\mu,\alpha,i}_{\theta},{Y}^{t,\hat \mu,y,\alpha,i}_{\theta}\right)},\alpha,j}_T\;,
\enqs
for all $i\in \Vc_\theta^{t,\mu}$ and $j\in \Vc_T^{t,\mu}$ such that $i\preceq j$.

We therefore get
\beqs
Y^{\theta,\delta_{\left(i,{X}^{t,\mu,\alpha,i}_{\theta},{Y}^{t,\hat \mu,\alpha,i}_{\theta}\right)},\alpha,j}_T & \geq & g_j\left(X^{\theta,\delta_{\left(i,{X}^{t,\mu,\alpha,i}_{\theta}\right)},\alpha,j}_T\right)\;,
\qquad j \in \Vc^{\theta,\delta_{\left(i,{X}^{t,\mu,\alpha,i}_{\theta}\right)}}_T \;,
\enqs
for all $i\in \Vc^{t,\mu}_\theta$.
Given the definition of the value function $v$, we get $Y^{t,\hat \mu,\alpha,i}_{\theta}  \geq  v\left(\theta,\delta_{(i,X^{t,\mu,\alpha,i}_{\theta})}\right)$ for all $i \in \Vc^{t,\mu}_\theta$ a.s. and $(y_i)_{i\in V} \in \Yc^{\theta}(t,\mu)$.

We now turn to the reverse inequality $v_\theta(t,\mu)\geq v (t,\mu)$. To this end, we prove that $\Yc^{\theta}_\eps(t,\mu) \subset \Yc(t,\mu)$ for any $\eps>0$, where
\beqs
\Yc^{\theta}_\eps(t,\mu) & = & \left\{ (y_i+\eps)_{i\in V}~:~(y_i)_{i\in V}\in \Yc^{\theta}(t,\mu) \right\}\;.
\enqs
 Let $(y_i)_{i\in V}  \in \Yc^{\theta}(t,\mu)$ and $\alpha \in \Ac$ such that $Y^{t,\hat\mu,\alpha,i}_{\theta}  \geq  v\left(\theta,\delta_{(i,X^{t,\mu,\alpha,i}_{\theta})}\right)$ for all $i \in \Vc^{t,\mu}_\theta$ a.s. where $\hat \mu=\sum_{i\in V}\delta_{(i,x_i,y_i)}$. Fix now $\varepsilon > 0$ and set $\hat\mu=\sum_{i\in V}\delta_{(i,x_i,y_i)}$ and ${\hat\mu}_\eps=\sum_{i\in V}\delta_{(i,x_i,y_i+\eps)}$.
From the definition of the value function and the strict monotonicity of the flow w.r.t. the initial value, we get $Y^{t,\hat\mu,\alpha,i}_{\theta}(\omega) < Y^{t,\hat\mu_\varepsilon,\alpha,i}_{\theta}(\omega) \in \Yc\left(\theta, \delta_{(i,X^{t,\mu,\alpha,i}_{\theta})}\right)(\omega)$ for all $i \in \Vc^{t,\mu}_\theta$ for $\P$-a.e. $\omega\in\Omega$.
Consider the probability measure  $\nu$ induced on $S$ by
\beqs
\omega \mapsto
\left(\theta, \hat{Z}^{t,\hat \mu_\eps,\alpha}_{\theta}\right)(\omega)\ ,
\enqs
and $\phi_{\nu}$ the measurable map defined in Lemma \ref{measurable_selector}. We have
\beq\label{DPPproof:negligeable_sets}
Y^{\tilde{t},\tilde{\mu},\tilde{y}, \phi_{\nu}(\tilde{t},\tilde{\mu},\tilde{y}),i}_T  & \geq &   g_i\left(X^{\tilde{t},\tilde{\mu},\phi_{\nu}(\tilde{t},\tilde{\mu},\tilde{y}),i}_T\right) \qquad \forall i \in \Vc_T ~ \P\mbox{-a.s. for }\nu\mbox{-a.e. }(\tilde{t},\tilde{\mu},\tilde{y})\in D\ .
\enq

%To better analysing the flow properties,
We define $\hat{\Xi} := \left(\theta, \hat{Z}^{t,\hat\mu_\eps,\alpha}_{\theta}\right)$ and $\Xi := \left(\theta, Z^{t,\mu,\alpha}_{\theta}\right)$. For $i_T \in \Vc^{t,\mu}_T$, if $i_\theta \in \Vc^{t,\mu}_\theta$ such that $i_\theta \preceq i_T$, the initial conditions at time $\theta$ for $Y_T^{i_T}$ and $X_T^{i_T}$ are respectively  $\Xi^{}$ and $\hat\Xi^{}$.

Binding flow properties with the measurable selector, we can find a negligible set $N_1$ of $\Fc$ such that there exists negligible set $N_{2,\omega_1}$ of $\Fc$ for each $\omega_1\in N_1^c$ such that
\beqs
Y^{\hat\Xi(\omega_1),\phi_\nu(\hat\Xi(\omega_1)),i}_T(\omega_2) & \geq & g_{i_T}\left(X^{\Xi(\omega_1),\phi_\nu(\hat\Xi(\omega_1)),i}_T (\omega_2)\right)\qquad \forall i \in \Vc^{\Xi(\omega_1)}_T(\omega_2)
\enqs
for all $\omega_1\in N_1^c$ and $\omega_2\in N_{2,\omega_1}^c$.

We now define the set $\bar{N}:=\{\omega\ :\ \omega \in N_1^c\ , \omega \in N_{2,\omega}\}$ and we prove that $\bar N$ is negligible.
We first have $\bar{N}\subset N_1^c\cap \bar{N}_2$ where
\beqs
\bar{N}_2 & = & \left\{\omega\ \in \Omega :\
\exists i \in \Vc_T^{\Xi(\omega)}(\omega)\;, \quad
Y^{\hat \Xi^{}(\omega),\phi_\nu(\hat \Xi^{}(\omega)),i}_T(\omega) < g_{i}\left(X^{\Xi(\omega),\phi_\nu(\hat \Xi^{(i)}(\omega)),i}_T (\omega)\right)
\right\}~.
\enqs
The set $\bar N_2$ can be rewritten as
\beqs{}
\bar{N}_2 &=& \left\{\omega\ \in \Omega :\
\prod_{ i \in \Vc_T^{\Xi}}\1_{
Y^{\hat \Xi^{},\phi_\nu(\hat\Xi^{}),i}_T \geq g_{i}\left(X^{\Xi,\phi_\nu(\hat \Xi^{}),i}_T \right)}(\omega) = 0 \right\}\;.
\enqs
Taking the conditional expectation w.r.t. $\Fc_\theta$, we have up to a negligible set
\beqs
\bar{N}_2&=& \left\{\omega\ \in \Omega :\
\mathbb{E}\left[\prod_{ i \in \Vc_T^{\Xi}}\1_{
Y^{\hat \Xi^{},\phi_\nu(\hat \Xi^{}),i}_T \geq g_{i}\left(X^{\Xi,\phi_\nu(\hat\Xi^{}),i}_T \right)}\Bigg| \Fc_\theta\right](\omega) = 0 \right\}~.
\enqs
Using Theorem \ref{condTHM} wet get
\beqs
\bar{N}_2 &=& \left\{\omega\ \in \Omega :\
\int_{\Omega} \prod_{ i \in \Vc_T^{\Xi(\omega)}(\omega\oplus_\theta\omega')}\1_{
Y^{\hat\Xi(\omega),\phi_\nu(\hat\Xi(\omega)),i}_T(\omega\oplus_\theta\omega') \geq g_i\left(X^{\Xi(\omega),\phi_\nu(\hat\Xi(\omega)),i}_T (\omega\oplus_\theta\omega')\right)
}d\P(\omega') = 0 \right\} \\
&=& \left\{\omega\ \in \Omega :\ \P(N_{2,\omega})= 0 \right\}~.
\enqs
Therefore we get, up to a negligible set, $\bar{N}_2 \subset N_1$ and $\P(\bar{N}_2)=0$

We now fix $\alpha\in A$  and define the control $\bar \alpha=(\bar\alpha^i)_{i\in\Ic}$ by
\beqs
\bar{\alpha}^i(\omega) := \begin{cases}
\alpha ^i(\omega) \1_{[0,\theta(\omega))} + \phi_{\nu}^i(\hat\Xi(\omega))(\omega)\1_{[\theta(\omega),T]} &\mbox{ if }\omega \in \Omega\setminus \bar{N}\\
a &\mbox{ if }\omega \in \bar{N}
\end{cases} \
\enqs
for all $i\in\Ic$ wirth $a\in A$.
Since $\left(Y^{\hat\Xi,\phi_\nu(\hat\Xi)}_T,X^{\Xi,\phi_\nu(\hat\Xi),i}_T \right)  =
\left(Y^{t,\hat\mu_\varepsilon, \bar{\alpha},i}_T, X^{t,\mu,\bar{\alpha},i}_T\right)$ for each $i \in \Vc^{t,\mu}_T$ a.s. and $\bar N_2$ is negligible, we get $(y_i+\varepsilon)_{i\in V} \in \Yc(t,\mu)$.
\end{proof}

\section{PDE characterisation}\label{sec4}
\subsection{Branching property}

Conditionally to their birth, the alive particles , and consequently their branches, are independent in the uncontrolled case. In out case, this \textit{branching property} is passed down to the value function in the following way.

\begin{Proposition}[Branching property] Let Assumption  \ref{Assumption_H_0} holds. The value function $v$ satisfies
\beq\label{branching-property}
v(t,\mu)  & = & \max_{i\in V}v(t,\delta_{(i,x^i)})
\enq
for any $(t,\mu = \sum_{i \in V} \delta_{(i,x^i)} ) \in [0,T]\times E_d$.
\end{Proposition}
\begin{proof}
For ${\mu}=\sum_{ i \in V}\delta_{(i,x_i)}\in E_d$, we define
\beqs
K^{{\mu}} := \left\{y\in\R~:~\exists \alpha\in \Ac, ~
Y^{t,\hat{\mu},\alpha,i}_T  \geq  g_i(X^{t,{\mu},\alpha,i}_T)\ \forall i \in \Vc^{t,{\mu},\alpha}_T a.s.\mbox{ with },~
\hat{\mu} = \sum_{ i \in V}\delta_{(i,x_i,y)}\right\}~ .
\enqs

Proving $v(t,\mu)  \geq \max_{i\in V} v(t,\delta_{(i,x^i)})$ comes to verify that $K^{\mu} \subseteq \bigcap_{j \in V} K^{\delta_{(j,x^j)}}$, i.e. $K^{\mu} \subseteq K^{\delta_{(j,x^j)}}$ for each $j \in V$. If $y \in K^{\mu}$, there exists $\alpha$ satisfying the constraints in $T$ a.s. With this same $\alpha$, zooming in on the sub-population generated by each $j \in V$, we must satisfy the condition of $K^{\delta_{(j,x^j)}}$. Therefore, $y \in K^{\delta_{(j,x^j)}}$.

 Let $j$ be the index that realises the maximum in the righthand side of \eqref{branching-property}. The monotonicity property given by Proposition \ref{monotonicity_property} implies $K^{\delta_{(j,x^j)}} \subseteq K^{\delta_{(i,x^i)}}$ for all $i \in V$. Then, if $y\in K^{\delta_{(j,x^j)}}$, let $\alpha^i$ be a control for $i\in V$ that meets the demand of $K^{\delta_{(i,x^i)}}$.
To prove $y\in K^{\mu}$ we must exhibit a control that satisfies the requirements of such a set. Having a control $\alpha$ taken as $\alpha^i$ on the branches generated by each $i \in V$, we meet the conditions of $K^\mu$. Therefore, $\max_{i\in V}v(t,\delta_{(i,x^i)}) = v(t,\delta_{(j,x^j)}) \leq v(t,\mu)$
\end{proof}
From this result, we can focus on the function $\bar{v}$ defined on $\Ic \times [0,T] \times \R^d$ by
\beqs
    \bar{v}_i(t,x)  &= & v(t,\delta_{(i,x)})
\enqs
 for $(i,t,x) \in \Ic \times [0,T] \times \R^d$.
We provide in the next sections a PDE characterisation of the function $\bar v$.

\subsection{Dynamic programming equation}\label{subsection:main_result}
\subsubsection{The equation on the parabolic interior}
In a stochastic target problem, wishing to hit a given target with probability one, we must degenerate along certain directions. Moreover, we also need to control the uncertainty related the possible branching. This property enables the characterisation of the value function $\bar v$ as a solution the following PDE
\beq\label{PDE1}
\min \left\{-\partial _t\bar{v}_i(t,x)+ F\left(x,\bar{v}_i(t,x),D  \bar{v}_i(t,x),D^2_x \bar{v}_i(t,x)\right)~;~ \bar v_{i}(t,x) - \sup_{0\leq k< \bar{K}}\bar v_{ik}(t,x)\right\}& = & 0\qquad
\enq
for $(t,x)\in[0,T)\times\R^d$,
where
\beqs
\bar{K} &=& \sup\left\{k+1\in\N~:~p_k>0\right\}~,
\\
F_{}(\Theta)&=& \sup\left\{ \lambda_{Y}(x,y,a)
-\lambda(x,a)^\top p - \frac{1}{2}\mathrm{Tr}\left(\sigma\sigma^\top(x,a)M\right)~:~
a\in\Nc(x,p) \right\}
\enqs
for $\Theta =(x,y,p,M) \in \R^d \times \R \times \R^d \times \S^d$, and
\beqs
\Nc(x,p) = \left\{ a\in A~:~ N^a(x,p)=0  \right\} ~ \text{ and }~ N^a(x,p) = \sigma_{Y}(x,a)-\sigma(x,a)^\top p
\enqs
for $x,p\in \R^d $.

Since the control set $A$ is not necessarily compact, the operator associated to this PDE may not be continuous. We therefore need to define a weak formulation of \eqref{PDE1}. For that, we  introduce the relaxed semilimits of $F$ given by
\beqs
F^*(\Theta) = \limsup_{\varepsilon\to 0,\Theta'\to\Theta} F_{\varepsilon}(\Theta')~ \text{ and } ~F_*(\Theta) = \liminf_{\varepsilon\to 0,\Theta'\to\Theta} F_{\varepsilon}(\Theta')
\enqs
where
\beqs
F_{\varepsilon}(\Theta)= \sup\left\{ \lambda_{Y}(x,y,a)
-\lambda(x,a)^\top p - \frac{1}{2}\mathrm{Tr}\left(\sigma\sigma^\top(x,a)M\right)~:~
a\in\Nc_\varepsilon(x,p) \right\}
\enqs
 for $\Theta =(x,y,p,M) \in \R^d \times \R \times \R^d \times \S^d$ and $\varepsilon\geq0$, and
\beqs
\Nc_\varepsilon(x,p) = \left\{ a\in A~:~ \left|N^a(x,p)\right|\leq \varepsilon \right\} ~ \text{ and }~ N^a(x,p) = \sigma_{Y}(x,a)-\sigma(x,a)^\top p
\enqs
 for $x,p\in \R^d$. Observe that $(\Nc_\varepsilon)_{\eps\geq 0}$ is non-decreasing so that
\beq\label{main:F*_liminf_F0}
F_*(\Theta) = \liminf_{\Theta'\to \Theta} F_0(\Theta')
\enq

Since some $\Nc_\varepsilon(x,p)$ may be empty, we shall use the standard convention $\sup \emptyset = -\infty$ all over this paper.
For ease of notations, we also write $F \varphi(t,x)$ in place of $F(x,\varphi(t,x),D  \varphi(t,x),D^2_x \varphi(t,x))$ for a regular function $\varphi$. We similarly use the notations $F^* \varphi$ and $F_* \varphi$.

As the value function may not be regular, we use the framework of discontinuous viscosity solutions.
To this end, we define the lower- and upper-semicontinuous envelopes $f_*$ and $f^*$  of a locally bounded function $f:[0,T]\times \R^d\times \Ic \to \R$ by
\beq\label{def-env}
f^*_i(t,x) = \limsup_{\begin{tiny}\begin{array}{c}
     (t',x')\rightarrow (t,x) \\
     t'<T
\end{array}\end{tiny}}f_i(t',x') & \text{ and } &
f_{i,*}(t,x) = \liminf_{\begin{tiny}\begin{array}{c}
     (t',x')\rightarrow (t,x) \\
     t'<T
\end{array}\end{tiny}}f_i(t',x') ~
\enq
for $(t,x,i)\in[0,T]\times\R^d\times\Ic$. We are now able to provide the definition of a viscosity solution to \eqref{PDE1}.
\begin{Definition} Let $u:~[0,T]\times \R^d\times \Ic\rightarrow\R$ be a locally bounded function.

\noindent (i)  $u$ is a \emph{viscosity supersolution} to \eqref{PDE1} if for any $(t_0,x_0,i_0)\in[0,T)\times\R^d\times\Ic$ and any $\varphi_i\in C^{1,2}([0,T]\times \mathbb{R}^{d})$ for $i\in\Ic$ and $\bar{\varphi}\in C^0([0,T]\times \mathbb{R}^{d})$ such that
\beqs
\sup_{i\in\Ic}|\varphi_i(t,x)| &\leq& \bar{\varphi}(t,x)~,\quad  \forall (t,x)\in[0,T]\times \mathbb{R}^{d}~,\\
0~=~\left(u_{i_0,*}- \varphi_{i_0}\right)(t_{0},x_{0}) & = & \min_{\Ic\times [0,T]\times \mathbb{R}^{d}}\left(u_{\cdot,*}-\varphi_\cdot\right)~.
\enqs
we have
\beqs
\min \left\{-\partial _t\varphi_{i_0}(t_0,x_0)+ F^*\varphi_{i_0}(t_0,x_0)~; ~\left( \varphi_{i_0} - \sup_{0\leq k<\bar{K}}\varphi_{i_0k}\right)(t_0,x_0)\right\}
& \geq & 0\;.
\enqs
\noindent (ii)  $u$ is a \emph{viscosity subsolution} to \eqref{PDE1} if for any $(t_0,x_0,i_0)\in[0,T)\times\R^d\times\Ic$ and any
$\varphi_i\in C^{1,2}([0,T]\times \mathbb{R}^{d})$ for $i\in\Ic$ and $\bar{\varphi}\in C^0([0,T]\times \mathbb{R}^{d})$ such that
\beqs
\sup_{i\in\Ic}|\varphi_i(t,x)| &\leq& \bar{\varphi}(t,x)~,\quad  \forall (t,x)\in[0,T]\times \mathbb{R}^{d}~,\\
0~=~\left(u_{i_0}^*- \varphi_{i_0})(t_{0},x_{0}\right) & = & \max_{\Ic\times [0,T]\times \mathbb{R}^{d}}\left(u_{\cdot}^*-\varphi_\cdot\right)~.
\enqs
we have
\beqs
\min \left\{-\partial _t\varphi_{i_0}(t_0,x_0)+ F_*\varphi_{i_0}(t_0,x_0)~; ~\left( \varphi_{i_0} - \sup_{0\leq k<\bar{K}}\varphi_{i_0k}\right)(t_0,x_0)\right\}
& \leq & 0\;.
\enqs

\noindent (iii)  $u$ is a \emph{viscosity solution} to \eqref{PDE1} if it is both a viscosity sub and supersolution to \eqref{PDE1}.
\end{Definition}

We notice that the definition of viscosity solution is slightly different from the classical one as  we impose a bound in the label particle $i$ for test functions.

Following \cite{BouchardElieTouzi}, we introduce the a continuity assumption on the kernel that is used to prove the subsolution property.

\vspace{2mm}

\begin{Assumption}\label{Assumption_H_2}
Let $B$ be a subset of $\R^d \times \R^d$ such that $\Nc_0\neq \emptyset$ on $B$. Then, for every $\varepsilon>0$, $(x_0,p_0)\in \text{int}(B)$, and $a_0 \in \Nc_0(x_0,p_0)$, there exists an open neighborhood $B'$ of $(x_0,p_0)$ and a locally Lipschitz map $\hat a$ defined on $B'$ such that $|\hat a(x_0,p_0)- a_0 |  \leq   \varepsilon $ and
\beqs
\hat a(x,p) \in \Nc_0(x,p) ~\text{ for all }(x,p)\in B'~.
\enqs
\end{Assumption}

We are now able to state our result.

\begin{Theorem}\label{theorem:result_PDE}
Suppose that $\bar v$ is locally bounded on $[0,T]\times\R^d\times\Ic$.
\begin{enumerate}[(i)]
    \item Under Assumptions \ref{Assumption_H_0}, the value function $\bar{v}_{}$ is a viscosity supersolution to \eqref{PDE1}
    \item If in addition Assumption \ref{Assumption_H_2} holds, $\bar v$ is a viscosity subsolution to \eqref{PDE1}
\end{enumerate}
\end{Theorem}

\subsubsection{Terminal condition}

To get a complete characterisation of the function $\bar v$, we need to add a terminal equation to \eqref{PDE1}. By the definition of the stochastic target problem, we have
\beq\label{NatTermCond}
\bar v_i(T,x) = g_i(x)
\enq
 for every $(x,i)\in\R^d\times\Ic$.
The possible discontinuities of $\bar v$ might imply that $\bar v_{*}$ and $\bar v^*$ do not agree with the  boundary condition \eqref{NatTermCond}. To get the proper terminal condition, we introduce the set-valued map
\beqs
\mathbf{N}(x,p) = \{r \in\R^m~:~r=N^a(x,p) ~\text{ for some }a \in A\}
\enqs
together with the signed distance function from its complement set $\mathbf{N}^c$ to the origin
\beqs
\delta = \text{dist}(0,\mathbf{N}^c) - \text{dist}(0,\mathbf{N})~,
\enqs
where dist stands for the Euclidean distance. Then,
\beq\label{eq:delta>0}
0\in \text{int} \mathbf{N}(x,p)  &  \Leftrightarrow  & \delta(x,p)>0
~.
\enq
For simplicity of notations, we will write $\delta\varphi(x)$ for $\delta(x, D  \varphi(x))$ for a regular function $\varphi$.
Then, the terminal condition takes the following form
\beq\label{PDE2}
\min \left\{~\bar{v}_i(T,x)-g_i(x)~;~\delta\bar{v}_i(T,x) ~;~ \left(\bar v_{i} - \sup_{0\leq k<\bar{K}}\bar v_{ik}\right)(T,x)~\right\}& = & 0
\enq
for $(x,i)\in\R^d\times\Ic$.

We give the definition of a viscosity solution to \eqref{PDE2}. We recall that the definitions of the envelopes $u^*$ and $u_*$ of a locally bounded function $u$ are given by \eqref{def-env}.

\begin{Definition}\label{def-viscosolT} Let $u:[0,T]\times \R^d\times \Ic\rightarrow\R$ be a locally bounded function.

\noindent (i)  $u$ is a \emph{viscosity supersolution} to \eqref{PDE2} if for any $(x_0,i_0)\in\R^d\times\Ic$ and any
$\varphi_i\in C^{2}( \mathbb{R}^{d})$ for $i\in\Ic$ and $\bar{\varphi}\in C^0(\mathbb{R}^{d})$ such that
\beqs
\sup_{i\in\Ic}|\varphi_i(x)| &\leq& \bar{\varphi}(x)~,\quad  \forall x\in \mathbb{R}^{d}~,\\
0~=~u_{i_0}^*(T,x_0)- \varphi_{i_0}(x_0) & = & \min_{\Ic\times\mathbb{R}^{d}}(u^{*}_{\cdot}(T,\cdot)-\varphi_{\cdot})
\enqs
we have
\beqs
\min\left\{~\varphi_{i_0}(x)-g_{i_0}(x_0)~;~\delta_*\varphi_{i_0}(x_0)~;~  \varphi_{i_0}(T,x_0) - \sup_{0\leq k<\bar{K}}\varphi_{i_0k}(T,x_0)~\right\} & \geq & 0\;.
\enqs
\noindent (ii)  $u$ is a \emph{viscosity subrsolution} solution to \eqref{PDE2} if for any $(x_0,i_0)\in\R^d\times\Ic$ and any
$\varphi_i\in C^{2}( \mathbb{R}^{d})$ for $i\in\Ic$ and $\bar{\varphi}\in C^0(\mathbb{R}^{d})$ such that
\beqs
\sup_{i\in\Ic}|\varphi_i(x)| &\leq& \bar{\varphi}(x)~,\quad  \forall x\in \mathbb{R}^{d}~,\\
0~=~u_{i_0,*}(T,x_0)- \varphi_{i_0}(x_0) & = & \max_{\Ic\times\mathbb{R}^{d}}(u_{\cdot,*}(T,\cdot)-\varphi_\cdot)~
\enqs
we have
\beqs
\min\left\{\left(\varphi_{i_0}(x)-g_{i}^{}(x)\right)\1_{F^*\varphi_{i_0}(x)<\infty};~\delta^*\varphi_{i_0}(x)~;~ \varphi_{i_0}(T,x) - \sup_{0\leq k<\bar{K}}\varphi_{i_0k}(T,x)~\right\} & \leq & 0\;.
\enqs

\noindent (iii)  $u$ is a \emph{viscosity solution} to \eqref{PDE2} if it is both a viscosity sub and supersolution to \eqref{PDE2}.
\end{Definition}
The terminal viscosity property is stated as follows.

\begin{Theorem}\label{theorem:result_PDET}
Suppose that $\bar v$ is locally bounded on $[0,T]\times\R^d\times\Ic$.
\begin{enumerate}[(i)]
    \item Under Assumptions \ref{Assumption_H_0} and \ref{Assumption_H_1}, $\bar v$  is a viscosity supersolution to \eqref{PDE2}.
    \item  If in addition Assumption \ref{Assumption_H_2} holds, $\bar v$
    is a viscosity subsolution to \eqref{PDE2}.
\end{enumerate}
\end{Theorem}

\subsection{Viscosity properties on \texorpdfstring{$[0,T)\times\R^d\times\Ic$}{f} }

\subsubsection{Viscosity supersolution property}\label{supersol_section}

Fix $(i_0,t_{0},x_{0})\in\Ic\times[0,T)\times \mathbb{R}^{d}$ and let $\varphi\in C^0([0,T]\times\R^d)$ and $\varphi_i\in C^{1,2}([0,T]\times \mathbb{R}^{d})$ for $i\in \Ic$ be such that
\beq\label{cond-domphi}
\sup_i|\varphi_i| & \leq & \varphi
\enq
and
\beq\label{condMin}
0=\left(\bar{v}_{i_0,*}- \varphi_{i_0}\right)(t_{0},x_{0})=\min_{(i,t,x)\in \Ic\times[0,T]\times \mathbb{R}^{d}}\left(\bar{v}_{i,*}-\varphi_{i}\right)(t,x)~.
\enq
Without loss of generality we can assume this minimum to be strict in $(t,x)$ once fixed $i_0$.

\paragraph{Step 1.} We first prove that $\varphi_{i_0}(t_0,x_0) - \sup_{0\leq \ell\leq k-1} \varphi_{i_0k}(t_0,x_0)\geq 0$ for any $k$ such that $p_k>0$.
 Let $(t_{n}, x_{n})$ be a sequence in $[0,T]\times \mathbb{R}^{d}$ such that
\beqs
(t_{n}, x_{n})\rightarrow(t_{0},x_{0})~\text{ and }~\bar{v}_{i_0}(t_{n},x_{n})\rightarrow \bar{v}_{i_0,*}(t_{0},x_{0})~\text{ as }n\rightarrow\infty.
\enqs
Set $y_{0}:=\varphi_{i_0}(t_0,x_0)$, $\hat{x}_{0}:=(x_{0}, y_{0})$, $y_{n}:=\bar{v}_{i_0}(t_{n},x_{n})+1/n$ and $\hat{x}_{n}:=(x_{n},y_{n})$.
Define the stopping time $\theta_n=\inf\{ s \geq t_n~:~Q^{i_0}((t_n,s]\times\N)\geq 1\}$ and the random variable $k_n$ such that $Q^{i_0}((t_n,\theta_n]\times\{ k_n\})=1$. From Theorem \ref{theorem:DPP}, the continuity of the trajectories and since $y_n>\bar v_{i_n}(t_n,x_n)$ there exists $\alpha^n\in \Ac$ such that
\beqs
Y^{t_n,\delta_{(i_0,\hat x_n)},\alpha^n,i_0}_{\theta_n-} & \geq & \max_{0\leq \ell\leq k_n-1}\bar v_{i_0 \ell}\left(\theta_n,X^{t_n,\delta_{(i_0,x_n)},\alpha^n,i_0}_{\theta_n-}\right)~ \geq ~ \max_{0\leq \ell\leq k_n-1}\varphi_{i_0\ell}\left(\theta_n,X^{t_n,\delta_{(i_0,x_n)},\alpha^n,i_0}_{\theta_n}\right)\;.
\enqs
on $\{\theta_n\leq T\}$.  To alleviate the notation, we shall denote $X^{n,i}_t := {X}^{t_n,\delta_{(i_0,{x}_n)}, \alpha^n, i}_t$ and $Y^{n,i}_t := {Y}^{t_n,\delta_{(i_0,{x}_n)},y_n, \alpha^n, i}_t$ for $n\geq 1$ and $t\in[t_n,T]$.
Therefore, we get
\beqs
\gamma\sum_{k=0}^{+\infty}p_k\int_0^T \E\left[\mathds{1}_{s\leq \theta_n\leq T} \mathds{1}_{Y^{n,i_0}_{s}< \max_{0\leq \ell\leq k-1} \varphi_{i_0\ell} \left(s,X^{n,i_0}_{s}\right)}  \right]\d s & = & 0\;,
\enqs
which means
\beqs
\int_0^T \E\left[ \mathds{1}_{s\leq \theta_n\leq T}\mathds{1}_{Y^{n,i_0}_{s}< \max_{0\leq \ell\leq k-1} \varphi_{i_0\ell}\left(s,X^{n,i_0}_{s}\right)} \right]\d s & = & 0
\enqs
for all $k\geq 1$ such that $p_k>0$. We therefore get
\beq\label{condAEs}
\E\left[ \mathds{1}_{s\leq \theta_n\leq T}\mathds{1}_{Y^{n,i_0}_{s}< \max_{0\leq \ell\leq k-1} \varphi_{i_0\ell} \left(s,X^{n,i_0}_{s}\right)} \right] & = & 0
\enq
for Lebesgue almost all $s\in[t_n,T]$.
Since the process $Y^{n,i_0}-\max_{0\leq \ell\leq k-1} \varphi_{i_0\ell} \left(\cdot,X^{n,i_0}\right)$ is continuous and $\P\left(\theta_n\in[t_n,T]\right)>0$, Fatou's Lemma applied to a sequence $(s_k)_k$ converging to $t_n$ and satisfying \eqref{condAEs} gives
\beqs
y_n & \geq & \max_{0\leq \ell\leq k-1} \varphi_{i_0\ell}(t_n,x_n)
\enqs
for all $k\geq 1$ such that $p_k>0$. Sending $n$ to infinity gives the result.

\paragraph{Step 2.} We now prove that
\beqs
-\frac{\partial\varphi_{i_0}}{\partial t}(t_0,x_0) + F^*\varphi_{i_0}(t_0,x_0) & \geq & 0
\enqs
Assume to the contrary that $(-\partial_t\varphi_{i_0}+ F^*\varphi_{i_0})(t_0,x_0) = -2\eta$ for some $\eta > 0$, and let us work towards a contradiction. By definition of $F^*$ , we may find $\varepsilon \in (0,T-t_0)$, such that
\beq
-\partial_t \varphi_{i_0}(t,x) +\lambda_Y(x,y,a) - L^a\varphi_{i_0}(t,x)\leq -\eta \quad \text{ for all }a\in\Nc_\varepsilon(x,D \varphi_{i_0}(t,x))\label{supersol:geq_eta}\\
\text{and } (t,x,y)\in[0,T]\times\R^d\times\R \text{ such that } (t,x)\in B_\varepsilon(t_0,x_0)\text{ and } |y-\varphi_{i_0}(t,x)|\leq \varepsilon~, \nonumber
\enq
where $B_\varepsilon(t_0,x_0)$ denotes the ball of radius $\varepsilon$ around $(t_0,x_0)$. Let $\partial_p B_\varepsilon(t_0,x_0) = \{t_0 + \varepsilon\} \times\text{cl}(B_\varepsilon(t_0,x_0)) \cup [t_0,t_0+  \varepsilon) \times \partial B_\varepsilon (x_0)$ denote the parabolic boundary of $B_\varepsilon(t_0,x_0)$ and observe that
\beq
\zeta = \min_{\partial_p B_\varepsilon(t_0,x_0)}(\bar v_{i_0,*} -\varphi_{i_0}) > 0 \label{supersol:parabolic_min}
\enq
since $(t_0,x_0)$ is a strict minimizer of $\bar v_{i_0,*} -\varphi_{i_0}$ on  $[0,T)\times\R^d$.

\paragraph{Step 3.} We now show that \eqref{supersol:geq_eta} and \eqref{supersol:parabolic_min} lead to a contradiction to \eqref{eq:DPP1}.
Let $(t_{n}, x_{n})$ in $[0,T]\times \mathbb{R}^{d}$ such that
\beqs
(t_{n}, x_{n})\rightarrow(t_{0},x_{0})~\text{ and }~\bar{v}_{i_0}(t_{n},\ x_{n})\rightarrow \bar{v}_{i_0,*}(t_{0},x_{0})~\text{ as }n\rightarrow\infty.
\enqs
We then set $y_{0}:=\varphi_{i_0}(t_{0},x_{0})$, $\hat{x}_{0}:=(x_{0}, y_{0})$, $y_{n}:=\bar{v}_{i_0}(t_{n},x_{n})+1/n$, $\hat{x}_{n}:=(x_{n},y_{n})$, $\beta_{n}:=y_{n}-\varphi_{i_0}(t_{n}, x_{n})$ and notice that
\beq\label{supersol:beta_to_0}
\beta_{n}\rightarrow 0~\text{ as }n\rightarrow\infty~.
\enq
From the definition of the value function and the fact that $y_{n}>\bar{v}_{i_0}(t_{n},\ x_{n})$ for each $n\geq1$, there exists some $\alpha^{n}$ in $\Ac$ such that $Y^{t_n,\delta_{(i_0,{x}_n)},y_n, \alpha^n, i}_T \geq g_i\left(X^{t_n,\delta_{(i_0,{x}_n)}, \alpha^n, i}_T\right)$ for all $i \in \Vc^{t_n,\delta_{(i_0,{x}_n)}, \alpha^n}_T$. To alleviate the notation, we shall denote
\beqs
{X}^{n,i}_t := {X}^{t_n,\delta_{(i_0,{x}_n)}, \alpha^n, i}_t~,~
{Y}^{n,i}_t := {Y}^{t_n,\delta_{(i_0,{x}_n)},y_n, \alpha^n, i}_t~\text{and} ~
\Vc^{n}_t := \Vc^{t_n,\delta_{(i_0,{x}_n)}, \alpha^n}_t~
\enqs
for $n\geq 1$ and $t\in[t_n,T]$.
Define the following stopping times
\beqs
\tau_{n}&:=&\inf\{s\geq t_{n}~:~ \exists i \in \Vc^n_s~,~\left(s,X^{n,i}_s\right)\notin B_\eps(t_0,x_0)\}~,\\
\tau_{n}^\varepsilon&:=&\inf\{s\geq t_{n}~:~ \exists i \in \Vc^n_s~,~|Y^{n,i}_s -\varphi_i\left(s,X^{n,i}_s\right)|\geq\eps\}~,\\
\tau^r_{n}&:=&\inf\{s\geq t_{n}~:~Q^{i_0}((t_n,s]\times \N )=1
~\}~,\\
\theta_n &:=&\tau_{n}\wedge\tau_{n}^\eps
\wedge\tau_{n}^r
~.
\enqs
We also set
\beq\label{supersol:def_A_n}
A_n &= &\left\{s\in[t_n,\theta_n)\;:\;
-\partial_t\varphi_{i_0}(s,{X}^{n,i_0}_s)+ \lambda_Y({X}^{n,i_0}_s,{Y}^{n,i_0}_s,\alpha^n_{i_0}) - L^{\alpha^n_{i_0}}\varphi_{i_0}(s,{X}^{n,i_0}_s)> -\eta\right\}\;,~\qquad  \\
\psi^{n}_s& =& N^{\alpha^n_{i_0}}({X}^{n,i_0}_s,D  \varphi_{i_0}(s,{X}^{n,i_0}_s))\;.\nonumber
\enq
We notice that \eqref{supersol:geq_eta} implies
\beq\label{supersol:psi_geq_eps}
|\psi^{n}_s|&>&\eps \quad \text{ for }s \in A_n~.
\enq
It follows from Theorem \ref{theorem:DPP} that
\beqs
Y^{n, i}_{t\wedge\theta_n}  & \geq &  \bar{v}_i\left(t\wedge\theta_n,X^{n, i}_{t\wedge\theta_n}\right)\quad \forall i \in \Vc^{n}_{t\wedge\theta_n},~t\in[ t_n,T]\;.
\enqs
and since $\bar{v}_{i}\geq \bar{v}_{i,*}\geq\varphi_i$
\beq\label{supersol:Y_geq_phi_i}
Y^{n, i}_{\theta_n\wedge t} \geq \varphi_i\left(\theta_n\wedge t,X^{n, i}_{\theta_n}\wedge t\right)\quad \forall i \in \Vc^{n}_{\theta_n}\;.
\enq
Using the definition of $\zeta$ in \eqref{supersol:parabolic_min} and $\theta_n$, and the continuity of the trajectories, we get
\beqs
Y^{n, i_0}_{t\wedge\theta_n} &\geq& \varphi_{i_0}\left(t\wedge\theta_n,X^{n, i_0}_{t\wedge\theta_n}\right) + (\zeta\1_{\{ \theta_n=\tau_n \}} +
\eps\1_{\{\tau_n^\varepsilon = \theta_n\}\cap\{\theta_n<\tau_n\}}
)\1_{\{\theta_n\leq  t\}\cap\{\theta_n<\tau^r_n\}} \\
&\geq& \varphi_{i_0}\left(t\wedge\theta_n,X^{n, i_0}_{t\wedge\theta_n}\right) + \zeta\wedge\eps\1_{\{\theta_n\leq t\}\cap\{\theta_n< \tau^r_n\}}~.
\enqs
Therefore, from \eqref{supersol:Y_geq_phi_i} and the previous inequality, we have
\beqs
- \zeta\wedge\eps\1_{\{\theta_n>t\}\cup \{\theta_n=\tau^r_n\}} \leq -\zeta\wedge\eps + Y^{n, i}_{t\wedge\theta_n} - \varphi_{i}\left(t\wedge\theta_n,X^{n, i}_{t\wedge\theta_n}\right)
\enqs
Applying the dynamics \eqref{dynhatZ}  of $\hat{Z}^{t_n,\delta_{(i_0,\hat{x}_n)},\alpha^n}_{\cdot}$ to the function $(t,x,y,i)\mapsto y-\varphi_{i_0}(t,x)$, it follows from the definition of $\psi_n$ and $\theta_n$, and \eqref{supersol:def_A_n} that
\beqs\nonumber
- \zeta\wedge\eps\1_{\{\theta_n>t\}\cup \{\theta_n=\tau^r_n\}} &\leq&  \beta_{n} - \zeta\wedge\eps + \int_{t_{n}}^{t\wedge\theta_n}
\psi^n_s {}^\top\d B^{i_0}_u \nonumber\\
&+& \int_{t_{n}}^{t\wedge\theta_n}\left[
-\partial_t\varphi_{i_0}\left(u,X^{n,i_0}_u\right)+
\lambda_{Y}\left(X^{n, i_0}_u, Y^{n, i_0}_u,\alpha^{n,i_0}_u\right) -  L^{\alpha^{n,i_0}_u}\varphi_{i_0}\left(u,X^{n,i_0}_u\right)\right]\d u \nonumber\\
& + & \int_{(t_{n},\theta_n\wedge t]} \sum_{k\geq 0}
\left((k-1)Y^{n, i_0}_u - \left(\sum_{\ell=0}^{k-1}\varphi_{i_0\ell} - \varphi_{i_0} \right)\left(u,X^{n, i_0}_u\right)\right) Q^{i_0}(\d u\d k)\nonumber
\\
&\leq&  \beta_{n}- \zeta\wedge\eps + \int_{t_{n}}^{t\wedge\theta_n}
\psi^n_s {}^\top\d B^{i_0}_u \nonumber\\
&+&\int_{t_{n}}^{t\wedge\theta_n}
\left[ -\partial_t\varphi_{i_0}\left(u,X^{n,i_0}_u\right)+ \lambda_{Y}\left(X^{n, i_0}_u, Y^{n, i_0}_u,\alpha^{n,i_0}_u\right) -  L^{\alpha^{n,i_0}_u}\varphi_{i_0}\left(u,X^{n,i_0}_u\right)\right]\1_{A_n}(u)
\d u \nonumber\\
& + & \int_{(t_{n},\theta_n\wedge t]} \sum_{k\geq 0}
\left((k-1)Y^{n, i_0}_u - \left(\sum_{\ell=0}^{k-1}\varphi_{i_0\ell} - \varphi_{i_0} \right)(u,X^{n, i_0}_u)\right) Q^{i_0}(\d u\d k)~.\nonumber\\
\enqs
We then get
\beq
- \zeta\wedge\eps\1_{\{\theta_n>t\}\cup \{\theta_n=\tau^r_n\}} &\leq& M^{B,n}_{t\wedge\theta_n}  + M^{Q,n}_{t\wedge\theta_n}
~,\label{supersol:infinitesimal_development}
\enq
where
\beqs
M^{B,n}_{s}&=& \beta_{n}- \zeta\wedge\eps +\int_{t_{n}}^{s}b^n_u \d u + \int_{t_{n}}^{s}
\psi^n_s {}^\top\d B^{i_0}_u~,\\
b^n_s &=&
\left[-\partial_t\varphi\left(s,X^{n,i_0}_s\right)+ \lambda_{Y}\left(X^{n, i_0}_s, Y^{n, i_0}_s,\alpha^{n,i_0}_s\right) -  L^{\alpha^{n,i_0}_s}\varphi\left(s,X^{n,i_0}_s\right)\right]\1_{A_n}(s) +\\
&&+ \sum_{k\geq 0}
\left((k-1)Y^{n, i_0}_s - \left(\sum_{\ell=0}^{k-1}\varphi_{i_0\ell} - \varphi_{i_0} \right)\left(s,X^{n, i_0}_s\right)\right) \gamma p_k
~,\\
M^{Q,n}_{s}  &=& \int_{(t_{n},s]} \sum_{k\geq 0}
\left((k-1)Y^{n, i_0}_u - \left(\sum_{\ell=0}^{k-1}\varphi_{i_0\ell} - \varphi_{i_0} \right)(u,X^{n, i_0}_u)\right) \left(Q^{i_0}(\d u\d k) - \gamma p_k \d u \right)
~.
\enqs
for $s\in[t_n,T]$.
From to Step 1, the definition of $\theta_n$, the domination condition \eqref{cond-domphi} and Assumption \ref{Assumption_H_0}, $M^{Q,n}_{\cdot\wedge\theta_n}$ is a pure jump martingale.
Let $L^n$ be the exponential local martingale defined by $L^n_{t_n}=1$ and
\beqs
\d L^n_{s} = -L^n_{s} b^n_s |\psi^n_s|^{-2} \psi^n_s \1_{A_n}(s){}^\top\d B^{i_0}_s
\enqs
for $s\in[t_n,T]$. $L^n$ is well defined by \eqref{supersol:psi_geq_eps}, Assumption \ref{Assumption_H_0} and the definition of the set of admissible controls $\Ac$. Moreover, From the definition of $\theta_n$, $L^n_{\cdot\wedge\theta_n}$ is a martingale. From Girsanov Theorem for jump diffusion processes (see e.g. Theorem 1.35 in \cite{OS2007}) and the definition of $\theta_n$, we get that $L^n_{\cdot\wedge\theta_n}M^{B,n}_{\cdot\wedge\theta_n}+L^n_{\cdot\wedge\theta_n}M^{Q,n}_{\cdot\wedge\theta_n}$ is a martingale.  It follows from \eqref{supersol:infinitesimal_development} that
\beqs
- \zeta\wedge\eps\E[\1_{\{\theta_n=\tau^r_n\}}L^n_{\theta_n}]  & \leq &
\E\left[L^n_{\theta_n}M^{B,n}_{\theta_n} + L^n_{\theta_n}M^{Q,n}_{\theta_n}\right] \\
&\leq&
L^n_{t_n}M^{B,n}_{t_n}+L^n_{t_n}M^{Q,n}_{t_n} = \beta_n - \zeta \wedge \eps\;.
\enqs
Since $L^n_{\cdot\wedge\theta_n}$ is a martingale and $\theta_n$ is a stopping time bounded by $\eps$, we have $\E[L^n_{\theta_n}]=L^n_{t_n}=1$. Therefore, the previous inequality becomes
\beq
\zeta\wedge\eps\E\left[\1_{\{\theta_n<\tau^r_n\}}L^n_{\theta_n}\right] & \leq & \beta_n~. \label{supersol:last_inequality}
\enq
We next define the probability measure on $\Fc_T$ by the Radon-Nikodym derivative
\beqs
\left.\frac{\d \P^n}{d\P}\right|_{\Fc_T} & = & L^n_{\theta_n}
\enqs
and denote by $\E^n$ the expectation under $\P^n$. Using Girsanov Theorem, we notice that $\tau_n ^r$ has the same law under $\P$ and $\P^n$. In particular, we have
\beqs
\E\left[\1_{\{\theta_n<\tau^r_n\}}L^n_{\theta_n}\right] & \geq & \E\left[\1_{\{\tau^r_n>\eps\}}L^n_{\theta_n}\right]~~=~~\E^n\left[\1_{\{\tau^r_n>\eps\}}\right]
~~ = ~~ \E\left[\1_{\{\tau^r_n>\eps\}}\right]~~=~~\exp(-\eps\gamma)\;.
\enqs
Comparing with \eqref{supersol:last_inequality}, we have
\beqs
0 \leq \beta_n - \zeta\wedge\eps~ \exp(-\eps\gamma)~,
\enqs
which contradicts \eqref{supersol:beta_to_0} for $n$ large enough.

\subsubsection{Viscosity subsolution property}\label{subsubsection:visc_subsol}
\paragraph{Step 1.}
Let $\varphi\in C^0([0,T]\times\R^d)$, $\varphi_i\in C^{1,2}([0,T]\times \R^{d})$ for $i\in\Ic$ and $(t_{0},x_{0},i_0)\in [0,T)\times\R^d\times \Ic$ such that
\beqs
\sup_i|\varphi_i| & \leq & \varphi
\enqs
and
\beq\label{subsol:condMax}
0=\left(\bar{v}_{i_0}^*- \varphi_{i_0}\right)(t_{0},x_{0})=\max_{(t,x,i)\in [0,T]\times \mathbb{R}^{d}\times\Ic}\left(\bar{v}_{i}^*-\varphi_{i}\right)(t,x)~.
\enq
Without loss of generality we can assume that the maximum is strict in $(t,x)$ once fixed $i_0$. We  then argue by contradiction and assume that
\beq\label{subsol:4_eta}
4 \eta ~=~ \min\left\{
\left(-\partial_t\varphi_{i_0} + F_*\varphi_{i_0}\right)(t_0,x_0)~;~
\left(\varphi_{i_0} - \sup_{0\leq k<\bar{K}} \varphi_{i_0k}\right)(t_0,x_0)\right\}  > 0~\;.
\enq
By \eqref{main:F*_liminf_F0}, Assumption \ref{Assumption_H_2} and \eqref{subsol:4_eta}
we may find $\eps>0$ such that
\beq
\label{subsol:inside_F_geq_eta}
\rho(t,x,y) = -\partial_t\varphi_{i_0}(t,x) + \lambda_Y(x,y,
\hat{a}(x,D  \varphi_{i_0}(t,x))) -
L^{\hat{a}\left(x,D  \varphi_{i_0}(t,x)\right)}\varphi_{i_0}(t,x) &\geq& \eta ~,
\\
\label{subsol:jump_geq_eta}
\left(\varphi_{i_0} - \sup_{0\leq k<\bar{K}} \varphi_{i_0k}\right)(t,x)&\geq& \eta
\enq
for all $(t,x,y)\in[0,T)\times\R^d\times\R$ such that $(t,x)\in B_\eps(t_0,x_0)$ and $ |y-\varphi_{i_0}(t,x)|\leq \eps$, where $\hat{a}$ is a locally Lipschitz map satisfying
\beq\label{subsol:param_kernel}
\hat{a}(x,D  \varphi_{i_0}(t,x)) \in \Nc_0(x,D  \varphi_{i_0}(t, x)) \text{ on } B_\varepsilon(t_0,x_0)~.
\enq
Observe that, since $(t_0,x_0)$ is a strict maximizer, we have
\beq\label{subsol:parabolic_max_zeta}
- \zeta &  = & \max_{\partial_p B_\eps(t_0,x_0)}(\bar{v}_{i_0}^*- \varphi_{i_0})(t,x) ~~<~~0~,
\enq
where $\partial_p B_\varepsilon(t_0,x_0) = \{t_0+\eps\}\times\text{cl}(B_\eps(t_0,x_0))\cup[t_0,t_0 +\eps)\times\partial B_\varepsilon(t_0,x_0)$ denotes the parabolic boundary of $B_\varepsilon(t_0,x_0)$.

\paragraph{Step 2.}
We now show that \eqref{subsol:inside_F_geq_eta}, \eqref{subsol:jump_geq_eta}, \eqref{subsol:param_kernel} and \eqref{subsol:parabolic_max_zeta} lead to a contradiction to \eqref{eq:DPP1}.
Let $(t_{n}, x_{n})_{n\geq 1}$ be a sequence such that
\beqs
(t_{n},x_{n})\rightarrow(t_{0},x_{0})~\text{ and }\bar{v}_{i_0}(t_{n},x_{n})\rightarrow \bar{v}^{*}_{i_0}(t_{0},x_{0}) & \mbox{ as } & n \rightarrow+\infty\;.
\enqs
 Set $y_{0}:=\varphi_{i_0}(t_{0},x_{0})$, $\hat{x}_{0}:=(x_{0}, y_{0})$  and $y_{n}:=\bar{v}_{i_0}(t_{n},x_{n})-n^{-1}$, $\hat{x}_{n}:=(x_{n},y_{n})$ for $n\geq 1$ and notice that
\beq\label{subsol:beta_to_zero}
\beta_n := y_n-\varphi_{i_0}(t_{n},x_{n}) & \xrightarrow[n\rightarrow+\infty]{} & 0\;.
\enq
Define the following stopping times
\beqs
\tau_{n}&:=&\inf\{s\geq t_{n}~:~ \exists i \in \Vc^n_s~,~\left(s,X^{n,i}_s\right)\notin B_\eps(t_0,x_0)\}~,\\
\tau_{n}^\varepsilon&:=&\inf\{s\geq t_{n}~:~ \exists i \in \Vc^n_s~,~|Y^{n,i}_s -\varphi_{i}\left(s,X^{n,i}_s\right)|\geq\eps\}~,\\
\tau^r_{n}&:=&\inf\{s\geq t_{n}~:~Q^{i_0}((t_n,s]\times \N ) \geq1
~\}~,\\
\theta_n &:=&\tau_{n}\wedge\tau_{n}^\eps
\wedge\tau_{n}^r
~.
\enqs
To alleviate the notations, we shall write
\beqs
{X}^{n,i}_. := {X}^{t_n,\delta_{(i_0,{x}_n)}, \alpha^n, i}_.~,~
{Y}^{n,i}_. := {Y}^{t_n,\delta_{(i_0,{x}_n)},y_n, \alpha^n, i}_.~,~\hat{X}^{n,i}_\cdot = \left({X}^{n,i}_\cdot,{Y}^{n,i}_\cdot\right)\\
\hat{Z}^{n}_\cdot =\hat{Z}^{t_n,\delta_{(i_0,\hat{x}_{n})}, \hat{\alpha}^n}_\cdot
~\text{ and }~
\Vc^{n}_\cdot =\Vc^{t_n,\delta_{(i_0,{x}_{n})}, \hat{\alpha}^n}_\cdot
\enqs
where  $\hat{\alpha}^{n}$ is the feedback control process given by
$\hat{\alpha}^{n,i}_\cdot=\hat{a}( {X}^{n,i}_\cdot,D  \varphi_{i_0}(\cdot, {X}^{n,i}_\cdot))$ defined on $[t_n,\theta_n)$ for $n\geq 1$. Since $\hat{a}$ is locally Lipschitz, this solution is well-defined.
Since $\bar{v}_{i} \leq \bar{v}_{i}^* \leq \varphi_{i}$, we
then deduce from \eqref{subsol:parabolic_max_zeta} and the definition of $\theta_n$ that on $\{\theta_n< \tau^{r}_n\}$ we have
\beqs
Y^{n,i_0}_{\theta_n} - \bar{v}_{i_0}\left(\theta_n, X^{n,i_0}_{\theta_n}\right) &\geq&
\1_{\{\theta_n= \tau^\eps_{n}\}}
\left(Y^{n,i_0}_{\theta_n} - \varphi_{i_0}\left(\theta_n, X^{n,i_0}_{\theta_n}\right)\right) \\
&&\qquad
+ \1_{\{\theta_n= \tau_{n}\}}
\left(Y^{n,i_0}_{\theta_n} - \bar{v}_{i_0}^*\left(\theta_n, X^{n,i_0}_{\theta_n}\right)\right)\\
&=&\eps\1_{\{\theta_n= \tau^\eps_{n}\}} + \1_{\{\theta_n= \tau_{n}<\tau_n^\eps\}}
\left(Y^{n,i_0}_{\theta_n} - \bar{v}_{i_0}^*\left(\theta_n, X^{n,i_0}_{\theta_n}\right)\right)\\
&\geq& \eps\1_{\{\theta_n= \tau_{n}^\eps\}} + \1_{\{\theta_n= \tau_{n}<\tau_n^\eps\}}
\left(Y^{n,i_0}_{\theta_n}+\zeta - \varphi_{i_0}\left(\theta_n, X^{n,i_0}_{\theta_n}\right)\right)
\\
&\geq& \eps\wedge\zeta+ \1_{\{\theta_n= \tau_{n}<\tau_n^\eps\}}
\left(Y^{n,i_0}_{\theta_n}- \varphi_{i_0}\left(\theta_n, X^{n,i_0}_{\theta_n}\right)\right)~.
\enqs
Secondly, on $\{\theta_n= \tau^{r}_n\}$, using the continuity of the trajectories of the particles $Y^{i_0\ell}_{\theta_n} = Y^{i_0}_{\theta_n}$ and $X^{i_0\ell}_{\theta_n} = X^{i_0}_{\theta_n}$ for all $i_0\ell \in \Vc^n_{\theta_n}$, we have
\beqs
Y^{n,i_0\ell}_{\theta_n}-\varphi_{i_0\ell}\left(\theta_n,X^{n,i_0\ell}_{\theta_n}\right)&=&
Y^{n,i_0}_{\tau^{r}_n}-\varphi_{i_0}\left(\tau^{r}_n,X^{n,i_0}_{\tau^{r}_n}\right)+\varphi_{i_0}\left(\tau^{r}_n,X^{n,i_0}_{\tau^{r}_n}\right)-\varphi_{i_0\ell}\left(\tau^{r}_n,X^{n,i_0 }_{\tau^{r}_n}\right)~,
\enqs
and from \eqref{subsol:jump_geq_eta},
\beq\label{subsol:inegbranches}
Y^{n,i_0\ell}_{\theta_n}-\varphi_{i_0\ell}\left(\theta_n,X^{n,i_0\ell}_{\theta_n}\right)&\geq&
Y^{n,i_0}_{\theta_n}-\varphi_{i_0}\left(\theta_n,X^{n,i_0}_{\theta_n}\right)+\eta~,
\enq
for all $i_0\ell \in \Vc^n_{\theta_n}$.

From \eqref{subsol:inside_F_geq_eta} and  \eqref{subsol:param_kernel}, we get by Itô's formula
\beqs
Y^{n,i}_{\theta_n} - \bar{v}_{i}\left(\theta_n, X^{n,i}_{\theta_n}\right) &\geq&\eps\wedge\zeta\wedge\eta+\beta_n \qquad \forall i\in \Vc^n_{\theta_n}\;.
\enqs
Since $y_n = \bar{v}_{i_0}(t_n, x_n) - n^{-1} < \bar{v}_{i_0}(t_n, x_n)$, this is  in contradiction with the dynamic programming principle \eqref{eq:DPP1} for sufficiently large $n$ by \eqref{subsol:beta_to_zero}.

\subsection{Viscosity properties on \texorpdfstring{$\{T\}\times\R^d\times\Ic$}{f}}
\subsubsection{Viscosity supersolution}

Let $(x_0,i_0)\in\R^d\times\Ic$ and $\varphi_i\in C^2(\R^d)$ for $i\in\Ic$ satisfying
\beqs
0 ~=~ \bar{v}_{i_0,*}(T,x_0) - \varphi_{i_0} (x_0)  & = &  \min_{\Ic\times\R^d} \left(\bar{v}_{\cdot,*}(T,\cdot) - \varphi_{\cdot}\right)\;.
\enqs
Without loss of generalities we can take this minimum to be strict in $x$ once fixed $i_0$.

 \paragraph{Step 1.} From the convention $\sup \emptyset :=-\infty$ and since $\bar{v}$ is a viscosity supersolution for \eqref{PDE1} on $[0,T)\times\R^d\times\Ic$, we have
\beqs
\delta^*\bar{v}_{\cdot,*}\geq 0 ~\text{ on }[0,T)\times\R^d\times\Ic
\enqs
in the viscosity sense. From the upper-semicontinuity of $\delta^*$, we can then deduce by a standard argument (see e.g. proof of Lemma 5.2 in \cite{SonerTouziSIAM}) that $\delta^* \varphi(x_0) \geq 0$.

\paragraph{Step 2.}
We now prove
\beqs
\varphi_{i_0}(x_0) - \sup_{0\leq k<\bar{K}} \varphi_{i_0k}(x_0) & \geq& 0\;.
\enqs
From the definition of $\bar{v}_{*}$, there exists a sequence $(s_n,\xi_n)_{n\geq1}$ converging to $(T,x_0)$ such that $s_n<T$ for $n\geq1$ and
\beqs
\lim_{n \to \infty}\bar{v}_{i_0,*}(s_n,\xi_n)=\bar{v}_{i_0,*}(T,x_0)~.
\enqs
For $n\geq 1$, consider the auxiliary test function
\beqs
\varphi_{n,i}(t,x):=\varphi_{i}(x)-\frac{1}{2}|x-x_{0}|^{2}+\frac{T-t}{(T-s_{n})^2}\;~(t,x,i)\in[0,T]\times\R^d\times\Ic.
\enqs
Let $B_1(x_0)$ be the unit open ball in $\R^d$ centered at $x_0$. Choose $(t_n , x_n)\in [s_n,T]\times \bar{B}_1(x_0)$, which minimizes the difference $\bar{v}_{i_0,*} - \varphi_{n,i_0}$ on $[s_n,T]\times \bar{B}_1(x_0)$.

We claim that, for $n$ large enough $t_n<T$, and $x_n$ converges to $x_0$.
Indeed, for sufficiently large $n$ we have
\beqs
(\bar{v}_{i_0,*} - \varphi_{n,i_0})(s_n,x_0)  & \leq &  - \frac{1}{(T-s_n)} ~< ~0\;.
\enqs
On the other hand, for any $x\in\bar{B}_1(x_0)$
\beqs
(\bar{v}_{i_0,*} - \varphi_{n,i_0})(T,x) & = & \bar{v}_{i_0,*}(T,x) - \varphi_{i_0}(x) + \frac{1}{2} |x-x_0|^2 ~\geq~ \bar{v}_{i_0,*}(T,x) - \varphi_{i_0}(x) ~\geq~ 0\;.
\enqs
Comparing the two inequalities leads us to conclude that $t_n < T$ for large $n$. Let $x^*$ be an adherence value of the sequence $(x_n)_{n\geq1}$.
Since $t_n \geq s_n$ and $(t_n , x_n)$ minimizes the difference $(\bar{v}_{i_0,*} - \varphi_{n,i_0})$, we have
\beqs
(\bar{v}_{i_0,*}(T,.) - \varphi_{i_0})(x^*) - (\bar{v}_{i_0,*}(T,.) - \varphi_{i_0})(x_0)   &\leq & \\
\liminf_{n\to\infty}(\bar{v}_{i_0,*} - \varphi_{n,i_0})(t_n,x_n) - (\bar{v}_{i_0,*} - \varphi_{n,i_0})(s_n,\xi_n) -  \frac{1}{2}|x_n-x_0|^2 &\leq& \\
\limsup_{n\to\infty}(\bar{v}_{i_0,*} - \varphi_{n,i_0})(t_n,x_n) - (\bar{v}_{i_0,*} - \varphi_{n,i_0})(s_n,\xi_n) -  \frac{1}{2}|x_n-x_0|^2
&\leq& \\
-  \frac{1}{2}|x^*-x_0|^2\;.
\enqs
Since $x_0$ minimizes the difference $\bar{v}_{i_0,*}(T, \cdot) - \varphi_{i_0}$ we have
\beqs
0 & \leq & (\bar{v}_{i_0,*}(T, \cdot) - \varphi_{i_0})(x^*) - (\bar{v}_{i_0,*}(T, \cdot)  - \varphi_{i_0})(x_0) ~\leq~ -  \frac{1}{2}|x^*-x_0|^2\;.
\enqs
Hence $x^*=x_0$ and $(x_n)_{n\geq 1}$ converges to $x_0$.

We now use the viscosity supersolution property of  $\bar{v}$ on $[0, T)\times \R^d\times\Ic$ with the test function $\tilde \varphi_{n}=\varphi_n+\bar{v}_{*,i_0}(t_n,x_n)-\varphi_{n,i_0}(t_n,x_n)$ and we have
\beq \label{eq:terminal_cond:H_geq_0}
\tilde \varphi_{n,i_0}(t_n,x_n) -
\sup_{0\leq k<\bar{K}}\tilde \varphi_{n,i_0k}(t_n,x_n) & \geq & 0
\enq
for all $n\geq 1$. We clearly have
\beqs
\varphi_{i_0}(x_n)-
\sup_{0\leq k<\bar{K}}\varphi_{i_0k}(x_n)=
\tilde \varphi_{n,i_0}(t_n,x_n)-
\sup_{0\leq k<\bar{K}}\tilde \varphi_{n,i_0k}(t_n,x_n)~.
\enqs
Since $x_n$ converges to $x_0$, we get  by  sending $n$ to infinity that $\varphi_{i_0}(x_0)-
\sup_{0\leq k<\bar{K}}\varphi_{i_0k}(x_0)\geq  0$.

\paragraph{Step 3.}
We now prove the last assertion. Assume that
\beqs
F^*\varphi_{i_0}(x_0) < \infty ~\text{ and }~
\varphi_{i_0}(x_0) = \bar v_{i_0,*}(T,x_0) < g_{i_0}
\enqs
and let us work towards a contradiction. Since $\bar v_\cdot(T,\cdot) = g_\cdot$ by the definition of the problem, there is a constant $\eta>0$ such that
\beqs
\varphi_{i_0} - \bar v_{i_0}(T,\cdot) &= & \varphi_{i_0} - g_{i_0} ~\leq~ - \eta~\text{ on }B_\eps(x_0)
\enqs
for some $\eps > 0$. Since $x_0$ is a strict minimizer, let $\zeta$ be
\beqs
2\zeta = \min_{ x\in\partial B_\eps(x_0)} \bar v_{i_0,*}(T,x) - \varphi_{i_0}(x) > 0~.
\enqs
It follows that there exists $r > 0$ such that $\bar v_{i_0}(t,x) - \varphi_{i_0}(x)\geq \zeta > 0$ for all $(t, x) \in [T-r, T ]\times \partial B_\eps(x_0)$. This holds, otherwise, for each $r>0$, we could find $(t_r, x_r) \in [T-r,T] \times \partial B_\eps(x_0)$ such that $\bar v_{i_0}(t_r, x_r) - \varphi_{i_0}(x_r) \leq \zeta$. Sending $r$ to $0$, since $\partial B_\eps(x_0)$ is compact, up to a subsequence we would have
$\bar v_{i_0,*}(T,x^*) - \varphi_{i_0}(x^*)\leq \zeta$ for some $x^*\in\partial B_\eps (x_0)$, in contradiction with the definition of $\zeta$.

Therefore, we have
\beqs
\bar v_{i_0}(t,x) - \varphi_{i_0}(x)\geq \zeta\wedge\eta > 0~\text{ for }(t,x)\in
\left([T-r,T] \times \partial B_\eps(x_0)\right)\cup \left(\{T\} \times B_\eps(x_0)\right)
~.
\enqs
Since $F^* \varphi_{i_0}(x_0) < \infty$, up to smaller $\eps > 0$, we have
\beqs
\lambda_Y(x,y,a) - L^a\varphi_{i_0}(x) \leq C \quad\text{ for all }a \in\Nc_\eps(x,D  \varphi_{i_0}(x))\\
\text{and } (x,y)\in \R^d\times\R \text{ such that } x\in B_\varepsilon(x_0)\text{ and } |y-\varphi_{i_0}(x)|\leq \varepsilon~.
\enqs
for some constant $C > 0$. Let $\tilde{\varphi}_{i}(t, x) := \varphi_{i}(x) + 2C(t-T)$. Then, for sufficiently small $r>0$,
\beqs
\bar v_{i_0}(t,x) - \tilde\varphi_{i_0}(t,x)\geq \frac{1}{2}(\zeta\wedge\eta) & >& 0
\enqs
for $(t,x)\in
\left([T-r,T] \times \partial B_\eps(x_0)\right)\cup \left(\{T\} \times B_\eps(x_0)\right)$,  and
\beqs
 -\partial_t\tilde \varphi_{i_0}(t,x)+\lambda_Y(x,y,a) - L^a\tilde\varphi_{i_0}(t,x) & \leq & -C
\enqs
 for all $a \in\Nc_\eps(x,D \tilde \varphi_{i_0}(t,x))$ and $(x,y)\in \R^d\times\R$ such that $x\in B_\varepsilon(x_0)$ and $|y-\tilde \varphi_{i_0}(t,x)|\leq \varepsilon$.
By following the same  arguments as in Step 3 of Section \ref{supersol_section}, the latter inequalities lead to a contradiction
of \eqref{eq:DPP1}.

\subsubsection{Viscosity subsolution}

Let $(x_0,i_0)\in\R^d\times\Ic$ and $\varphi_{i}\in C^2(\R^d)$ for $i\in\Ic$ satisfying
\beqs
0 ~=~ \bar{v}_{i_0}^*(T,x_0) - \varphi_{i_0} (x_0)  & = &  \max_{\Ic\times\R^d} (\bar{v}_{\cdot}^*(T,\cdot) - \varphi_{\cdot})\;.
\enqs
Without loss of generalities we can take this maximum to be strict in $x$ once have fixed $i_0$.
We argue by contradiction and assume $\delta_*\varphi_{i_0}(x_0) >0$ and
\beq\label{subsol_terminal:contradiction}
4 \eta ~=~ \min\left\{
\varphi_{i_0}(x_0) - g_{i_0}(x_0)
~;~
\left(\varphi_{i_0} - \sup_{0\leq k<\bar{K}} \varphi_{i_0k}\right)(x_0)\right\}  > 0~\;.
\enq

\paragraph{Step 1.} By \eqref{eq:delta>0} and Assumption \ref{Assumption_H_2}, we can find $r > 0$ and a locally Lipschitz map $\hat{a}$ satisfying
\beq\label{subsol_terminal:control}
\hat{a}(x,D  \varphi_{i_0}(x))\in\Nc_0(x, D  \varphi_{i_0}(x))
\enq
 for all $x\in B_r(x_0)$. Set $\tilde\varphi_{i}(t,x):=\varphi_{i}(x)+\sqrt{T-t}$. Since $\partial\tilde\varphi_{i}(t,x)\to-\infty$ as $t\to T$, we deduce that, for $r,\eps > 0$ small
enough,
\beq \label{subsol_terminal:infinitesimal}
\rho(t,x,y) =- \partial\tilde\varphi_{i_0}(t,x) +\lambda_Y\left(x,y,
\hat{a}\left(x,D  \tilde\varphi_{i_0}(t,x)\right)\right) -
L^{\hat{a}\left(x,D \tilde \varphi_{i_0}(t,x)\right)}\tilde\varphi_{i_0}(t,x) \geq \eta ~,
\enq
for all $(t,x,y)\in[T-r,T)\times\R^d\times\R$ such that $x\in B_r(x_0)$ and $|y-\tilde\varphi_{i_0}(t,x)|\leq \eps$. We can also notice that
\beqs
\left(\tilde\varphi_{i_0} - \sup_{0\leq k<\bar{K}} \tilde\varphi_{i_0k}\right)(t,x_0) = \left(\varphi_{i_0} - \sup_{0\leq k<\bar{K}} \varphi_{i_0k}\right)(x_0)~.
\enqs
Therefore,  we get from \eqref{subsol_terminal:contradiction}
\beq\label{subsol_terminal:jump}
\left(\tilde\varphi_{i_0} - \sup_{0\leq k<\bar{K}} \tilde\varphi_{i_0k}\right)(t,x)\geq \eta ~,~\text{ for all }(t,x)\in[T-r,T]\times B_r (x_0)~.
\enq
for $r>0$ small enough.

Also observe that, since $\bar{v}_{i_0}^*- \tilde\varphi_{i_0}$ is upper-semicontinuous and $\left(\bar{v}_{i_0}^*- \tilde\varphi_{i_0}\right)(T,x_0)=0$, we have
\beq\label{subsol_terminal:eps_over_2}
\bar{v}_{i_0}^*(t,x) \leq \tilde\varphi_{i_0}(t,x) + \eps/2 ~\text{ for all }(t,x)\in[T-r,T]\times B_r (x_0)~.
\enq
for $r>0$ small enough.
Since $\bar v_{\cdot}(T,\cdot) = g_{\cdot}$, we have for $r$ small enough
\beqs
\tilde\varphi_{i_0} - \bar v_{i_0}(T,\cdot) = \tilde\varphi_{i_0} - g_{i_0} \geq  \eta~\text{ on }B_r(x_0)~.
\enqs
Since $x_0$ is a strict maximizer for $v_{i_0,*}(T,\cdot) - \varphi_{i_0}$, we can define  $\zeta>0$ such that
\beqs
-2\zeta = \max_{ x\in\partial B_r(x_0)} \bar v_{i_0}^*(T,x) - \varphi_{i_0}(x) < 0~.
\enqs
for $r>0$ small enough. It follows that, for $r > 0$ small enough, $\bar v_{i_0}(t,x) - \tilde\varphi_{i_0}(x)\leq -\zeta < 0$ for all $(t, x) \in [T-r, T ]\times \partial B_r(x_0)$. This means
\beq\label{subsol_terminal:boundary_cond}
\bar v_{i_0}(t,x) - \tilde\varphi_{i_0}(x)\leq -\zeta \wedge \eta~ \text{ for all }
(t, x) \in \left([T-r',T]\times \partial B_r(x_0)\right)\cup\left(\{T\}\times B_r(x_0)
\right)\;.
\enq
By following the arguments in Step 2 of Section \ref{subsubsection:visc_subsol}, we see that \eqref{subsol_terminal:control}, \eqref{subsol_terminal:infinitesimal},  \eqref{subsol_terminal:jump}, \eqref{subsol_terminal:eps_over_2}, \eqref{subsol_terminal:boundary_cond},
lead to a contradiction of \eqref{eq:DPP1}.

\subsection{Uniqueness}

We turn to the uniqueness of solution to the dynamic programming equation \eqref{PDE1}-\eqref{PDE2}. To this end, we need to introduce additional assumptions.
We first recall that the Hausdorff distance $d_{\mathcal{H}}$ on closed subsets of $A$ is defined by
\beqs
d_{\mathcal{H}}(B,C) & = & \min\left\{r\geq0~:~B\subset C_r \mbox{ and }C\subset B_r\right\}
\enqs
for $B,C\subset A$ closed and nonempty, with
\beq\label{convHdist}
D_r & = & \left\{a\in A~:~\exists a'\in D\;,~d_A(a,a')\leq r \right\}
\enq
for any $D\subset A$ and any $r\geq 0$. We use the convention
\beqs
d_{\mathcal{H}}(B,C) & = & +\infty
\enqs
if $B=\emptyset$ or $C=\emptyset$.

\begin{Assumption}\label{Assumption_H_3}
\begin{enumerate}[(i)]
    \item The the functions $\lambda$ and $\sigma$ do not depend on the control, i.e. $\lambda:\R^d \to \R^d$ and $\sigma:\R^d \to \R^{d \times m}$.

    \item There exist two constants $C>0$ and $\eta \in(0,1]$ such that the function $w$ appearing in Assumption \ref{Assumption_H_0}(iii) satisfies $w(x)\leq C x^\eta$ for $x\in\R_+$.
    \item There exists a constant $C>0$ such that
    \beqs
    d_\Hc\left(~\Nc_\eps(x,p)~,~\Nc_{\eps'}(x',p')~\right)\leq C\left(|p-p'| +\eps+\eps'\right)(1+|x|) + C|x-x'|
    \enqs
    for all $\eps,\eps'\geq0$, $x,x',p,p'\in\R^d$.

    \item $\Nc_0(0,0)\neq \emptyset$.
    \end{enumerate}
\end{Assumption}

\begin{Remark}\label{rem-kernel-nonempty}
Since we use the convention \eqref{convHdist}, the combination of (iii) and (iv) implies that $\Nc_\eps(x,p)\neq\emptyset$ for any $(\eps, x,p)\in\R_+\times\R^d\times\R^d$.

In particular we always have that $\delta\varphi \geq 0$ for any $\varphi\in C^2(\R^d)$. Therefore, the terminal viscosity supersolution solution property takes the following form
\beq\label{PDE2:uniqueness}
\min \left\{~\varphi_i(x)-g_i(x)~;~ \left(\varphi_{i} - \sup_{0\leq k<\bar{K}}\varphi_{ik}\right)(x)~\right\}& \geq  & 0
\enq
for $(x,i)\in\R^d\times\Ic$ and $(\varphi_j)_{j\in\Ic}$ a test function according to Definition \ref{def-viscosolT} (i).
\end{Remark}

\begin{Lemma}\label{lemma:uniqueness:supersol}
let $u:~[0,T]\times\R^d\times\Ic$ be a lower semi-continuous supersolution of  \eqref{PDE1}-\eqref{PDE2:uniqueness}. Define the function $\Lambda:~[0,T]\times\R^d\rightarrow\R$ by
\beqs
\Lambda(t,x) = \theta e^{- \kappa t} (1+|x|^{2 \gamma +2})\;,\quad (t,x)\in[0,T]\times\R^d\;.
\enqs
with $\theta,\kappa,\gamma\in\R_+$. Then, under Assumptions \ref{Assumption_H_0} and \ref{Assumption_H_3}, for any $\gamma\geq 0$ there exists $\kappa_0 > 0$
such that for any $\kappa \geq \kappa_0$ and $\theta>0$, the function $u+\Lambda$ is a supersolution to \eqref{PDE1}-\eqref{PDE2}.
\end{Lemma}
\begin{proof}
Let $\varphi_j\in C^{1,2}([0,T]\times\R^d)$ for $j\in\Ic$ be such that the function $\varphi_\cdot -(u + \Lambda)$ has a local maximum in $(t,x,i)$ which is equal to $0$ and $\bar{\varphi}\in C^0([0,T]\times \mathbb{R}^{d})$ such that $\sup_{j\in\Ic}|\varphi_j|\leq \bar{\varphi}$. Since $u$ is a supersolution for \eqref{PDE1}, we have
\beqs
\min\left\{-\partial_t (\varphi_{i} -\Lambda)(t,x) + F^*(\varphi_{i} -\Lambda)(t,x)~;~
\left( (\varphi_{i} -\Lambda) - \sup_{0\leq k<\bar{K}}(\varphi_{ik} -\Lambda)\right)(t,x)
\right\} & \geq & 0\;.
\enqs
We then have
\beq
\left(\varphi_{i} - \sup_{0\leq k<\bar{K}}\varphi_{ik}\right)(t,x) & = &
\left( (\varphi_{i} -\Lambda) - \sup_{0\leq k<\bar{K}}(\varphi_{ik} -\Lambda)\right)(t,x) ~~\geq~~ 0~. \label{uniqueness:supersolution:G_v_eq}
\enq
We now  prove that
\beqs
-\partial_t \varphi_{i}(t,x) + F^*\varphi_{i}(t,x) & \geq & 0\;.
\enqs
If $F^*\varphi_{i}(t,x)=+\infty$, then the inequality is obvious.  Suppose that $F^*\varphi_{i}(t,x)<+\infty$. From Assumption \ref{Assumption_H_3}, we get that $F^*$ is locally bounded.
Since $u$ is a viscosity supersolution to \eqref{PDE1}, we have
\beqs
-\partial_t (\varphi_{i} -\Lambda)(t,x) + F^*(\varphi_{i} -\Lambda)(t,x) & \geq & 0\;.
\enqs
Using the definition of $\Lambda$ and $F$, Assumption \ref{Assumption_H_3} and the continuity of the functions considered, we get
\beq
-\partial_t \varphi_{i}(t,x)- \theta \kappa e^{- \kappa t} (1+|x|^{2 \gamma +2}) & & \nonumber \\
+\lim_{\eps\to 0} \sup_{\begin{tiny}\begin{array}{c}|x-x'|\leq\eps\\|(\varphi_{i}-\Lambda)(t,x)-y'|\leq\eps\\|D(\varphi_{i}-\Lambda)(t,x)-p|\leq\eps\end{array}\end{tiny}}\sup_{a \in \Nc_\eps(x',p)}
\left\{\lambda_{Y}(x',y' ,a)\right\} & &  \nonumber\\
-\lambda(x)^\top D  \varphi_{i}(t,x)+\theta e^{- \kappa t} \lambda(x)^\top D  |x|^{2 \gamma +2}  & & \nonumber\\
- \frac{1}{2}\text{Tr}\left(\sigma\sigma^\top(x)D ^2\varphi_{i}(t,x)\right)
+\theta e^{- \kappa t} \frac{1}{2}\text{Tr}\left(\sigma\sigma^\top(x)D ^2|x|^{2 \gamma +2}\right) & \geq & 0~. \label{uniqueness:supersol:2nd_ineq}
\enq
We next define the function $\Gamma_\eps: \R^d\times\R\times\R^d\to \R$ by $\Gamma_\eps('x,y',p)=\sup_{a \in \Nc_\eps(x',p)} \left\{\lambda_{Y}(x',y',a)\right\}$ for $(x',y',p)\in\R^d\times\R\times\R^d$.
Then,  we get from \eqref{uniqueness:supersol:2nd_ineq}
 \beq
-\partial_t \varphi_{i}(t,x)+\lim_{\eps\to 0} \sup_{\begin{tiny}\begin{array}{c}|x-x'|\leq\eps\\|\varphi_{i}(t,x)-y'|\leq\eps\\|D\varphi_{i}(t,x)-p|\leq\eps\end{array}\end{tiny}}\Gamma_\eps(x',y',p) & & \nonumber \\
-\lambda(x)^\top D  \varphi_{i}(t,x) - \frac{1}{2}\text{Tr}\left(\sigma\sigma^\top(x)D ^2\varphi_{i}(t,x)\right) & \geq &
\nonumber\\
\theta \kappa e^{- \kappa t} (1+|x|^{2 \gamma +2})+
\lim_{\eps\to 0} \sup_{\begin{tiny}\begin{array}{c}|x-x'|\leq\eps\\|\varphi_{i}(t,x)-y'|\leq\eps\\|D\varphi_{i}(t,x)-p|\leq\eps\end{array}\end{tiny}}\Gamma_\eps(x',y',p) & &
\nonumber\\
-\lim_{\eps\to 0} \sup_{\begin{tiny}\begin{array}{c}|x-x'|\leq\eps\\|(\varphi_{i}-\Lambda)(t,x)-y'|\leq\eps\\|D(\varphi_{i}-\Lambda)(t,x)-p|\leq\eps\end{array}\end{tiny}}\Gamma_\eps(x',y',p)  & &
\nonumber\\
-\theta e^{- \kappa t} \lambda(x)^\top D  |x|^{2 \gamma +2}-\theta e^{- \kappa t} \frac{1}{2}\text{Tr}\left(\sigma\sigma^\top(x)D ^2|x|^{2 \gamma +2}\right) &  = &
\nonumber\\
\theta \kappa e^{- \kappa t} (1+|x|^{2 \gamma +2})-\theta e^{- \kappa t} \lambda(x)^\top D  |x|^{2 \gamma +2}  & & \nonumber\\
- \theta e^{- \kappa t} \frac{1}{2}\text{Tr}\left(\sigma\sigma^\top(x)D ^2|x|^{2 \gamma +2}\right)+\Delta \Gamma^1(t,x) + \Delta \Gamma^2(t,x)
~, \label{uniqueness:supersolution:last_inequality}
\enq
where
\beqs
\Delta \Gamma^1(t,x) &=& \lim_{\eps\to 0} \sup_{\begin{tiny}\begin{array}{c}|x-x'|\leq\eps\\|\varphi_{i}(t,x)-y'|\leq\eps\\|D\varphi_{i}(t,x)-p|\leq\eps\end{array}\end{tiny}}\Gamma_\eps(x',y',p)-\lim_{\eps\to 0} \sup_{\begin{tiny}\begin{array}{c}|x-x'|\leq\eps\\|\varphi_{i}(t,x)-y'|\leq\eps\\
|D(\varphi_{i}-\Lambda)(t,x)-p|\leq\eps\end{array}\end{tiny}}\Gamma_\eps(x',y',p)\;,\\
\Delta \Gamma^2(t,x) &=& \lim_{\eps\to 0} \sup_{\begin{tiny}\begin{array}{c}|x-x'|\leq\eps\\|\varphi_{i}(t,x)-y'|\leq\eps\\
|D(\varphi_{i}-\Lambda)(t,x)-p|\leq\eps\end{array}\end{tiny}}\Gamma_\eps(x',y',p) -
\lim_{\eps\to 0} \sup_{\begin{tiny}\begin{array}{c}|x-x'|\leq\eps\\|(\varphi_{i}-\Lambda)(t,x)-y'|\leq\eps\\
|D(\varphi_{i}-\Lambda)(t,x)-p|\leq\eps\end{array}\end{tiny}}\Gamma_\eps(x',y',p)
~.
\enqs
From Assumptions \ref{Assumption_H_0} and \ref{Assumption_H_3}, we get a constant $C_1$ that does not depend on $(t,x,i)$ such that
\beqs
\Delta \Gamma^1(t,x) &\geq&-\lim_{\eps\to0}\sup_{\begin{tiny}\begin{array}{cc}|x-x'|\leq\eps & |x-\tilde x'|\leq\eps\\|\varphi_{i}(t,x)-y'|\leq\eps & |\varphi_{i}(t,x)-\tilde y'|\leq\eps\\|D\varphi_{i}(t,x)-p|\leq\eps & |D(\varphi_{i}-\Lambda)(t,x)-\tilde p|\leq\eps\end{array}\end{tiny}}\Gamma_\eps(x',y',p)-\Gamma_\eps(\tilde x',\tilde y',\tilde p)
\\
&\geq& - C_1|D \Lambda(t,x)|^\eta(1+|x|^\eta)~,
\enqs
Analogously, for the second term we get a constant $C_2>0$ that does not depend on $(t,x,i)$ such that
\beqs
\Delta \Gamma^2(t,x) &\geq&
-C_2 \Lambda(t,x)\;.
\enqs

Considering the right-hand side of \eqref{uniqueness:supersolution:last_inequality} and taking into account the growth condition of the different terms, there exists a constant $\kappa_0$, which does not depend on $\theta$, such that if $\kappa \geq \kappa_0$ this expression is non-negative. Henceforth, with \eqref{uniqueness:supersolution:G_v_eq}, we obtain that $u+\Lambda$ is a viscosity supersolution to \eqref{PDE1}.

We finally take $(i,x)\in \Ic\times \R^d$ and functions $\varphi_j\in C^2(\R^d)$  and $\bar \phi\in C^0(\R^d)$ such that $\sup_{i\in\Ic}|\varphi_i| \leq \bar{\varphi}$ and
\beqs
0~=~u_{i_0,*}(T,x)+\Lambda(T,x)- \varphi_{i}(x) & = & \max_{\Ic\times\mathbb{R}^{d}}(u_{\cdot,*}(T,.)+\Lambda(T,.)-\varphi_\cdot)\;.
\enqs
Since $u$ is a supersolution to \eqref{PDE2}, we have
\beqs
\varphi_i(x)-\Lambda(T,x) & \geq & g_i(x)~,
\enqs
since $\Lambda \geq0$, we get $\varphi_i(T,x)  \geq g_i(x)$. Combining it with \eqref{uniqueness:supersolution:G_v_eq}, we obtain from Remark \ref{rem-kernel-nonempty} that $u + \Lambda $ is a viscosity supersolution to \eqref{PDE2}.
\end{proof}

%%%%%%%%%%%%%%%%%%%%%%%%%%%%%%%%%%%%%%%%%%%%%%%%%%%%%%%%%%%%%
We turn to the main result of this section which is a comparison theorem.
We recall that the definition of $|.|$ on $\Ic$ is given in Section \ref{subsecdef}.

\begin{Theorem}\label{Thm:comparison}
Let $\bar{w}_\cdot$ (resp. $\bar{u}_\cdot$) be a lsc (resp. usc) viscosity supersolution (resp. subsolution) to \eqref{PDE1}-\eqref{PDE2:uniqueness}.
Suppose that there exists $\gamma>0$ such that
\beq\label{thm:growthinx}
\sup_{(t,x,i)\in [0,T] \times \R^d\times \Ic} \frac{|\bar{w}_i(t,x)|+|\bar{u}_i(t,x)|}{1+|x|^\gamma}< +\infty~ \;,
\enq
and
\beq\label{uniqueness:hyp:i_to_infty}
\sup_{(t,x)\in [0,T] \times \R^d} |\bar{w}_i(t,x)|+|\bar{u}_i(t,x)| \xrightarrow[|i|\to \infty]{} 0\;.
\enq
Then, under Assumption \ref{Assumption_H_0}-\ref{Assumption_H_1}-\ref{Assumption_H_3}, we have  $\bar{u}_\cdot \leq \bar{w}_{\cdot}$ on $[0,T] \times \R^d\times \Ic$.
\end{Theorem}
%%%%%%%%%%%%%%%%%%%%%%%%%%%%%%%%%%%%%%%%%%%%%%%%%%%%%%%%%%%%%
\begin{proof} We proceed in six steps.

\paragraph{Step 1.} We define $\Lambda_{\theta,\kappa}(t,x) = \theta e^{- \kappa t} (1+|x|^{2 \gamma +2})$ for $(t,x)\in[0,T]\times\R^d$   with $\theta,\kappa\in\R_+$. From Lemma \ref{lemma:uniqueness:supersol}, there exist $\kappa$ large enough such that for any $\theta> 0$, $\bar w_\cdot+\Lambda_{\theta,\kappa}$ is also a supersolution for \eqref{PDE1}-\eqref{PDE2}. Set $\bar w_{i,\theta,\kappa}(t,x)=\bar w_i(t,x)+\Lambda_{\theta,\kappa}(t,x)$, $(i,t,x)\in\Ic\times[0,T]\times\R^d$.

For some $\eta,~\eta'>0$ to be chosen below, let $\beta_t = e^{(\eta+\eta')t}$ for $t\in[0,T]$. A straightforward derivation shows that $\beta_t \bar{w}_{i,\theta,\kappa}$ (resp. $\beta_t \bar{u}_i$) is a viscosity supersolution (resp. subsolution) to
\beq
\min\Bigg\{
\eta w_i -\partial_t w_i +
\tilde{F}(t,x,w_i,D  w_i) - \lambda^\top D  w_i - \frac{1}{2}\text{Tr}\left(\sigma\sigma^\top D ^2 w_i\right)
~;&&\nonumber\\
w_{i}-\sup_{0\leq k<\bar{K}}w_{ik}\Bigg\} &=&0 \mbox{ on } [0,T]\times\R^d,~~\label{uniqueness:new_PDE1}\\
\min\left\{ w_i  -\tilde{g}~;\delta w_{i}~;~ w_{i}-\sup_{0\leq k<\bar{K}}w_{ik} \right\}
&=&0\mbox{ on }\{T\}\times \R^d\;.\label{uniqueness:new_PDE2}
\enq
where
\beqs
\tilde{F}(t,x,y,p) = \sup_{a \in \tilde{\Nc_0}(t,x,p)} \tilde{\lambda}_{Y}(x,y,a) &,&\tilde{\Nc_0}(t,x,p)=\Nc_0(x,\beta_t^{-1}p)~,
\\
\tilde{\lambda}_{Y}(t,x,y,a) = \beta_t \lambda_{Y}(x,\beta_t^{-1}y,a)+\eta' y~&,& \tilde{g}_i(x)=\beta_T g_i(x) ~,
\\
\enqs
for all $(t,x,i,y,p,a)\in [0,T] \times \R^d\times \Ic \times\R\times\R^d\times A$. Since $\lambda_{Y}$ is Lipschitz, we can choose $\eta'$ large enough so that  $\tilde{\lambda}_{Y}$ and, consequently, $\tilde{F}$ are nondecreasing in $y$.

Let $\eps>0$. From an analogous computation, using the monotonicity of $\tilde F$, we see that $\beta_t \bar{w}_i+\eps/2^{|i|}$ is a viscosity supersolution to
\beq \label{uniqueness:new_PDE1_w}
\eta w_i -\partial_t w_i +
\tilde{F}(t,x,w_i,D  w_i) - \lambda(x)^\top D  w_i - \frac{1}{2}\text{Tr}\left(\sigma\sigma(x)^\top D ^2 w_i\right) &\geq& 0~,\\
\min\left\{ w_i (T,\cdot) -\tilde{g}~;~\delta w_{i}\right\}
&\geq&0 ~,\label{uniqueness:new_PDE2_w}\\
w_{i}-\sup_{0\leq k<\bar{K}}w_{ik} \geq \frac{\eps}{2^{|i|+1}}
%\frac{\eps}{2^{|i|}}-\frac{\eps}{2^{|i_0|}}
=: \Delta_i &>&0~.\label{uniqueness:new_PDE_jump_w}
\enq

\paragraph{Step 2.} Set $\tilde{u}_i = \beta_t \bar{u}_i$ and $\tilde{w}_{i,\theta,\kappa,\eps} = \beta_t \bar{v}_i + \beta_t \Lambda_{\theta,\kappa} + \eps/2^{|i|}=\beta_t \bar{w}_{i,\theta,\kappa} + \eps/2^{|i|}$ .
To prove our result, it is enough to show that
\beqs
\tilde{u}_i(t,x) & \leq &  \tilde{w}_{i,\theta,\kappa,\eps}(t,x)
\enqs
for each $(i,t,x)\in\Ic\times[0,T]\times\R^d$ and $\theta,\eps>0$. Then taking the limit as $\theta\to0$ and $\varepsilon\to0$, we obtain the desired result. For simplicity, we write $\tilde{w}_{i,\theta,\kappa,\eps}$ for $\tilde w_{i}$ in the sequel.
We argue by contradiction and suppose that
\beq\label{uniqueness:u_minus_v_positive}
\sup_{\Ic\times[0,T]\times\R^d}\tilde u_{\cdot} -\tilde w_\cdot >0~.
\enq
Due to the growth condition on $\tilde{u}_\cdot$ and $\tilde{w}_\cdot$, there exist $R > 0$ such that
\beq\label{uniqueness:negative_norm_geq_R}
 \tilde{u}_{i}(t, x) - \tilde{w}_i(t, x)  & <  & 0
\enq
for all $(i,t,x)\in \Ic \times [0,T] \times \R^d$ such that $|x| \geq R$. Then from \eqref{uniqueness:hyp:i_to_infty} and since $u_{\cdot} -\tilde w_\cdot$ is upper semicontinuous, there exist $(i_0,t_0,x_0)\in \Ic\times[0,T] \times \R^d$ such that
\beq \label{uniqueness:x0_max_point}
\sup_{(i,t,x)\in \Ic \times [0,T] \times \R^d}(\tilde{u}_i-\tilde{w}_{i})(t,x) &  =  & (\tilde{u}_{i_0}-\tilde{w}_{i_0})(t_0,x_0)~>~0\;.
\enq

\paragraph{Step 3.} For $n\geq 1$, we define the function
\beqs
\Theta_n(t, x, y,i) = \tilde{u}_{i}(t,x) -\tilde{w}_{i}(t,y)-\varphi_n(t,x,y,i)
\enqs
with
\beqs
\varphi_n(t,x,y,i)=n|x-y|^2+|x-x_0|^4+|t-t_0|^2+\1_{i\neq i_0}.
\enqs
for all $(t, x, y,i) \in [0, T ] \times\R^d \times\R^d\times\Ic$.
By the growth assumption on $\tilde{u}$ and $\tilde{v}$ and \eqref{uniqueness:hyp:i_to_infty}, for all $n$, there exists $(t_n, x_n, y_n, i_n)\in[0,T]\times \R^d\times \R^d\times \Ic$ attaining the maximum of $\Theta_n$ on $[0,T]\times \R^d\times \R^d\times \Ic$. We have
\beqs
\Theta_n(t_n, x_n, y_n,i_n) \geq \Theta_n(t_0,x_0,x_0,i_0)=(\tilde{u}_{i_0}-\tilde{w}_{i_0})(t_0,x_0)~.
\enqs
By \eqref{uniqueness:negative_norm_geq_R} and \eqref{uniqueness:hyp:i_to_infty}, up to a subsequence, $(t_n, x_n, y_n, i_n)$ converge to $(\hat{t}, \hat{x}, \hat{y}, \hat{i})$. Sending $n$ to infinity, this provides
\beqs
\bar \ell := \limsup_{n\to\infty} \varphi_n(t_n, x_n, y_n,i_n)
&\leq& \limsup_{n\to\infty}
[\tilde{u}_{i_n}(t_n,x_n) -\tilde{w}_{i_n}(t_n,y_n)-
(\tilde{u}_{i_0}-\tilde{w}_{i_0})(t_0,x_0)] \\
&\leq&\tilde{u}_{\hat{i}}(\hat{t},\hat{x}) -\tilde{w}_{\hat{i}}(\hat{t},\hat{y})-
(\tilde{u}_{i_0}-\tilde{w}_{i_0})(t_0,x_0)
~.
\enqs
In particular, $\bar \ell < +\infty$ and $\hat{x}=\hat{y}$. Using the definition of $(t_0,x_0,i_0)$ as a maximizer of
$\tilde{u}_{\cdot}-\tilde{w}_{\cdot}$
, we see that:
\beqs
0  & \leq & \bar \ell ~~\leq~~
(\tilde{u}_{\hat{i}}-\tilde{w}_{\hat{i}})(\hat{t},\hat{x}) - (\tilde{u}_{i_0}-\tilde{w}_{i_0})(t_0,x_0)
~~\leq~~ 0~,
\enqs
which implies
\beq
(t_n, x_n, y_n,i_n) &\to& (t_0,x_0,x_0,i_0)~, \label{uniqueness:conv_x_n_etc}
\\
n |x_n- y_n|^2&\to& 0~,\label{uniqueness:ordre_conv_xn_yn}\\
\tilde{u}_{i_n}(t_n,x_n) -\tilde{w}_{i_n}(t_n,y_n) &\to& (\tilde{u}_{i_0}-\tilde{w}_{i_0})(t_0,y_0)~.\label{uniqueness:conv_V_minus_U}
\enq
Being $\Ic$ endowed with the discrete topology, we can assume $i_n=i_0$ for all $n\geq 1$.

\paragraph{Step 4.} We now show that for $n$ large enough
\beq\label{uniqueness:2nd_eq_geq_0}
\tilde{u}_{i_0}(t_n,x_n) - \sup_{0\leq k<\bar{K}}\tilde{u}_{i_0k}(t_n,x_n) & > & 0~.
\enq
On the contrary, up to a subsequence, we would have for all $n$,
\beq \label{uniqueness:part3_1}
\tilde{u}_{i_0}(t_n,x_n) - \sup_{0\leq k<\bar{K}}\tilde{u}_{i_0k}(t_n,x_n)  & \leq &  0~.
\enq
Moreover, by the viscosity supersolution property of $\tilde{w}$ to \eqref{uniqueness:new_PDE_jump_w}, we have
\beqs
\tilde{w}_{i_0}(t_n,y_n) - \sup_{0\leq k<\bar{K}}\tilde{w}_{i_0k}(t_n,y_n)~~\geq~~ \Delta_{i_0} & > & 0~.
\enqs
We deduce from the two previous inequalities
\beq
\tilde{u}_{i_0}(t_n,x_n) - \sup_{0\leq k<\bar{K}}\tilde{u}_{i_0k}(t_n,x_n) &\leq& \tilde{w}_{i_0}(t_n,y_n) - \sup_{0\leq k<\bar{K}}\tilde{w}_{i_0k}(t_n,y_n) - \Delta_{i_0} \nonumber\\
\tilde{u}_{i_0}(t_n,x_n) -\tilde{w}_{i_0}(t_n,y_n) + \Delta_{i_0} &\leq & \sup_{0\leq k<\bar{K}}\tilde{u}_{i_0k}(t_n,x_n) - \sup_{0\leq k<\bar{K}}\tilde{w}_{i_0k}(t_n,y_n)\nonumber\\
&\leq & \sup_{0\leq k<\bar{K}}[\tilde{u}_{i_0k}(t_n,x_n) - \tilde{w}_{i_0k}(t_n,y_n)]\;. \label{uniqueness:part3:sup}
\enq
Since $\Delta_{i_0} >0$, for all $n$ there exists $k_n$ such that
\beqs
\sup_{0\leq k<\bar{K}}[\tilde{u}_{i_0k}(t_n,x_n) - \tilde{w}_{i_0k}(t_n,y_n)] - \frac{\Delta_{i_0}}{2} \leq \tilde{u}_{i_0k_n}(t_n,x_n) - \tilde{w}_{i_0k_n}(t_n,y_n)~.
\enqs
From \eqref{uniqueness:hyp:i_to_infty}, up to a subsequence, we may assume that $(k_n)$ converges to $k_0$ in $\N$. Hence, by sending $n$ to infinity into  \eqref{uniqueness:part3:sup}, it follows with \eqref{uniqueness:conv_V_minus_U} and the upper (resp. lower)-semicontinuity of $\tilde{u}$ (resp. $\tilde{w}$) that :
\beqs
(\tilde{u}_{i_0}-\tilde{w}_{i_0})(t_0,x_0) + \frac{\Delta_{i_0}}{2} \leq (\tilde{u}_{i_0 k_0}-\tilde{w}_{i_0 k_0})(t_0,x_0) ~,
\enqs
which is a contradiction to \eqref{uniqueness:x0_max_point}.

\paragraph{Step 5.} Let us check that, up to a subsequence, $t_n < T$ for all $n$. On the contrary, $t_n = t_0 = T$ for $n$ large enough, and from \eqref{uniqueness:2nd_eq_geq_0}, and the viscosity subsolution property of $\tilde{u}$ to \eqref{uniqueness:new_PDE2}, we would get
\beqs
\tilde{u}_{i_0}(T,x_n) \leq \tilde{g}_{i_0}(x_n)~.
\enqs
On the other hand, by the viscosity supersolution property of $\tilde{w}$ to \eqref{uniqueness:new_PDE2}, we have $\tilde{w}(T, y_n)\geq \tilde{g}_{i_0}(y_n)$, and so
\beqs
\tilde{u}_{i_0}(T,x_n) - \tilde{w}(T, y_n) \leq \tilde{g}_{i_0}(x_n) - \tilde{g}_{i_0}(y_n)~.
\enqs
By sending $n$ to infinity, and from Assumption \eqref{Assumption_H_1} and \eqref{uniqueness:conv_V_minus_U}, this would imply $\tilde{u}_{i_0}(t_0,x_0) - \tilde{v}(t_0, x_0) \leq 0$, a contradiction to \eqref{uniqueness:u_minus_v_positive}.

\paragraph{Step 6.} We may then apply Ishii's lemma (see Theorem 8.3 in \cite{crandall1992users}) to $(t_n, x_n, y_n)\in [0,T)\times\R^d\times\R^d$ that attains the maximum of $\Theta_n(.,{i_0})$ and we get $(p_{\tilde{u}}^n, q_{\tilde{u}}^n , M_n ) \in \bar J^{2,+}\tilde{u}_{i_0}(t_n, x_n)$ and $(p_{\tilde{U}}^n, q_{\tilde{v}}^n , N_n ) \in \bar J^{2,-}\tilde{v}_{i_0}(t_n, y_n)$ such that
\beqs
p_{\tilde{u}}^n - p_{\tilde{v}}^n &=& \partial_t \varphi_n(t_n, x_n, y_n,i_0) = 2(t_n-t_0)
~,\\
q_{\tilde{u}}^n &=&D_{x} \varphi_n(t_n, x_n, y_n,i_0) = n(x_n-y_n)+4(x_n-x_0)|x_n-x_0|^2~,\\
q_{\tilde{w}}^n &=& -D_{y} \varphi_n(t_n, x_n, y_n,i_0) = n(x_n-y_n)~,
\enqs
and
\beq\label{uniqueness:A_n}
\begin{pmatrix}
M_n & 0\\
0 & -N_n
\end{pmatrix} & \leq & A_n + \frac{1}{2n} A_n^2~,
\enq
where
\beqs
A_n = D_{(x,y)}^2 \varphi_n(t_n, x_n, y_n,i_0)= n \begin{pmatrix}
\I_{d} & -\I_{d}\\
 -\I_{d} & \I_{d}
\end{pmatrix} - \begin{pmatrix}
4|x_n-x_0|^2\I_{d} + 8 (x_n-x_0)(x_n-x_0)^\top& \mathds{O}_d\\
\mathds{O}_d & \mathds{O}_d
\end{pmatrix}~,
\enqs
with $\I_{d}$ and $\mathds{O}_d$  respectively  the identity and  the zero  matrix of $R^{d\times d}$. We can therefore bound the right-hand side of \eqref{uniqueness:A_n} by
\beq\label{uniqueness:bound_matrices}
A_n + \frac{1}{2n} A_n^2  & \leq &  3n\begin{pmatrix}
\I_{d} & -\I_{d}\\
 -\I_{d} & \I_{d}
\end{pmatrix} + \begin{pmatrix}
A_n'& \mathds{O}_d\\
\mathds{O}_d & \mathds{O}_d
\end{pmatrix}~,
\enq
with $A_n'$ such that $\limsup_{n\to\infty}\frac{1}{|x_n-x_0|^2}|A_n'|<+\infty$.
From the viscosity supersolution property of $\tilde{w}_{i_0}$ to \eqref{uniqueness:new_PDE1}, we have
\beqs
\eta \tilde{w}_{i_0}(t_n, y_n) - p_{\tilde{w}}^n +
\tilde{F}^*(t_n, y_n,\tilde{w}_{i_0}(t_n, y_n),q_{\tilde{w}}^n )
- \lambda(y_n) ^\top q_{\tilde{w}}^n -\frac{1}{2} \textrm{Tr}(\sigma\sigma^\top(y_n)N_n)\geq 0~.
\enqs
On the other hand, from \eqref{uniqueness:2nd_eq_geq_0} and the viscosity subsolution property of $\tilde{u}$ to \eqref{uniqueness:new_PDE1}, we have
\beqs
\eta \tilde{u}_{i_0}(t_n, x_n) - p_{\tilde{u}}^n +
\tilde{F}_*(t_n, x_n,\tilde{u}_{i_0}(t_n, x_n), q_{\tilde{u}}^n )- \lambda(x_n)^\top q_{\tilde{u}}^n -\frac{1}{2} \textrm{Tr}(\sigma\sigma^\top(x_n)M_n) \leq 0~.
\enqs
By subtracting the two previous inequalities, we obtain
\beq
\eta( \tilde{u}_{i_0}(t_n, x_n) - \tilde{w}_{i_0}(t_n, y_n) ) &\leq&
p_{\tilde{u}}^n - p_{\tilde{w}}^n +\tilde{F}^*(t_n, y_n,\tilde{w}_{i_0}(t_n, y_n),q_{\tilde{w}}^n )-\tilde{F}_*(t_n, x_n,\tilde{u}_{i_0}(t_n, x_n), q_{\tilde{u}}^n )+\nonumber\\
&&
+\lambda(x_n)^\top q_{\tilde{u}}^n - \lambda(y_n) ^\top q_{\tilde{w}}^n
+\frac{1}{2} \textrm{Tr}(\sigma\sigma^\top(x_n)M_n)-\frac{1}{2} \textrm{Tr}(\sigma\sigma^\top(y_n)N_n)\nonumber\\
&=&p_{\tilde{u}}^n - p_{\tilde{u}}^n + \Delta C_n^1+\Delta C_n^2+\Delta C_n^3
\label{uniqueness:differential_contradiction}
\enq
where
\beqs
\Delta C_n^1 &=&\tilde{F}^*(t_n, y_n,\tilde{w}_{i_0}(t_n, y_n),q_{\tilde{w}}^n )-\tilde{F}_*(t_n, x_n,\tilde{u}_{i_0}(t_n, x_n), q_{\tilde{u}}^n ) ~,\\
\Delta C_n^2 &=& \lambda(x_n)^\top q_{\tilde{u}}^n - \lambda(y_n) ^\top q_{\tilde{w}}^n~,\\
\Delta C_n^3 &=& \frac{1}{2} \textrm{Tr}(\sigma\sigma^\top(x_n)M_n)-\frac{1}{2} \textrm{Tr}(\sigma\sigma^\top(y_n)N_n)~.
\enqs
From \eqref{uniqueness:conv_x_n_etc}),
we have $p_{\tilde{u}}^n - p_{\tilde{u}}^n \to 0$ as $n\to 0$. From the Lipschitz continuity of $\lambda$ and \eqref{uniqueness:ordre_conv_xn_yn}, we have $\Delta C^2 \to 0$ as $n\to 0$. From \eqref{uniqueness:bound_matrices}, \eqref{uniqueness:conv_x_n_etc}, \eqref{uniqueness:ordre_conv_xn_yn}, and the Lipschitz property of $\sigma$, we also have $\Delta C^3 \to 0$ as $n\to 0$. Following the same argument as in the proof of lemma \ref{lemma:uniqueness:supersol}, we get from Assumptions \ref{Assumption_H_0} and \ref{Assumption_H_3} and \eqref{uniqueness:conv_x_n_etc} $\Delta C^1_n \to 0$ as $n\to\infty$.

Therefore, by sending $n \to \infty$ into \eqref{uniqueness:differential_contradiction}, we conclude with \eqref{uniqueness:conv_V_minus_U} that $\eta( \tilde{u}_{i_0}(t_0, x_0) - \tilde{w}_{i_0}(t_0, y_0) ) \leq 0$, a contradiction with \eqref{uniqueness:x0_max_point}.
\end{proof}

From Theorems \ref{theorem:result_PDE},
\ref{theorem:result_PDET} and
\ref{Thm:comparison}, we get the following characterisation of the function $\bar v$.

\begin{Corollary} Suppose that $\bar v$
satisfies
\beq\label{uniqueness:hyp:growth}
\sup_{(t,x,i)\in [0,T] \times \R^d\times \Ic} \frac{|\bar{v}_i(t,x)|+|\bar{u}_i(t,x)|}{1+|x|^\gamma}< +\infty~ \;,
\enq
for some $\gamma>0$ and
\beq\label{uniqueness:hyp:i_to_infty2}
\sup_{(t,x)\in [0,T] \times \R^d} |\bar{v}_i(t,x)|+|\bar{u}_i(t,x)| \xrightarrow[|i|\to \infty]{} 0\;.
\enq
Under Assumptions \ref{Assumption_H_0}-\ref{Assumption_H_1}-\ref{Assumption_H_2}-\ref{Assumption_H_3}, $\bar{v}$ is the unique viscosity solution to \eqref{PDE1}-\eqref{PDE2:uniqueness} satisfying \eqref{uniqueness:hyp:growth}-\eqref{uniqueness:hyp:i_to_infty2}. Moreover, $\bar{v}$ is continuous on $[0,T)\times\R^d\times\Ic$.
\end{Corollary}
We recall that Section \ref{example} provides an example of a value function satisfying conditions \eqref{PDE1}-\eqref{PDE2:uniqueness}.
\appendix

\section{Appendix}
\subsection{Set of atomic finite measures}
\begin{Proposition}\label{PropEclosed}
For $\ell\geq 1$, $E_\ell $ is a closed subset of $\mathcal{M}_F(\Ic \times \mathbb{R}^\ell)$ for the topology of the weak convergence of measures.
\end{Proposition}
\begin{proof}
Let $(\mu_n)_{n \in \mathbb{N}}$ be a sequence of $E_\ell $ such that $\mu_n = \sum_{i \in V_n} \delta_{(i,x^i_n)} \overset{w} \rightarrow \mu \in \mathcal{M}_F(\Ic \times \mathbb{R}^\ell)$. We prove that $\mu$ is an element of $E_\ell $, i.e. it can be written as $\mu = \sum_{i \in V} \delta_{(i,x^i)}$ for some set $V\subseteq \Ic$, $|V| < \infty$ and some points $(x_i)_{i \in V}$.

 Consider the continuous functions $\1_{\{i\}\times \R^\ell}$ for $i \in \Ic$. We then have
\beqs
\langle \mu_n \ , \ \1_{\{i\}\times \mathbb{R}^\ell} \rangle =
    \begin{cases}
      1 \quad & \text{if } i \in V_n\\
      0 \quad & \text{if } i \notin V_n
    \end{cases}
\enqs
For each $ i \in \Ic$, we have that the sequence $(\langle \mu_n \ , \ \1_{\{i\}\times \mathbb{R}^\ell} \rangle)_n$ is a convergent sequence in $\{0,1\}$, which is in particular stationary. Let $V$ be defined as follow:
\beqs
V := \left\{ i \in \Ic \ : \ \langle \mu_n \ , \ \1_{\{i\}\times \mathbb{R}^\ell} \rangle \underset{n \rightarrow \infty}\longrightarrow 1\right\}\ .
\enqs

Let $i \in V$. Since the functions previously described converge, they are constant from a certain rank and there exists $n_i \in \mathbb{N}$ such that for $n \geq n_i$ we have $i \in V_n$. For $f \in C(\mathbb{R}^\ell)$ and consider the function $\1_{\{i\}} \otimes f : \Ic \times \mathbb{R}^\ell \rightarrow \mathbb{R}$. We have:
\beqs
f(x^i_n) = \langle \mu_n , \1_{\{i\}}\otimes f \rangle \longrightarrow \langle \mu , \1_{\{i\}}\otimes f \rangle \in \mathbb{R} .
\enqs
This means that for each $i \in V$ and $f \in C(\mathbb{R}^\ell)$ the sequence $(f(x^i_n))_n$ converges, therefore $(x^i_n)_n$ converges to a point $x^i \in \mathbb{R}^\ell$.

We then notice that any continuous and bounded function $f$  on $\Ic\times\R^\ell$ is of the form $f=\sum_{i\in\Ic}f_i$ with $f_i$ is continuous and bounded on $\R^\ell$. In particular, we get
\beqs
\int_{\Ic\times\R^\ell}fd\mu_n & = & \sum_{i\in V}f_i(x^i_n)
\enqs
for $n$ large enough, and
\beqs
\int_{\Ic\times\R^\ell}fd\mu_n & \xrightarrow[n\rightarrow+\infty]{} & \int_{\Ic\times\R^\ell}fd\left(\sum_{i \in V} \delta_{(i,x^i)}\right)
\enqs
so we have $\mu = \sum_{i \in V} \delta_{(i,x^i)}$.

To finally prove that $\mu \in E_\ell$ we need to show that there do not exist $ \ i,j \in V$ such that $i \prec j$. Fix $i,j\in V$. From the previous steps, there exists some n such that $i,j\in V_n$. Since $\mu_n\in E_\ell$, we get $i \nprec j$ and  $j \nprec i$. Therefore, we have $\mu \in E_\ell$.
\end{proof}

\subsection{Branching martingale controlled problem}\label{sec-mart-cont-pb}

We first set our controlled martingale problem. We define the set $\tilde{E}_{m+1}$ as the set of finite measure $\mu$ on $\Ic\times\N\times\R^{m+1}$ of the form
\beqs
\mu & = & \sum_{i\in\Ic,n\in\N}\frac{1}{2^{2(|i|+n)}}\delta_{(i,n,b^i,q^{i,n})}
\enqs
with $b^i\in\R^m$ and $q^{i,n}\in\R$ for $i\in\Ic$ and $n\in\N$. From the same argument as for $E_\ell$, we have that $\tilde E _{m+1}$ is Polish.

We then set $\mathbf X=\D([0,T],E_{d+1})\times \D([0,T],\tilde{E}_{m+1})$  the space pairs of \textit{c\`adl\`ag} functions from $[0,T]$ to $E_{d+1}$ and $\tilde{E}_{m+1}$.
We denote by $\mathbf x$ the canonical process and by $\G=(\Gc_t)_{t\in[0,T]}$ the canonical filtration on $\mathbf X$.

For $\bar x =(\sum_{i\in \Vc_s}\delta_{(i,\hat x_s^i)},\sum_{i\in \Ic,n\in\N}\frac{1}{2^{2(|i|+n)}}\delta_{(i,n,b_s^i, q_s^{i,n})})_{s\in[0,T]}\in \mathbf{X}$, we write
\beq\label{def-projX}
{}^1\bar x ~ = ~ \left(\sum_{i\in \Vc_s}\delta_{(i,\hat x_s^i)}\right)_{s\in[0,T]} & \mbox{ and } & {}^2\bar x ~= ~ \left(\sum_{i\in \Ic,n\in\N}\frac{1}{2^{2(|i|+n)}}\delta_{(i,n,b_s^i, q_s^{i,n})}\right)_{s\in[0,T]}
\enq
We also define ${}^1\mathbf x$ and ${}^2\mathbf x$ the first and second component of the canonical process.

Let $C^{1,2}([0,T]\times \Ic\times \R^{d+1})$ (resp. $C^{1,2}([0,T]\times \Ic\times\N\times \R^{m+1})$) be the set of functions $f: ~[0,T]\times \Ic\times \R^{d+1}\rightarrow\R$ (resp. $g: ~[0,T]\times \Ic\times\N\times \R^{m+1}\rightarrow\R$) such that $f_i(\cdot)\in C^{1,2}([0,T]\times \R^{d+1})$ (resp. $g_{i,n}(\cdot)\in C^{1,2}([0,T]\times \R^{m+1})$) for all $i\in \Ic$ (resp. $(i,n)\in \Ic\times\N$) and $C_c^{1,2}([0,T]\times \Ic\times \R^{d+1})$ (resp. $C_c^{1,2}([0,T]\times \Ic\times\N\times \R^{m+1})$) the set of $f\in C^{1,2}([0,T]\times \Ic\times \R^{d+1})$ (resp. $g\in C^{1,2}([0,T]\times \Ic\times\N\times \R^{m+1})$) such that there exists a compact $K\subset \R^{d+1}$ (resp. $K\subset \R^{m+1}$) satisfying $f_i(t,x)=0$ (resp. $g_{i,n}(t,x)=0$) for $(t,i,x)\in[0,T]\times\Ic\times \R^{d+1}$ such that $x\notin K$ (resp. $(t,i,n,x)\in[0,T]\times\Ic\times\N\times \R^{m+1}$ such that $x\notin K$).

We define the operator $\Delta$ by
\beqs
\Delta g(s,\mu) & = & \sum_{i\in \Ic,n\in\N}\frac{1}{2^{2(|i|+n)}}\Delta_b g_{i,n}(b^i, q^{i,n})
\enqs
for $g\in C_c^{1,2}([0,T]\times \Ic\times\N\times \R^{m+1})$ and $\mu\in \sum_{i\in \Ic,n\in\N}\frac{1}{2^{2(|i|+n)}}\delta_{(i,n,b^i, q^{i,n})}\in \tilde E _{m+1}$, where $\Delta_b$ stands for the Laplacian operator with respect to the third variable of the function  $g$.
For a given control $\alpha\in\Ac$, let $\bar L ^\alpha$ be the following second order local operator
  \beqs
 \bar L ^{\alpha} F_{f,g} (s,\bar x) & = &
 \partial_1 F (f(s,{}^1\bar x_s),g(s,{}^2\bar x_s))\sum_{i\in\Vc_s}\big(\partial_t+\hat L^{\alpha^i(s,{}^2\bar x_s)}\big)f_i(s,\hat x^i_s)\\
  & & +\partial_1 F (f(s,{}^1\bar x_s),g(s,{}^2\bar x_s))\Big(\partial_t+\frac{1}{2}\Delta \Big)g(s,{}^2\bar x_s)\\
   & & + \frac{1}{2} \partial_{11} F (f(s,{}^1\bar x_s),g(s,{}^2\bar x_s))\sum_{i\in\Vc_s}\big|\hat \sigma(\hat x_s^i,\alpha(s,{}^2\bar x)\big)^\top Df_i(s,\hat x^i_s)\big|^2\\
   & & + \partial_{12} F (f(s,{}^1\bar x_s),g(s,{}^2\bar x_s))\\
    & & \sum_{i\in\Vc_s,n\in\N}\frac{1}{2^{2(|i|+n)}}\partial_b g_{i,n}(s,b^i_s,q^{i,n}_s)^\top\hat \sigma(\hat x_s^i,\alpha(s,{}^2\bar x)\big)^\top Df_i(s,\hat x^i_s)\\
   & & + \frac{1}{2} \partial_{22} F (f(s,{}^1\bar x_s),g(s,{}^2\bar x_s))\sum_{i\in\Ic,n\in\N}\left|\frac{1}{2^{2(|i|+n)}}\partial_b g_{i,n}(s,b^i_s,q^{i,n}_s)\right|^2\\
    &  & + \gamma\sum_{i\in\Ic,n\in\N} p_n\Big\{ F\Big(f(s,{}^1\bar x_s)+\mathds{1}_{\Vc_s}(i)\sum_{\ell=0}^{n-1} (f_{i\ell}-f_i)(s,\hat x^i_s),\\
     & & \qquad\qquad g(s,{}^2\bar x_s)+\frac{1}{2^{2(|i|+n)}}\big(g(i,n,s,b^i_s,q^{i,n}_s+1)-g(i,n,s,b^i_s,q^{i,n}_s)\Big)\\
      & & \qquad\qquad-F(f(s,{}^1\bar x_s),g(s,{}^2\bar x_s))\Big\}
 \enqs
for $s\in[0,T]$,  $\bar x =\left(\sum_{i\in \Vc_s}\delta_{(i,\hat x_u^i)},\sum_{i\in \Ic,n\in\N}\frac{1}{2^{2(|i|+n)}}\delta_{(i,n,b_u^i, q_u^{i,n})}\right)_{u\in[0,T]}\in \mathbf{X}$, $f\in C_c^{1,2}([0,T]\times\Ic\times \R^{d+1})$, $g\in C^{1,2}_c([0,T]\times\Ic\times\N\times \R^{m+1})$ and $F\in C_c^2(\R^2)$ with $F_{f,g}=F\circ (f,g)$. We then define the process $\bar M^{t,\alpha,F_f}$ by
\beqs
\bar M^{t,\alpha,F_{f,g}}_s & = & F_f(s,\mathbf x_s)-\int_t^s\bar L^{\alpha} F_{f,g}(u,\mathbf x)d u\;,\quad s\in[t,T]\;.
\enqs

\begin{Definition}[Martingale problem]\label{defMP}
Consider the initial condition $(t,\bar x)\in [0,T]\times\mathbf X$, and a control $\alpha\in \Ac$. A probability measure $\bar \P^{t,{}^1\bar x,\alpha}$ is a \emph{solution to the controlled martingale problem} if the process $\bar M^{t,\alpha,F_{f,g}}$ is a $\G$-martingale under $\bar \P^{t,{}^1\bar x,\alpha}$ for any $f\in C^{1,2}_c([0,T]\times \Ic\times \R^{d+1})$, $g\in C^{1,2}_c([0,T]\times \Ic\times\N\times \R^{m+1})$, and any $F\in C_c^2(\R^2)$, $\bar \P^{t,{}^1\bar x,\alpha}({}^1\mathbf x_s={}^1 \bar x_s\mbox{ for }s\in[0,t])=1$ and  $\bar \P^{t,{}^1\bar x,\alpha}({}^2\mathbf x\in {}^2G )=\mathbb{W}({}^2G)$  for any ${}^2G\in {}^2\Gc_t$ where $\mathbb{W}$ stands for the law of the process $\xi$ and ${}^2\G=({}^2\Gc_t)_{t\in[0,T]}$ stands for the canonical filtration on $\D([0,T],\tilde E_{m+1}$).
\end{Definition}

\begin{Definition}[Shifted martingale problem]\label{Shifted martingale problem}
Consider the initial condition $(t,\bar x)\in [0,T]\times\mathbf X$, and a control $\alpha\in \Ac$. A probability measure $\bar \P^{t,\bar x,\alpha}$ is a \emph{solution to the shifted controlled martingale problem} if the process $\bar M^{t,\alpha,F_{f,g}}$ is a $\G$-martingale under $\bar \P^{t,\bar x,\alpha}$ for any $f\in C^{1,2}_c([0,T]\times \Ic\times \R^{d+1})$, $g\in C^{1,2}_c([0,T]\times \Ic\times \R^{m+1})$, and any $F\in C_c^2(\R)$, and $\bar \P^{t,\bar x,\alpha}(\mathbf x_s=\bar x_s \mbox{ for }s\in[0,t])=1$.
\end{Definition}

We are now able to state the main result of this section. For that we need the following notations. We first extend the definition of the concatenation operator $\oplus$ on $\D([0,T],\tilde E_{m+1})$ as follows:
\beqs
(y\oplus_t \tilde y)_s & = & \sum_{i\in\Ic,n\in\N}\frac{1}{2^{2(|i|+n)}}\delta_{(i,n,(b ^i\oplus_t \tilde b^i)_s,(q ^{i,n}\oplus_t \tilde q^{i,n})_s)}
\enqs
with
\beqs
(b ^i\oplus_t \tilde b^i)_s & = & b ^i_{s\wedge t}\mathds{1}_{s<t}+ (\tilde b^i_s-\tilde b^i_t+b^i_t)\mathds{1}_{s\geq t}\\
(q^{i,n}\oplus_t \tilde q^{i,n})_s & = & q^{i,n}_{s\wedge t}\mathds{1}_{s<t}+ (\tilde q^{i,n}_s-\tilde q^{i,n}_t+q^{i,n}_t)\mathds{1}_{s\geq t}
\enqs
for $s\in[0,T]$, $y=(\sum_{i\in\Ic,n\in\N}\frac{1}{2^{2(|i|+n)}}\delta_{(i,n,b ^i_u,q ^{i,n}_u)})_{u\in[0,T]}$  and $\tilde y=(\sum_{i\in\Ic,n\in\N}\frac{1}{2^{2(|i|+n)}}\delta_{(i,n,\tilde b ^i_u,\tilde q ^{i,n}_u)})_{u\in[0,T]}$.
In particular, we have
\beqs
\xi(\omega\oplus_t\tilde \omega) & = & \xi(\omega)\oplus_t\xi(\tilde \omega)
\enqs
for $\omega,\tilde \omega\in\Omega$ and $t\in[0,T]$.
For $\eta:[0,T]\times {}^2\mathbf{X}\rightarrow \R$ and $\bar x\in \mathbf{X}$ we finally define the function $\eta^{t,{}^2\bar x}$ by
\beqs
\eta^{t,{}^2\bar x}( s,{}^2\bar x')  & = & \eta(s,{}^2\bar x\oplus_t{}^2\bar x')
\enqs
for $\bar x'\in\mathbf{X}$ and $s\in[0,T]$.
\begin{Theorem}\label{IdCondProbCanSpace}
Suppose that Assumption \ref{Assumption_H_0} holds and that there exists a unique solution to the martingale problem and the shifted martingale problem for each initial condition and control. Let $(t,\bar x,\alpha)\in [0,T]\times\mathbf X\times \Ac$ and $\tau$ a $\G$-stopping time valued in $[t,T]$. Then, we have
\beqs
\bar \P_{\bar x'}^{t,{}^1\bar x,\alpha} & = & \P^{\tau{(\bar x')},\bar x',\alpha^{\tau(\bar x'),{}^2\bar x'}}\;,\quad  \bar \P^{t,\bar x,\alpha}(d \bar x')-a.s.
\enqs
where $(\bar \P_{\bar x'}^{t,\bar x,\alpha},~\bar x'\in \mathbf{X})$ is a \emph{regular conditional probability distribution} of $\bar \P_{}^{t,\bar x,\alpha}$ given $\Gc_\tau$.
\end{Theorem}
\begin{proof} We first define the function $\hat \lambda^\alpha$ and $\hat \sigma^\alpha$ by
\beqs
\hat \lambda^\alpha(s,\bar x)  =  \hat \lambda({}^1\bar x_s,\alpha(s,{}^2\bar x)) & \mbox{ and } & \hat \sigma^\alpha(s,\bar x)  =  \hat \sigma({}^1\bar x_s,\alpha(s,{}^2\bar x))\;,\quad (s,\bar x)\in[0,T]\times \mathbf{X}\;,
\enqs
for $\alpha\in\Ac$.
Since $C^2_c(\R^2)\times C_c^{1,2}([0,T]\times \Ic\times \R^{d+1})\times C_c^{1,2}([0,T]\times \Ic\times\N\times \R^{m+1}) $ admits a dense countable subset, we can apply Theorem 6.1.3 of \cite{SV97} to our framework and we get a negligible set $N\in \Gc_\tau$ such that  for any $(F,f,g)\in C^2_c(\R^2)\times C_c^{1,2}([0,T]\times \Ic\times \R^{d+1})\times C_c^{1,2}([0,T]\times \Ic\times\N\times \R^{m+1}) $, the process $(\bar M^{t,\alpha,F_{f,g}}_s)_{s\in[\tau(\bar x'),T]}$ is a $\G$-martingale under $\bar \P_{\bar x'}^{t,{}^1\bar x,\alpha}$ for any $\bar x'\in \mathbf{X}\setminus N$. We notice that
\beqs
\bar \P_{\bar x'}^{t,{}^1\bar x,\alpha}\left(\left\{\bar x''\in\mathbf{X}~:~\hat \lambda^{\alpha^{\tau(\bar x'),{}^2\bar x'}}(s,\bar x'')=\hat \lambda^\alpha(s,\bar x'')\;\mbox{ for all }s\in  [\tau(\bar x'),T] \right\}\right) & = & 1\enqs
and
\beqs
\bar \P_{\bar x'}^{t,{}^1\bar x,\alpha}\left(\left\{\bar x''\in\mathbf{X}~:~\hat \sigma^{\alpha^{\tau(\bar x'),{}^2\bar x'}}(s,\bar x'')=\hat \sigma^\alpha(s,\bar x'')\;\mbox{ for all }s\in  [\tau(\bar x'),T] \right\}\right) & = & 1
\enqs
for any $\bar x'\in \mathbf{X}\setminus N$. Therefore, for any $(F,f,g)\in C^2_c(\R^2)\times C_c^{1,2}([0,T]\times \Ic\times \R^{d+1})\times C_c^{1,2}([0,T]\times \Ic\times\N\times \R^{m+1}) $, the process $(\bar M^{\tau{(\bar x')},\alpha^{\tau(\bar x'),{}^2\bar x'},F_{f,g}}_s)_{s\in[\tau(\bar x'),T]}$ is a $\G$-martingale under $\bar \P_{\bar x'}^{t,{}^1\bar x,\alpha}$ for any $\bar x'\in \mathbf{X}\setminus N$. By uniqueness to the shifted controlled martingale problem with initial condition $(\tau{(\bar x')},\bar x')$ and control $\alpha^{\tau(\bar x'),{}^2\bar x'}$,  we get
\beqs
\bar \P_{\bar x'}^{t,{}^1\bar x,\alpha} & = & \P^{\tau{(\bar x')},\bar x',\alpha^{\tau(\bar x'),{}^2\bar x'}}
\enqs
for any $\bar x'\in \mathbf{X}\setminus N$.
\end{proof}
\begin{Theorem}\label{uniqueness}
Under Assumption \ref{Assumption_H_0}, the martingale problem and the shifted martingale problem admit unique solutions for any initial condition $(t,\bar x)\in [0,T]\times \mathbf X$ and any control $\alpha\in \Ac$.
\end{Theorem}
To prove Theorem \ref{uniqueness}, we need to consider an extended process $\mathfrak{x}$ defined by
\beqs
{ \mathfrak{x}}_s & = & (s,(\mathbf x_{u\wedge s}))\;,\quad s\in [t,T]
\enqs
The process ${\mathfrak{x}}$ is valued in $\mathfrak{X}=\R\times \mathbf X$ which is separable. We introduce the domain $\Dc$ as the set of function $h:~  \mathfrak{X}\rightarrow \R$ of the form
\beqs
h(s,\bar{{x}}) = H\left( F^1_{f^1,g^1}(s,\bar x_{s\wedge t_1}),\ldots,F^p_{f^p,g^p}(s,\bar x_{s\wedge t_p})\right)\;,\quad (s,\bar x)\in \mathfrak{X}\;,
\enqs
for some $p\geq 1$, $0\leq t_1 <\cdots<t_p\leq T$, $H\in C^{2}(\R^p)$, $F^1,\ldots,F^p\in C^{1,2}_c(\R^{2})$, $f^1,\ldots,f^p\in C^{1,2}_c([0,T]\times \Ic\times \R^{d+1})$ and $g^1,\ldots,g^p\in C^{1,2}_c([0,T]\times \Ic\times\N \times \R^{m+1})$.
We then define on $\Dc$ the  operator $\mathbf{L}^{t,\alpha}$  by
\beqs
\mathbf{L}^{t,\alpha} h(s,\bar x)  & = &
  {\mathfrak{L}}^{t,\alpha}(s,\bar x)\cdot D H(
 h^1(s,\bar x_{s\wedge t_1}),\dots,h^p(s,\bar x_{s\wedge t_p})
)\\
 & & +  \frac{1}{2} \sum_{i\in\Ic}\textrm{Tr}\left(\mathfrak{S}^{\alpha}(\mathfrak{S}^{\alpha})^\top(s,\bar x)D^2H( F^1_{f^1,g^1}(s,\bar x_{s\wedge t_1}),\ldots,F^p_{f^p,g^p}(s,\bar x_{s\wedge t_p})
)\right)\\
 & & + \sum_{j=1}^p\mathds{1}_{t_{j-1}<s\leq t_j}\sum_{i\in \Ic} \sum_{k\geq 0}\gamma p_k\\
  & &  \qquad \bigg(  H\left( F^1_{f^1,g^1}(s,\bar x_{ t_1}),\ldots, F^{j-1}_{f^{j-1},g^{j-1}}(s,\bar x_{ t_{j-1}}),\mathfrak{G}_{i,k} F^j_{f^j,g^j}(s,\bar x_{s}),\ldots,\mathfrak{G}_{i,k} F^p_{f^p,g^p}(s,\bar x_{s})\right)\\
  & & \qquad \qquad \qquad \qquad \qquad \qquad \qquad \qquad -H\left( F^1_{f^1,g^1}(s,\bar x_{s\wedge t_1}),\ldots, F^p_{f^p,g^p}(s,\bar x_{s\wedge t_p})\right)\bigg)
\enqs
with $t_0=0$, where
\begin{equation*}
{\mathfrak{L}}^{t,\alpha}(s,\bar x) ~ = ~ \left(\begin{array}{c}{\mathds{1}_{s\leq t_1}\mathfrak{L}}^{t,\alpha,1}(s,\bar x)\\ \vdots\\
\mathds{1}_{s\leq t_1}{\mathfrak{L}}^{t,\alpha,p}(s,\bar x)
\end{array}\right)
\end{equation*}
with
\beqs
\mathfrak{L}^{t,\alpha,q}(s,\bar x) &= &\bar L ^{\alpha} F^q_{f^q,g^q} (s,\bar x)\\   & & -\gamma\sum_{i\in\Ic,n\in\N} p_n \Big\{F^q\Big(f^q(s,{}^1\bar x_s)+\mathds{1}_{\Vc_s}(i)\sum_{\ell=0}^{n-1} (f^q_{i\ell}-f^q_i)(s,\hat x^i_s),\\
     & & \qquad\qquad g^q(s,{}^2\bar x_s)+\frac{1}{2^{2(|i|+n)}}\big(g^q(i,n,s,b^i_s,q^{i,n}_s+1)-g^q(i,n,s,b^i_s,q^{i,n}_s)\Big)\\
      & & \qquad\qquad -F^q(f^q(s,{}^1\bar x_s),g^q(s,{}^2\bar x_s)) \Big\}
\enqs
and
\begin{equation*}
\mathfrak{S}^{t,\alpha}(s,\bar x) ~ = ~  \left(\begin{array}{c} \mathfrak{S}^{t,\alpha,1}(s,\bar x) \\ \vdots\\ \mathfrak{S}^{t,\alpha,p}(s,\bar x) \end{array}\right)
\end{equation*}
with
\beqs
\mathfrak{S}^{t,\alpha,q}(s,\bar x) & = & \sum_{i\in \Ic}\mathds{1}_{s\leq t_q}\Big(\partial_2 F^q (f^q(s,{}^1\bar x_s),g^q(s,{}^2\bar x_s))\sum_{n\in\N}\frac{1}{2^{2(|i|+n)}}\partial_b g_{i,n}^q(s,b^i_s,q^{i,n}_s)\\
 & &  +\mathds{1}_{i\in\Vc_s} \partial_1 F^q (f^q(s,{}^1\bar x_s),g^q(s,{}^2\bar x_s)) Df^q_i(s,\hat x^i_s)^\top \hat \sigma(\hat x_s^i,\alpha(s,{}^2\bar x)\big)\Big)
\enqs
for $q=1,\ldots,p$, $ (s,\bar x)\in [0,T]\times \mathbf{X}$ and
\beqs
\mathfrak{G}_{i,n} F^j_{f^j,g^j}(s,\bar x) &  = &
F\Big(f(s,{}^1\bar x_s)+\mathds{1}_{\Vc_s}(i)\sum_{\ell=0}^{n-1} (f_{i\ell}-f_i)(s,\hat x^i_s),\\
     & & \qquad\qquad g(s,{}^2\bar x_s)+\frac{1}{2^{2(|i|+n)}}\big(g(i,n,s,b^i_s,q^{i,n}_s+1)-g(i,n,s,b^i_s,q^{i,n}_s)\Big)
\enqs
for $ (s,\bar x)\in [0,T]\times \mathbf{X}$, $i\in\Ic$ and $n\geq 0$.
We then notice that for $\bar \P^{t,{}^1\bar x,\alpha}$ (resp. $\bar \P^{t,\bar x,\alpha}$) solution to the martingale problem (resp. shifted martingale problem) with initial condition $(t,{}^1\bar x)$ (resp.  $(t,\bar x)$) and control $\alpha$ the process
\beqs
h({\mathfrak{x}}_s)-\int_t^s \mathbf{L}^{t,\alpha} h({\mathfrak{x}}_u)du\;,\quad t\leq u\leq T\;,
\enqs
is a $\G$-martingale under $\bar \P^{t,{}^1\bar x,\alpha}$ (resp. $\bar \P^{t,\bar x,\alpha}$)

\begin{Lemma}\label{lem1marg}
Let $(t,\bar x,\alpha)\in [0,T]\times\mathbf{X}\times \Ac$ and $\bar \P_1^{t,{}^1\bar x,\alpha}$ and $\bar \P_{2}^{t,{}^1\bar x,\alpha}$ (resp. $\bar \P_1^{t,\bar x,\alpha}$ and $\bar \P_{2}^{t,\bar x,\alpha}$) two solutions to the martingale problem  (resp. shifted martingale problem)  with initial condition $(t,{}^1\bar x)$ (resp.  $(t,\bar x)$) and control $\alpha$. Then, $\bar \P_1^{t,\bar x,\alpha}$ and $\bar \P_{2}^{t,\bar x,\alpha}$ have the same one dimensional marginals:
\beq\label{eq1dim-marg}
\bar \P_{1}^{t,{}^1\bar x,\alpha}(\mathfrak{x}_s\in B) & = & \bar \P_{2}^{t,{}^1\bar x,\alpha}(\mathfrak{x}_s\in B)\\
(resp. ~~\bar \P_{1}^{t,\bar x,\alpha}(\mathfrak{x}_s\in B) & = & \bar \P_{2}^{t,\bar x,\alpha}(\mathfrak{x}_s\in B) )\label{eq1dim-marg-shift}
\enq
for $s\in[t,T]$ and $B\in \Bc(\mathfrak{X})$.
\end{Lemma}
\begin{proof} We first endow the measurable space $(\mathbf{X}\times  \mathbf{X}, \Gc_T\otimes \Gc_T)$ with the probability measure  $\bar \P=\bar \P_{1}^{t,{}^1\bar x,\alpha}\otimes\bar \P_{2}^{t,{}^1\bar x,\alpha}$ (resp. $\bar \P=\bar \P_{1}^{t,\bar x,\alpha}\otimes\bar \P_{2}^{t,\bar x,\alpha}$).  For $h\in \Dc$,
 we have
\beqs
\E^{\bar \P}\left[ h\otimes h(\mathfrak{x}_s,\mathfrak{x}_t) \right] & = & \E^{\bar \P}\left[ h\otimes h(\mathfrak{x}_t,\mathfrak{x}_s) \right]
\enqs
Indeed, the processes
\beqs
h\otimes h(\mathfrak{x}_s,\mathfrak{x}_t)-\int_t^s \mathbf{L}^{t,\alpha} h(\mathfrak{x}_u)h(\mathfrak{x}_t)du\;,\quad t\leq s \leq T
\enqs
and
\beqs
h\otimes h(\mathfrak{x}_t,\mathfrak{x}_s)-\int_t^s h(\mathfrak{x}_t)\mathbf{L}^{t,\alpha} h(\mathfrak{x}_u)du\;,\quad t\leq s \leq T
\enqs
are both martingales under $\bar \P$. Since all the considered functions are bounded, we can take the expectation and we get
\beqs
\E^\P\left[h\otimes h(\mathfrak{x}_t,\mathfrak{x}_s)\right] & = & \E^\P\left[h\otimes h(\mathfrak{x}_s,\mathfrak{x}_t)\right]
\enqs
and
\beqs
\E^{\bar \P_{1}^{t,{}^1\bar x,\alpha}} \left[h(\mathfrak{x}_s)\right] & = & \E^{\bar \P_{2}^{t,{}^1\bar x,\alpha}} \left[h(\mathfrak{x}_s)\right]\\
(resp. ~\E^{\bar \P_{1}^{t,{}\bar x,\alpha}} \left[h(\mathfrak{x}_s)\right] & = & \E^{\bar \P_{2}^{t,{}\bar x,\alpha}} \left[h(\mathfrak{x}_s)\right]~)\;.
\enqs
Since any bounded $\Bc(\mathfrak{X})$-measurable function can be approximated almost everywhere for $\bar \P_{1}^{t,{}^1\bar x,\alpha}$ and $\bar \P_{2}^{t,{}^1\bar x,\alpha}$ (resp. $\bar \P_{1}^{t,{}\bar x,\alpha}$ and $\bar \P_{2}^{t,{}\bar x,\alpha}$) by a sequence of $\Dc$ we get \eqref{eq1dim-marg} (resp. \eqref{eq1dim-marg-shift}).
\end{proof}
\begin{proof}[Proof of Theorem \ref{uniqueness}]
The proof is a direct consequence of Theorem 4.2 in \cite{EthierKurtz} and Lemma \ref{lem1marg}.
\end{proof}
\subsection{Proof of Theorem \ref{condTHM}}\label{proof-condtioning}

We keep the notations of Section \ref{sec-mart-cont-pb}. Fix $(t,\hat \mu,\alpha)\in [0,T]\times E_{d+1}\times\Ac$.
From Proposition \ref{well-pose-branch-pop}, the law $\Lc^\P(\hat Z^{t,\hat \mu,\alpha},\xi)$ under $\P$ of $(\hat Z^{t,\hat \mu,\alpha},\xi)$ provides a solution to the controlled martingale problem with initial condition $(t,\bar x)$,  where $\bar x\in \mathbf{X}$ such that ${}^1\bar x_s=\hat \mu$ for $s\in[0,T]$, and control $\alpha$ given by Definition \ref{defMP}. Therefore, we get from Theorem \ref{uniqueness}
 \beq\label{identif-law-1}
 \Lc^\P(\hat Z^{t,\hat \mu,\alpha},\xi) & = & \bar \P ^{t,{}^1\bar x,\alpha}
 \enq
 In the same way, for $\beta\in \Ac$, $\bar \omega\in\Omega$ such that $\xi(\bar\omega)={}^2\bar x$, the law $\Lc^\P(\hat Z^{t,\hat \mu,\beta},\xi(\bar\omega\oplus_t\cdot) )$ under $\P$ of $(\hat Z^{t,\hat \mu,\beta},\xi(\bar\omega\oplus_t\cdot))$ is the unique solution to the shifted controlled martingale problem with initial condition $(t,\bar x)$ and control $\beta$ given by Definition \ref{Shifted martingale problem}. Therefore, we also get from Theorem \ref{uniqueness} that
 \beq\label{identif-law-2}
 \Lc^\P(\hat Z^{t,\hat \mu,\beta}, \xi(\bar\omega\oplus_t\cdot)) & = & \bar \P ^{t,\bar x,\beta}\;.
 \enq
 Fix now  an $\Fc_\tau$-measurable random variable $Y$.
 From Doob's functional representation Theorem (see Lemma 1.13 in \cite{book:KALLENBERG-FMP}) there exists a random time $\tilde \tau:~\D([0,T],\tilde E_{m+1})\rightarrow\R_+$ that is a stopping time with respect to the filtration generated by the canonical process on $\D([0,T],\tilde E_{m+1})$, and a measurable function $g_Y:~\D([0,T],\tilde E_{m+1})\rightarrow\R$ such that
 \beqs
 \tau(\omega) ~ = ~ \tilde \tau \left(\xi(\omega)\right) & \mbox{ and } & Y(\omega)~=~g_Y\left(\xi(\omega_{\tau(\omega)\wedge.}) \right)~=~g_Y\left(\xi(\omega)\right)\;.
 \enqs
We then define $\bar \tau :~\mathbf{X}\rightarrow\R_+$ by  $\bar \tau =\tilde \tau\circ {}^2\mathbf{x}$ where we recall that ${}^1\mathbf{x}$ and ${}^2\mathbf{x}$ are given by \eqref{def-projX}.
We observe that $\bar \tau$ is a $\G$-stopping time and $g_Y\circ{}^2\mathbf{x}$ is $\Gc_{\bar \tau}$-measurable. We therefore have from \eqref{identif-law-1}
 \beqs
 \E\left[ f\left(\hat Z^{t,\hat \mu,\alpha}\right)Y \right] & = &  \E^{\bar \P ^{t,{}^1\bar x,\alpha}}\left[ f\left({}^1\mathbf{x}\right)g_Y({}^2\mathbf{x}) \right]\\
  & = & \int_{\mathbf{X}}\E^{\bar \P ^{t,{}^1\bar x,\alpha}_{\bar x'}}\left[ f\left({}^1\mathbf{x}\right)\right]  g_Y({}^2\mathbf{x}(\bar x')) d{\bar \P ^{t,{}^1\bar x,\alpha}}(\bar x')\;.
 \enqs
 where $(\bar \P_{\bar x'}^{t,\bar x,\alpha},~\bar x'\in \mathbf{X})$ stands for a \emph{regular conditional probability distribution} of $\bar \P_{}^{t,\bar x,\alpha}$ given $\Gc_{\bar\tau}$. Using Theorem \ref{IdCondProbCanSpace} we finally get
 \beqs
 \E\left[ f\left(\hat Z^{t,\hat \mu,\alpha}\right)Y \right] & = &  \int_{\mathbf{X}}\E^{\P^{\bar \tau{(\bar x')},\bar x',\alpha^{\bar\tau(\bar x'),{}^2\bar x'}}}\left[ f\left({}^1\mathbf{x}\right)\right]  g_Y({}^2\mathbf{x}(\bar x')) d{\bar \P ^{t,{}^1\bar x,\alpha}}(\bar x')\\
  & = & \int_{\mathbf{X}}F\left( \bar \tau{(\bar x')},\bar x',\alpha^{\bar \tau(\bar x'),{}^2\bar x'} \right)
   g_Y({}^2\mathbf{x}(\bar x')) d{\bar \P ^{t,{}^1\bar x,\alpha}}(\bar x')\;.
   \enqs
   Then, using \eqref{identif-law-2}, we get
   \beqs
 \E\left[ f\left(\hat Z^{t,\hat \mu,\alpha}\right)Y \right]    & = & \int_\Omega F\left( \tau(\omega),\hat Z^{t,\hat \mu,\alpha}_{.\wedge \tau}(\omega),\alpha^{\tau(\omega),\omega} \right)Y(\omega)d\P(\omega)\;.
 \enqs
 \bibliographystyle{plain}
\bibliography{main}

\end{document}